\documentclass{article}
\usepackage{amsmath}
\numberwithin{equation}{section}
\usepackage{amsthm}
\usepackage{dsfont}
\usepackage{amssymb}
\usepackage{graphicx, xcolor}
\usepackage{enumerate}
\usepackage[linesnumbered,lined,boxed,commentsnumbered]{algorithm2e}
\usepackage{bm}
\usepackage[hidelinks]{hyperref}
\usepackage{mathtools}
\usepackage{mathrsfs}

\usepackage{tikz}
\usepackage{tikz-3dplot}

\usepackage{color}

\theoremstyle{plain}
\newtheorem{theorem}{Theorem}[section]
\newtheorem{proposition}[theorem]{Proposition}
\newtheorem{lemma}[theorem]{Lemma}
\newtheorem{corollary}[theorem]{Corollary}

\theoremstyle{definition}
\newtheorem{definition}[theorem]{Definition}
\newtheorem{assumption}[theorem]{Assumption}
\newtheorem{convention}[theorem]{Convention}

\theoremstyle{remark}
\newtheorem{remark}[theorem]{Remark}

\newcommand{\RR}{\mathbb{R}}   
\newcommand{\CC}{\mathbb{C}}   
\newcommand{\HH}{\mathbb{H}}   
\newcommand{\ZZ}{\mathbb{Z}}   
\newcommand{\NN}{\mathbb{N}}   

\newcommand{\oa}{\mathfrak{a}} 
\newcommand{\ob}{\mathfrak{b}} 
\newcommand{\oc}{\mathfrak{c}} 
\newcommand{\od}{\mathfrak{d}} 
\newcommand{\err}{\mathcal{E}} 

\newcommand{\sgconst}{K_2}	
\newcommand{\pconst}{\rho}	
\newcommand{\nconst}{\nu}	
\newcommand{\uconst}{\Upsilon}	

\newcommand{\creg}{c^\lambda}	
\newcommand{\cshort}{c^\short}	
\newcommand{\clatt}{c^\latt}			

\newcommand{\pdev}{d_\pP}			
\newcommand{\dist}{d\hspace*{-0.08em}\bar{}\hspace*{0.1em}}	

\newcommand{\Ln}{\Lambda^n}				
\newcommand{\Lyl}{\Lambda_{\yy,\ell}}	
\newcommand{\LylP}{\Lyl(\pP)}			
\newcommand{\LylPx}{\Lyl(\px)}			
\newcommand{\LylPsx}{\Lyl(\pP_{\sigma\actn\xx})}
\newcommand{\Lij}{\Lambda_{ij}(\xx)}	
\newcommand{\On}{\Omega_{n/2}}			

\newcommand{\edge}{\mathfrak{E}}		
\newcommand{\edgem}{\edge^{(1)}_{\yy,\ell}(\pP)}	
\newcommand{\edgee}{\edge^{(2)}_{\yy,\ell}(\pP)}	
\newcommand{\edgel}{\edge_{\yy,\ell}(\pP)}			

\newcommand{\xx}{{\bm x}}		
\newcommand{\xxt}{{\bm{\tilde{x}}}}		
\newcommand{\yy}{{\bm y}} 		
\newcommand{\zz}{{\bm z}} 		
\newcommand{\ww}{{\bm w}}		
\newcommand{\bet}{{\bm \eta}}	
\newcommand{\forget}{\mathfrak{F}}	
\newcommand{\net}{\mathcal{N_\ell}}		
\newcommand{\supp}{\operatorname{supp}}

\newcommand{\family}{\bm{\mathcal{F}}}	
\newcommand{\length}[1]{m_{#1}}		
\newcommand{\len}{{\length{\xx\zz}}}

\newcommand{\pP}{{\mathcal{P}}} 	
\newcommand{\pQ}{{\mathcal{Q}}}		
\newcommand{\px}{\pP_\xx}			
\newcommand{\pz}{\pP_\zz}			
\newcommand{\pw}{\pP_\ww}			
\newcommand{\pww}{\pw\vee\pP_{\ww'}}	
\newcommand{\pl}{{\pP_{\yy,\ell}}}	
\newcommand{\PE}[1]{\bm{E}^{#1}}	
\newcommand{\ep}{\PE{\pP}}			
\newcommand{\epx}{\PE{\px}}		

\newcommand{\dat}{\bm{H}}	
\newcommand{\grn}{\bm{G}}	
\newcommand{\Vecs}{\bm{V}}	
\newcommand{\EV}{\bm{U}}	
\newcommand{\orth}{\bm{O}}	
\newcommand{\adj}{\bm{A}}	
\newcommand{\GOE}{\bm{Z}}	
\newcommand{\Bmat}{\bm{B}}	
\newcommand{\iden}{\bm{I}}	
\newcommand{\lperm}{\bm{A}}	
\newcommand{\mat}{\bm{M}}	
\newcommand{\cov}{\bm{R}}	

\newcommand{\vu}{\bm{\vec{u}}}	
\newcommand{\ve}{\bm{\vec{e}}}	
\newcommand{\vg}{\bm{\vec{g}}}	
\newcommand{\vh}{\bm{\vec{h}}}	
\newcommand{\vv}{\bm{\vec{v}}}	
\newcommand{\vw}{\bm{\vec{w}}}	
\newcommand{\vq}{\bm{\vec{q}}}	
\newcommand{\vorth}{\bm{\vec{O}}}	
\newcommand{\zero}{\bm{\vec{0}}}	

\newcommand{\discreteCauchy}{\mathcal{C}}	

\newcommand{\ip}[1]{\left\langle #1 \right\rangle}	
\newcommand{\ipr}[1]{\ip{#1}_\RR}					
\newcommand{\ipn}[1]{\ip{#1}_{\Ln}}			
\newcommand{\ipyl}[1]{\ip{#1}_{\Lambda_{y,\ell}}}	
\newcommand{\ipij}[1]{\ip{#1}_{\Lambda_{ij}(\xx)}}	

\newcommand{\eqpart}[1]{\overset{#1}{\sim}}	
\newcommand{\eqp}{\eqpart{\pP}}				
\newcommand{\eqpx}{\eqpart{\px}}		
\newcommand{\eqpsx}{\eqpart{\pP_{s_{ab}^{ij}\xx}}} 
\newcommand{\eqpww}{\eqpart{\pww}}			
\newcommand{\eqpl}{\eqpart{\pl}}			
\newcommand{\eqloc}[1]{\underset{#1}{\sim}}	
\newcommand{\eql}{\eqloc{\yy,\ell}}			
\newcommand{\eqij}{\approx_{ij}}			

\newcommand{\EE}[2][]{\mathbb{E}_{#1} \left[ #2 \right]}	
\newcommand{\ee}{\mathbb{E}}									
\newcommand{\prob}[1]{\mathbb{P} \left[ #1 \right]}	
\newcommand{\normal}{\mathcal{N}}						
\newcommand{\one}[1]{\mathds{1}_{#1}}					
\newcommand{\sign}{\mathrm{sign}}							
\newcommand{\fc}{\mathrm{fc}}								

\newcommand{\levlim}{\mathcal{U}}
\newcommand{\countable}{\mathcal{A}}

\newcommand{\op}{\mathcal{A}}		
\newcommand{\dom}{\mathcal{D}}		
\newcommand{\enint}{\mathcal{I}}	
\newcommand{\indint}{\mathcal{J}}	
\newcommand{\loc}{\one{\operatorname{loc}}}	

\newcommand{\oldgen}{\mathscr{B}}	
\newcommand{\gen}{\mathscr{L}}			
\newcommand{\move}{\mathscr{M}}		
\newcommand{\exch}{\mathscr{E}}		
\newcommand{\semi}{\mathscr{U}}		
\newcommand{\short}{\mathscr{S}}		
\newcommand{\latt}{\mathscr{W}}			

\newcommand{\kproj}{\bm{\Pi}}		
\newcommand{\proj}{\bm{\tilde \Pi}} 
\newcommand{\pyl}{\proj_{\yy,\ell}}	
\newcommand{\dir}{\mathcal{D}}		
\newcommand{\dyl}{\dir_{\yy,\ell}}	

\newcommand{\Av}{\mathrm{Av}}

\newcommand{\Stab}{\mathrm{Stab}}	
\newcommand{\actn}{\cdot}			
\newcommand{\actN}{\star}			
\newcommand{\sym}{\mathcal{G}}		
\newcommand{\symP}{\sym_\pP}		
\newcommand{\symPx}{\sym_{\px}}		
\newcommand{\symPl}{\sym_{\pl}}		
\newcommand{\tr}{\mathrm{Tr}}		
\newcommand{\bl}{\bm{\lambda}}		
\newcommand{\tbl}{\bm{\tilde \lambda}}	
\newcommand{\vspan}{\mathrm{span}}	
\newcommand{\rank}{\mathrm{rank}}
\newcommand{\fh}{\mathfrak{h}}		
\newcommand{\soB}{\mathfrak{X}}			
\newcommand{\texp}{\tilde{\exp}}	

\renewcommand{\Im}{\operatorname{Im}}

\title{High dimensional normality of noisy eigenvectors}
\author{
Jake Marcinek \\
Harvard University \\ 
\texttt{marcinek@math.harvard.edu}
\and 
Horng-Tzer Yau \\
Harvard University \\
\texttt{htyau@math.harvard.edu}
\thanks{H.-T. Y. is partially supported by NSF grants DMS-1606305 and DMS-1855509, and a
Simons Investigator award.}
}

\begin{document}
\maketitle

\begin{abstract}

We study joint eigenvector distributions for large symmetric matrices in the presence of weak noise.  Our main result asserts that every submatrix in the orthogonal matrix of eigenvectors converges to a multidimensional Gaussian distribution.  The proof  involves analyzing the \emph{stochastic eigenstate equation (SEE)} \cite{QUE} which describes the Lie group valued flow of eigenvectors induced by matrix valued Brownian motion.
We consider the associated \emph{colored eigenvector moment flow} defining an SDE on a particle configuration space.  This flow extends  the \emph{eigenvector moment flow} first introduced in \cite{QUE} to the multicolor setting. However, it is no longer driven by an underlying Markov process on configuration space due to the lack of positivity in the semigroup kernel.  Nevertheless, we prove the dynamics admit sufficient averaged decay and contractive properties.  This allows us to establish optimal time of relaxation to equilibrium for the colored eigenvector moment flow and prove joint asymptotic normality for eigenvectors.  Applications in random matrix theory include the explicit computations of joint eigenvector distributions for general Wigner type matrices and sparse graph models when corresponding eigenvalues lie in the bulk of the spectrum, as well as joint eigenvector distributions for L\'evy matrices when the eigenvectors correspond to small energy levels.
\end{abstract}

\tableofcontents

\section{Introduction}

In this paper, we study the joint distribution of
many eigenvectors simultaneously.  These eigenvectors belong to the matrix model $\adj+\Bmat$ where $\adj$ is deterministic symmetric matrix and $\Bmat$ is a small Gaussian perturbation in the form of a scaled Gaussian orthogonal ensemble (GOE).
Some assumptions must be imposed on $\adj$ regarding both its eigenvalues and eigenvectors.  The eigenvalue assumption essentially amounts to the existence of a local profile for the spectrum of $\adj$ and is stated in terms of diagonal Green's function entries.  The eigenvector assumption essentially amounts to a weak form of restricted delocalization in a selected set of directions and is stated in terms of the off-diagonal Green's function entries.  These assumptions are typically satisfied if $\adj$ is sampled from most random matrix ensembles.
The class of matrix models we deal with is more general, however, because the proof of the main theorem captures the regularizing effect of the random perturbation.
Classical comparison arguments can then be used to remove the $\Bmat$ contribution and recover the joint eigenvector distributions for a wide variety of random systems exhibiting a delocalized phase.  For example, the results hold for general type Wigner ensembles, sparse random graph models, and heavy-tailed L\'evy matrices.

Universality of eigenvalue statistics is a classical field stemming from the work on heavy atoms of Wigner \cite{W55} \cite{W58}, Dyson \cite{Dyson}, Gaudin, and Mehta \cite{mehta}.  After recent breakthroughs due to the method of analyzing Green's functions and Dyson Brownian motion, universality has reemerged as an active field of research.  
For examples of universality across a variety of eigenvalue statistics in Wigner matrices, see
 \cite{erdos2009bulk}, \cite{erdHos2010wegner}, \cite{localrelax1}, \cite{localrelax2}, \cite{TV09a}, \cite{TV10}, \cite{TV10a}, and \cite{TV11}.
This allowed the scope of universality and understanding of local statistics to improve drastically.  Relevant extensions of methodologies and general classes of random matrices include the following papers: 
\cite{wigfixed}, \cite{bourgade2016universality}, \cite{che2017local}, \cite{che2019universality}, \cite{scgeneral}, \cite{erdos2017dynamical}, \cite{huang2015spectral}, \cite{huang2015bulk}, \cite{fixed}, \cite{landon2017convergence}, \cite{rigidity}.
Following in the line of universality of eigenvalue statistics, 
universality of eigenvectors in random matrices has also been a focus of study 
--- both for its independent interest and as a tool for accessing finer eigenvalue data.  See the following for some examples: 
\cite{benaych2014central}, \cite{benigni2017eigenvectors}, \cite{benigni2019fermionic}, \cite{bordenave2013localization}, \cite{bourgade2018distribution}, \cite{bourgade2017eigenvector}, \cite{dumitriu2018sparse}, \cite{he2018local}, \cite{knowles2013eigenvector}, \cite{o2016eigenvectors}, \cite{tao2012random}.  
One cornerstone work that will be of primary relevance in this paper is \cite{QUE} which proves single direction normality for generalized Wigner ensembles. 

In particular, the $L^2$-normalized eigenvectors $\vu_i = (u_i^\alpha)_{\alpha=1}^N$, $i \in [N]$, of a GOE are Haar distributed so we expect asymptotic normality of any finite submatrix
\begin{equation}\label{eq:goe_evec}
(\sqrt{N} u_i^\alpha)_{i \in I, \alpha \in A} \rightarrow (g_i^\alpha)_{i \in I, \alpha \in A}
\end{equation}
where $I,A\subset[N]$ are any finite subsets and $g_i^\alpha$ are independent and identically distributed standard Gaussian random variables.  The primary contribution of \cite{QUE} is rowwise and columnwise convergence in moments
\begin{equation}\label{eq:genwig_evec}
(\sqrt{N} u_i^\alpha)_{i \in I} \rightarrow (g_i)_{i=1}^{|I|}
\quad\mbox{and}\quad
(\sqrt{N} u_i^\alpha)_{\alpha \in A} \rightarrow (g_\alpha)_{\alpha=1}^{|A|}
\end{equation}
for generalized Wigner ensembles.  The focus of this paper is in extending \eqref{eq:genwig_evec} to joint normality of the full $|I| \times |A|$ dimensional distribution as in \eqref{eq:goe_evec}. 

The paper \cite{QUE} also introduces the notion of probabilistic quantum unique ergodicity (QUE) in the random matrix theory setting which asserts concentration of 
\begin{equation}
p_{ii} = \sum_{\alpha \in A} N (|u_i^\alpha|^2 - 1)
\end{equation}
for growing subsets $A\subset[N]$ and fixed spectral index $i$.  
One key result of \cite{QUE} is a proof of 
the law of large numbers for this quantity for generalized Wigner ensembles.
Quantum unique ergodicity is a strong expression of the flatness of individual eigenvectors.     This notion was first introduced by Rudnick and Sarnak \cite{rudnick1994behaviour} in the context of Laplacian eigenvectors tending weakly to the volume form on hyperbolic manifolds. QUE  remains largely open in this general context except in the case of some arithmetic surfaces. 

A related quantity
$p_{ij} = \sum_{\alpha \in A} N u_i^\alpha u_j^\alpha$
for distinct spectral indices $i \neq j \in[N]$ can similarly be interpreted as an analogue to quantum weak mixing (QWM).
The central limit theorem scaling for QUE and QWM is proved in \cite{bourgade2018random} and strengthened in \cite{benigni2019fermionic}.   The former uses the perfect matchings flow to control moments of $p_{ii}$ and $p_{ij}$; the latter controls the quantity $\EE{\det\left((p_{ij})_{i,j \in I}\right)}$.
Since these results are tangential to the primary concern of this paper and in particular do not imply \eqref{eq:goe_evec}, we refer the readers to the original papers for further details.

The paper \cite{QUE} introduces the eigenvector moment flow and uses it to first prove fast uniform convergence 
  to standard Gaussians for 
one component of any finite number of eigenvectors,
  or by polarization, 
for any finite number of components from one eigenvector.
However, aside from polarizing to compute basic linear combinations of moments, no conclusions can be deduced regarding joint normality of multiple components of multiple eigenvectors.

The results and techniques in \cite{QUE} have inspired much recent research in the dynamical approach to random matrix theory.  Regarding eigenvectors, sparse models \cite{bourgade2017eigenvector}, diagonal initial conditions \cite{benigni2017eigenvectors}, band matrices \cite{bourgade2018random}, \cite{bourgade2019random}, \cite{yang2018random}, and correlations between eigenvectors \cite{benigni2019fermionic} all use a technique stemming from \cite{QUE} which is that eigenvector moment flow satisfies a maximum principle.  Furthermore, the dynamics of alternate spectral statsistics -- such as decoupling, homogenization, and eigenvalue gaps -- can be compared to eigenvector moment flow and hence the same maximum principle applies in these settings.  These observations are pushed through in \cite{wigfixed} and \cite{fixed} to prove fixed energy universality in the case of Wigner and general DBM initial data respectively, 
and in \cite{bourgade2018extreme} to prove minimum gap universality.  Along with \cite{landon2018comparison} which proves universality for the largest gaps by understanding Gibb's measure on local gap size, universality for extremal gaps is concluded.

With $\dat = (h_{\alpha\beta})_{\alpha,\beta=1}^N$ a self adjoint $N \times N$ matrix, define the time $s$ Gaussian perturbation of $\dat$ by $\dat(s) = \dat+\Bmat(s)$ where $\Bmat(s) = (B_{\alpha\beta}(s))_{\alpha,\beta=1}^N$ is the time dependent matrix whose normalized entries $\sqrt{\frac{N}{1+\delta_{\alpha\beta}}}B_{\alpha\beta}(s)$, $s \geq 0$, are independent and identically distributed standard Brownian motions for all $1 \leq \alpha \leq \beta \leq N$ and $B_{\beta\alpha}(s)=B_{\alpha\beta}(s)$.  Clearly, the SDE of this flow is simply the matrix Brownian motion defined by 
\begin{equation}
d \dat(s) = d \Bmat(s).
\end{equation}
All results will be stated in the setting of the real symmetric GOE universality class (which happens to contain the family of sparse graph models of interest), however as is the case with most dynamical eigenvector papers (e.g. \cite{QUE}, \cite{bourgade2018random}, \cite{benigni2019fermionic}), the same techniques carry over to hermitian and quaternionic self-adjoint (GUE and GSE) universality classes.  The induced stochastic processes on eigenvalues and eigenvectors are referred to as Dyson Brownian motion and the stochastic eigenstate equation, respectively.

Dyson Brownian motion (DBM) which, by orthogonal conjugation invariance of the GOE, decouples from its counterpart --- the stochastic eigenstate equation --- has been the focus of study in many recent works.
The properties of DBM have been studied thoroughly in \cite{landon2017convergence}, \cite{fixed}, \cite{localrelax2} and many other articles.  DBM is the central tool in the dynamical step of many proofs of universality for a wide variety of random matrix statistics (local law, edge, gaps, local statistics, fixed energy, sparse, etc).  
If $\lambda_i(s)$ is the $i$th largest eigenvector of the perturbed matrix at time $s$, $\dat(s)$, then Dyson Brownian motion is given by
\begin{equation}
d\lambda_i(s) = \frac{dW_{ii}(s)}{\sqrt{N}}  + \frac{1}{N} \sum_{j \neq i} \frac{1}{\lambda_i(s) - \lambda_j(s)} ds
\end{equation}
at all times $s > 0$ where $W_{ii}(s)/\sqrt{2}$ are independent standard Brownian motions for all $1 \leq i \leq N$.  Inherently, Dyson Brownian motion describes the Lengevin dynamics of a 1-dimensional log gas in an environment with temperature corresponding to the symmetry class of the underlying matrix, admitting a global mixing rate of 1 and local mixing rate of $N$.  

On the other hand, the stochastic eigenstate equation (SEE) generates a diffusion on the Lie group $SO(N)$.  The diffusion is nondegenerate and invariant with respect to right multiplication and therefore the limiting distribution must coincide with Haar measure.  The diffusion is not heat flow generated by the typical Laplacian, however.  Letting $\orth_{ij} = \ve_i \ve_j^\top - \ve_j \ve_i^\top$, $i < j$ be the standard orthogonal basis of rotation generators in $\mathfrak{so}(N)$, the stochastic eigenstate equation diffusion rate is inversely proportional to the squared distance between corresponding eigenvalues along the $\orth_{ij}$-axis.  If $\vu_i(s)$ is the $L^2$ normalized eigenvector for the matrix $\dat(s)$ with corresponding eigenvalues $\lambda_i(s)$, then the stochastic eigenstate equation is given by
\begin{equation}\label{eq:edbm_intro}
d \vu_i(s) = \frac{1}{\sqrt{N}} \sum_{j \neq i} \frac{dW_{ij}(s)}{\lambda_i(s) - \lambda_j(s)}\vu_j(s) - \frac{1}{2N} \sum_{j \neq i} \frac{ds}{\left(\lambda_i(s) - \lambda_j(s)\right)^2}\vu_i(s)
\end{equation}
at all times $s > 0$ where $W_{ij}$ are independent and identically distributed standard Brownians for all $i < j$ independent from $W_{kk}$ for all $k$, and $W_{ji}=W_{ij}$ for $i>j$.
Heuristically, eigenvectors will spin around one another quadratically faster as the corresponding eigenvalues approach one another.
A direct analysis of the SEE seems to be a difficult problem requiring reconciliation of the matrix entry marginals in a high dimensional geometrically described measure.  
Although the SEE was previously known to Bru \cite{bru1991wishart} in setting of Wishart ensembles, the first rigorous analysis of SEE to understand distributions of eigenvector components was in \cite{QUE}, where  the \emph{eigenvector moment flow} was introduced. We will explain some details of this flow  in subsequent paragraphs.

To motivate the main result on joint normality of eigenvector components, consider the following three observations.
By the local regularity of eigenvalues given by many works on universality of Dyson Brownian motion (for example, see \cite{landon2017convergence}), bulk eigenvalues roughly form a 1-dimensional lattice with spacing $N^{-1}$.
Taking the above diffusion rate into account, one might expect that after time $t$, the $Nt$ nearest eigenvectors will be uniformly distributed on the $Nt$ dimensional sphere and will be independent from the more distant eigenvectors.  As coordinates on a high dimensional sphere are approximately independed Gaussians, one might further expect that these nearby eigenvectors become independent and identically distributed standard Gaussian vectors in $\RR^N$.  All that remains from this heuristic is to establish the covariance structure of these Gaussians.  Fixing unit vectors $\vv, \vw \in \RR^N$ and letting $f_i(s) = \EE{N(\vu_i(s) \cdot \vv)(\vu_i(s) \cdot \vw) | \bl}$, the eigenvector covariance conditioned on an eigenvalue path $\bl = \{\lambda_i(s) | i \in [N], s \geq 0\}$ satisfies the PDE
\begin{equation}\label{eq:intro_cov_pde}
\partial_t f_i(s) = \sum_{j \neq i} \frac{f_j(s) - f_i(s)}{N(\lambda_i(s) - \lambda_j(s))^2}
\end{equation}
where the sum is taken over all spectral indices $j$ aside from $i$.
The PDE \eqref{eq:intro_cov_pde} is obtained through an application of It\^o's formula to the observable $f_i(s)$ using the SEE differential from \eqref{eq:edbm_intro}.  Through rigidity of eigenvalues, the operator generating these dynamics is reminiscent of a discrete analogue to the half Laplacian kernel
\begin{equation}
-\sqrt{-\Delta} \varphi(x) = \int_\RR \frac{\varphi(y) - \varphi(x)}{|y-x|^2}dy
\end{equation}
which generates the heavy-tailed stochastic process with Cauchy increments.  As the Stieltjes transform and the Green's function encode Cauchy convolutions with the empirical spectral density $\frac{1}{N} \sum_{i=1}^N \delta_{\lambda_i(s)}$ and the skewed empirical spectral measure $\sum_{i = 1}^N (\vu_i(s) \cdot \vv)(\vu_i(s) \cdot \vw) \delta_{\lambda_i(s)}$, respectively, one might expect for every spectral index $i$ with corresponding eigenvalue $\lambda_i$ belonging to a regular region of the spectrum,
\begin{align}\label{eq:cov_ansatz}
f_i(s) \approx \EE{f_{\discreteCauchy(i,s)}(0)} &= \left( \frac{1}{N} \sum_{j=1}^N \frac{s f_j(0)}{(\lambda_i(S) - \lambda_j(s))^2 + s^2}\right) \left( \frac{1}{N} \sum_{j=1} \frac{1}{(\lambda_i(s) - \lambda_j(s))^2 + s^2}\right)^{-1} \\
&= \frac{\ipr{ \vv, \Im \grn(\lambda_i(s) + is) \vw }}{\Im m(\lambda_i(s)+is)}
\end{align}
where $\discreteCauchy(i,s)$ is a family of discrete approximations to Cauchy random variables with widths $s$ and centers $\lambda_i(s)$.  In particular, the distributions are given by
\begin{equation}
\prob{\discreteCauchy(i,s) = k} = N^{-1} \left((\lambda_k(s) - \lambda_i(s))^2 + s^2\right)^{-1} m(\lambda_i(s) + is)^{-1}
\end{equation}
for all spectral indices $i,k \in [N]$, and times $s \gg \eta_*$ when we have strong control on both the numerator $\ipr{ \vv, \Im \grn(E+i\eta) \vw }$ and denominator $\Im m(E+i\eta)$ down to imaginary scales $\eta \gg N^{-1}$ slightly larger than the order of the eigenvalue lattice by local laws. Here, $\eta_*$ is the smallest imaginary scale on which the local laws of the original matrix are valid.

To make this heuristic rigorous, we aim to prove convergence of moments along a fixed set of test vectors.  The analysis follows an approach from \cite{QUE} where the dynamics of the following moments are studied
\begin{equation}
\bet \mapsto \EE{\prod_{i=1}^N \ipr{\vu_i, \vq}^{2\eta_i}}
\end{equation}
for a fixed $\vq\in\RR^N$, where the underlying moments are parameterized by $\bet = (\eta_i)_{i=1}^N$ such that $0 \leq \eta_i \in \ZZ$ and $\sum_{i=1}^N \eta_i = n$.  (Inner products on several spaces will be referred to in this paper, so for organizational purposes $\ipr{\vv,\vw} = \vv \cdot \vw$ is used to denote the standard dot product on $\RR^N$).  The flow of such moments induced by the SEE is found to be a stochastic Markov process.
Convergence to Gaussian moments in the generalized Wigner setting is proved by studying this flow which in turn implies convergence in distribution of single eigenvector components from multiple eigenvectors to normality.  In this paper, the \emph{eigenvector moment flow} is generalized to incorporate multiple components in multiple eigenvectors.  To this end, we fix unit test vectors $\vv_1,\ldots,\vv_n\in\RR^N$, and consider moments of the form
\begin{equation}
\xx \mapsto \EE{\prod_{a=1}^n \ipr{\vu_{x_a},\vv_a}}
\end{equation}
where the underlying moments are now parameterized by ordered $n$-tuples of indices $\xx=(x_1,\ldots,x_n)^\top\in\{1,\ldots,N\}^n$.  Such $\xx$ are interpereted as \emph{distinguishable particle configurations} with $n$ particles labeled by the index $a \in \{1,\ldots,n\}$ where particle $a$ is located at site $x_a\in\{1,\ldots,N\}$.  For such mixed moments, the SEE induces a dynamical process on $\xx$ reminiscent of a heavy-tailed random walk and can be described in terms of a particle jumping process.  
We will give a precise evolution equation for this dynamics in  Theorem~\ref{thm:levmf}. For now, we will continue with a heuristic description of this dynamics.  
In all allowable jump operations, two particles are chosen to jump between two sites at a rate inversely proportional to the squared distance between the sites.  For this reason, the moment flow always preserves both the total particle number $n$, as well as the particle number parity at each individual site.  This means the original renormalized configuration space $\{1,\ldots,N\}^n$ decomposes into many closed systems --- each corresponding to a particle number parity assignment to each site.  The largest subsystem corresponds to the particle configurations with even particle numbers at every site.  This is also the primary subsystem of interest as normalized eigenvectors are naturally defined only up to a sign $\vu_i\mapsto\pm\vu_i$, so moments consisting of odd powers of any eigenvector have little inherent meaning.  This leads us to the primary configuration space of interest, $\Ln \subset \{1,\ldots,N\}^n$ where $\xx\in\{1,\ldots,N\}^n$ is in the configuration space $\xx\in\Ln$ if and only if $|\{a\in\{1,\ldots,n\}|x_a=i\}|$ is even for every $i\in\{1,\ldots,N\}$.  Although only dynamics on $\Ln$ is presented here for this reason, a similar analysis carries through for the remaining odd-moment counterparts and analogous decay-to-equilibrium rates can be established for all closed subsystems.  For this reason, the total particle number $n$ is taken to be even.

\begin{figure}[!ht]
\center
\includegraphics{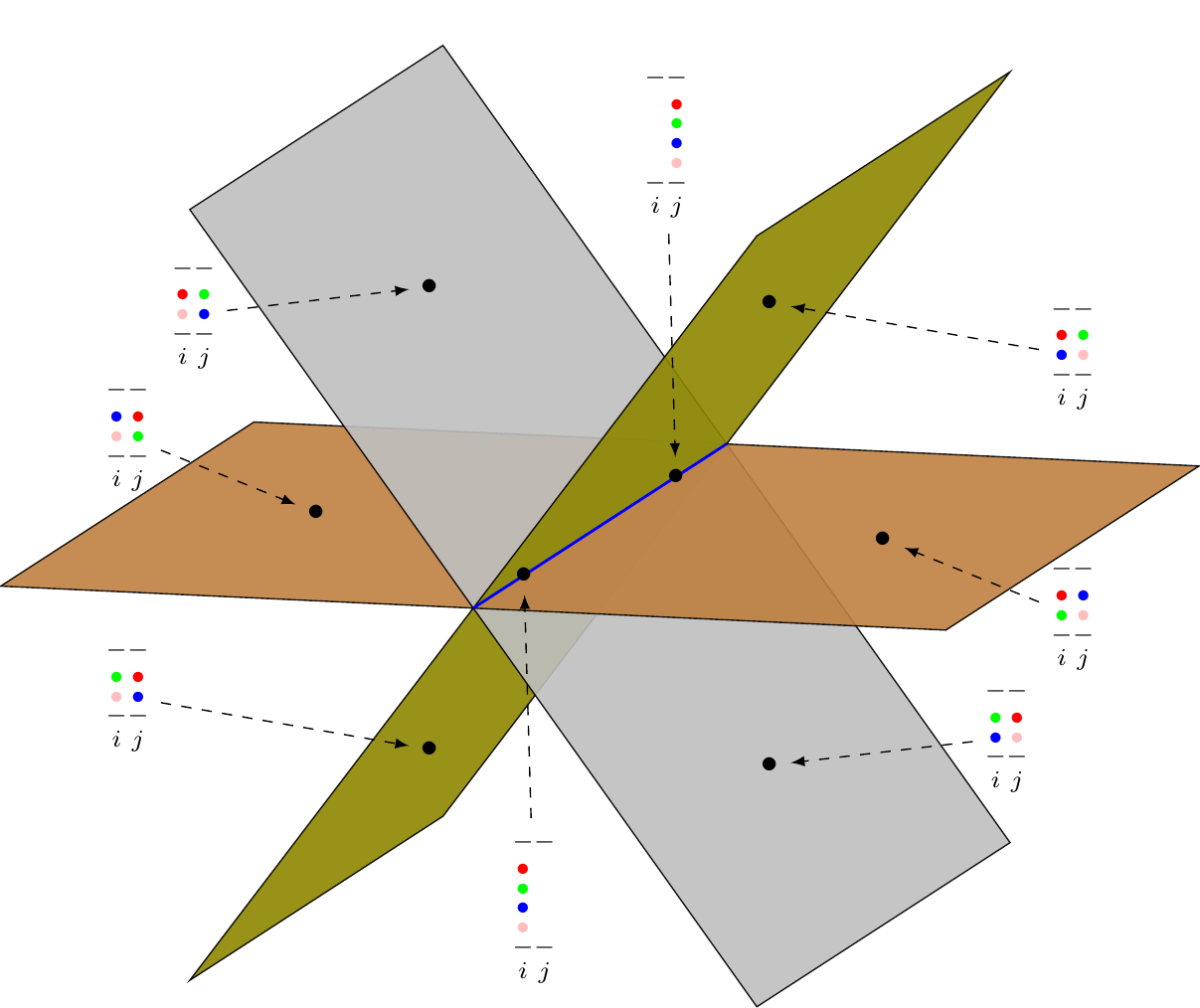}
\caption{\label{fig:config}This diagram shows the intersection of the three strata in $\Lambda^4$, the configuration space of $4$ distinguishable particles.  All particle configurations which are supported on sites $i$ and $j$ are identified and labeled.  The particles are labeled by colors: red, green, blue, and pink.  The three strata (planes in this case) are distinguished by which pairs of colored particles are matched together.  Each stratum is parameterized by local coordinates $(i,j)\in\ZZ^2$ specifying the sites of the two pairs.  For instance, the brown plane consists of all particle configurations where the red and green particles are positioned at the same site and the blue and pink particles are positioned at the same site.  Local coordinates $(i,j)\in\ZZ^2$ in the brown plane correspond to the particle configuration where the red and green particles appear at site $i$ and the blue and pink particles appear at site $j$.}
\end{figure}

The geometry of the even moment configuration space $\Ln$ is quite intricate.  See Figure~\ref{fig:config} to follow along with the description on an examplary plot of $\Lambda^4$.  For general $n$, there are $n!! = 2^{-n/2} n! / (n/2)!$ highest dimensional strata which look locally like a copy of $\ZZ^{n/2}$ and correspond to a unique perfect matching of the $n$ particles.  The positions of the $n/2$ pairs then provide a convenient coordinate chart within such an $(n/2)$-dimensional stratum.  These highest dimensional strata intersect where multiple pairs of particles lie at the same site.

\begin{figure}[!ht]
\center
\includegraphics[scale=1]{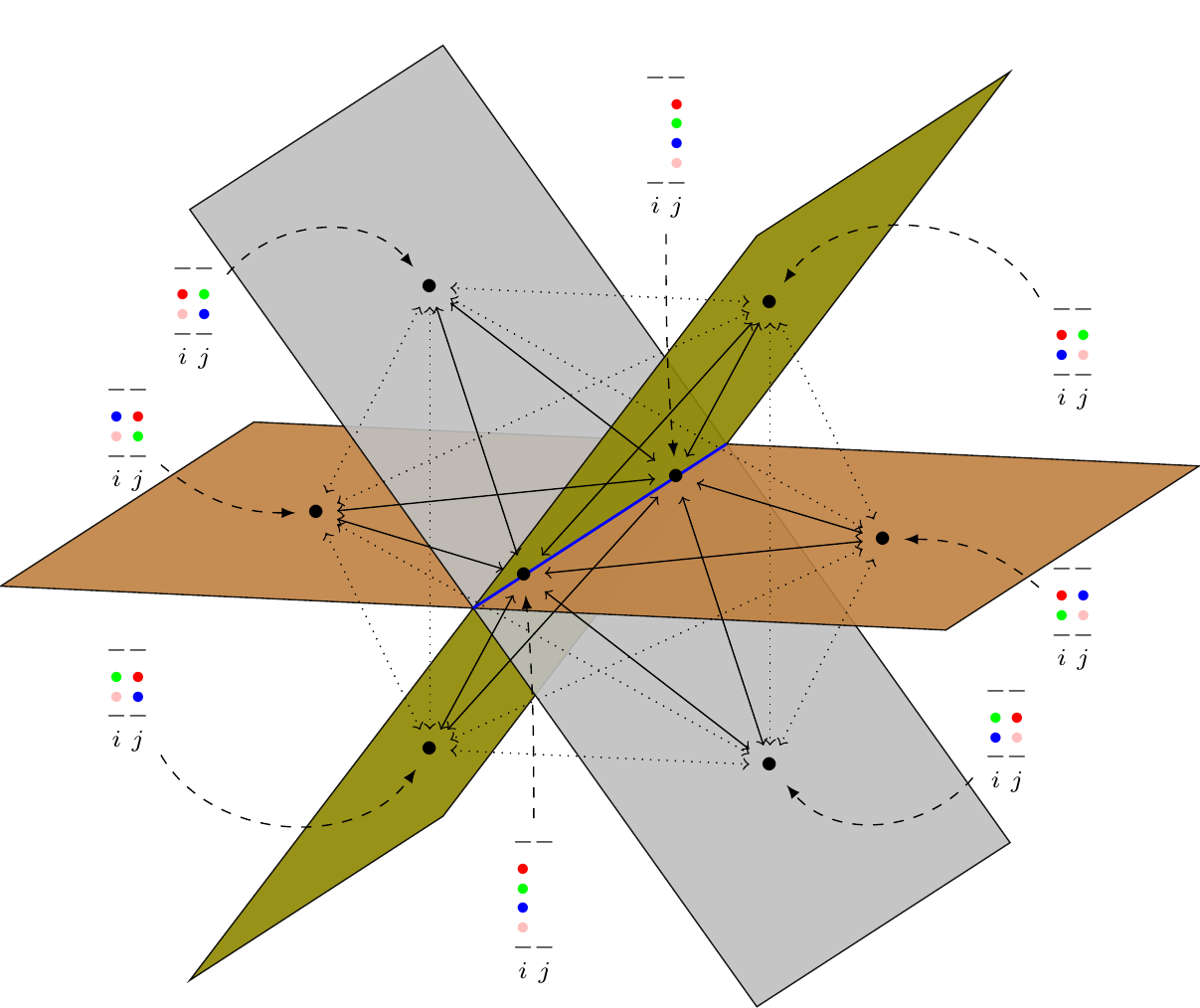}
\caption{\label{fig:dynamics}This diagram shows all interactions between the pairs of particles supported on sites $i$ and $j$.  All attracting move interactions are presented as solid arrows.  These terms allow two particles to jump from site $i$ to site $j$ or from site $j$ to site $i$.  In the case depicted here, all such terms have one endpoint in the interior of a plane and the other endpoint on the three-way intersection.  
All repelling local exchange interactions are presented as dotted lines.  These terms allow one particle at site $i$ to swap positions with one particle at site $j$.  In the case depicted here, all such terms have one endpoint in the interior of one plane and the other endpoint in the interior of another.  None of the configurations in the three-way intersection are involved in a two-particle swap.}
\end{figure}

As the description shifts towards the dynamical aspects, see Figure~\ref{fig:dynamics} to continue following along with the same $\Lambda^4$ example.  Eigenvector moment flow can be expressed as the \emph{difference}  of two positivity preserving operators, $\gen=\move-\exch$, in the sense that if $f \geq 0$ on $\Ln$, then $e^{t\move}f \geq 0$ and $e^{t\exch}f\geq0$ on $\Ln$ as well.  These two positivity preserving operators both generate random walks on $\Ln$, see Theorem~\ref{thm:levmf} for a precise definition.  
In this paper,  we adopt the convention that the positivity properties always refer to the operators  $-\move$ and $-\gen = (-\move) + \exch$.  With this convention,
The operator $-\move$ will be considered the positive contribution to $-\gen$ and $\exch$ the negative contribution.  More specifically, we will see that while $-\move \geq_\infty 0$ and $-\exch \geq_\infty 0$, this strength of positivity does not hold for the difference $-\gen = (-\move) + \exch \not\geq_\infty 0$. 
Here we used the notation that an operator $K$ and $p\in[1,\infty]$ satisfy $K\geq_p0$ if $\partial_t \|f_t\|_{L^p} \geq 0$ whenever $\partial_t f_t = Kf_t$.  However, a weaker notion of positivity does hold for all three operators, $-\gen,-\move,-\exch \geq_2 0$.  We now describe the steps of the two random walks on the level of particle configurations as well as how they interact with one another with respect to the geometry of $\Ln$.

The positive attracting contribution to eigenvector moment flow $-\move$ is the (negative) \emph{move operator}.  A step in the random walk generated by the move operator consists of one pair of particles jumping together from one site to another site.  The negative repelling contribution to eigenvector moment flow $\exch$ is the \emph{exchange operator}.  A step in the random walk generated by the exchange operator consists of two pairs of particles positioned at distinct sites swapping partners.  All steps occur at a rate inversely proportional to the squared distance between the initial and final positions.

On the interior of a highest dimensional stratum, the move operator dominates and eigenvector moment flow behaves like a heavy-tailed random walk in $\RR^{n/2}$ with Cauchy increments in each of the $n/2$ dimensions.  The corresponding kernel is the $n/2$-fold tensor product of Cauchy distributions with local centered density given explicitly by
\begin{equation}
\rho_t(x_1,\ldots,x_{n/2}) \propto \prod_{a=1}^{n/2} \frac{t}{x_a^2 + t^2}.
\end{equation}
Near the intersection of multiple highest dimensional strata, the exchange operator comes into effect. 
The exchange term partially cancels the mixing effect \emph{between hyperplanes} that is otherwise induced by the move operator.

This mixing scheme leads to a rich subspace of invariant functions given by separate constants along each highest dimensional stratum and on the intersection given by the average over all incident hyperplanes.
One primary difficulty of this scheme is that the negative contribution of the repelling exchange operator breaks the positivity preserving properties of $\gen$, ruling out the possibility of the eigenvector moment flow satisfying a maximum principle ($-\gen \not\geq_\infty 0$) in spaces with more than $4$ particles.

The maximum principle is a fundamental property of the colorblind flows discussed above which is used in a critical way to prove relaxation.  In all of the previous approaches to eigenvector flow, the maximum principle cannot be easily replaced.  In our paper, under the pretext of the non-positivity preserving colored eigenvector moment flow, we provide such a replacement --- the \emph{energy method}.
Although eigenvector moment flow is no longer positive in the $L^\infty$ sense, it is still positive definite with respect to a suitable reversible measure.  With the notation above, although $-\gen \not \geq_\infty 0$, we still have $-\gen \geq_2 0$.  From this observation, we may smooth coefficients to show sufficiently fast convergence of a local cutoff for the eigenvector moment observable in $L^2$.  

The role of the energy method is to turn this $L^2$ convergence into pointwise convergence of the colored eigenvector moment flow.  The authors believe this new approach can be used for many related problems involving high dimensional flows of random matrix theory statistics in place of the maximum principle.
The idea is first to establish a Poincar\'e inequality for the colored eigenvector moment flow implying certain mixing properties.  In our case, the mixing properties will be utilized in proving the Nash inequality. Following a line of reasoning similar to \cite{carlen1987upper}, the converse of Nash's original argument will finally imply an ultracontractive estimate ($L^2 \rightarrow L^\infty$) on the semigroup generated by colored eigenvector moment flow.  These steps are spelled out in slightly more detail for a toy model in the following paragraphs.

The Poincar\'e inequality requires rather subtle combinatorics to keep track of exchange terms, but is easier to understand heuristically from the Cauchy random walk analogue in $\ZZ^n$.  Suppose 
\begin{equation}
Lf(\xx) = \sum_{a \in [n]} \sum_{0 \neq k \in \ZZ} \frac{f(\xx + k \ve_a)-f(\xx)}{k^2}
\end{equation}
for every $\xx \in \ZZ^n$.  Let $B=[\ell]^n\subset\ZZ^n$ be a finite box of side length $\ell$.  The Poincar\'e inequality states that 
\begin{equation}
\sum_{\xx\in B} \left(f(\xx) - \frac{1}{\ell^n} \sum_{\yy\in B} f(\yy) \right)^2 \leq \ell \ip{f,(-L)f}_B := \ell\sum_{\xx\in B} \sum_{a\in[n]} \sum_{\substack{1-x_a \leq k \leq \ell-x_a \\ k \neq 0}} \frac{|f(\xx+k\ve_a)-f(\xx)|^2}{k^2}.
\end{equation}
To obtain this bound, use Jensen's inequality and a path counting argument.  Let $B^*$ be the set of edges $(\xx,\yy)\in B^2$ such that $\yy=\xx+k\ve_a$ for some $k\in[\ell]$ and $a\in[n]$.  Note that there is a path of length at most $n$ between any pairs of points in $B$ obtained by fixing one coordinate at a time and that there are at most $\ell^{n-1}$ paths passing through any given edge.  Then
\begin{equation}
\sum_{\xx\in B} \left(f(\xx) - \frac{1}{\ell^n} \sum_{\yy\in B} f(\yy) \right)^2 \leq \ell^{-n} \sum_{\xx\in B} \sum_{\yy\in B} |f(\xx)-f(\yy)|^2 \leq \ell^{-1} \sum_{(\xx,\yy)\in B^*} |f(\xx)-f(\yy)|^2 \leq \ell \ip{f,(-L)f}_B
\end{equation}
where the last inequality comes from the coefficient bound $k^{-2} \geq \ell^{-2}$ for every $k\in[\ell]$.

Taking the Poincar\'e inequality for granted on all translates of the box $B$, a Nash inequality is obtained by dissecting $\ZZ^n$ into boxes of optimal size and applying the Poincar\'e inequality to each component.  To be slightly more detailed, dissect $\ZZ^n$ into many smaller local boxes.  Any function $f$ may be bounded by its deviation to local equilibrium and the weight of its local equilibrium.  After taking suitable norms, this decomposition amounts to bounding $L^2$-norm by the Dirichlet form, $\dir(f) = \ip{ f, (-L) f }$, and $L^1$-norm 
\begin{equation}
\|f\|_2^2 \leq \sum_{\yy \in \ZZ^n} \sum_{\xx \in B+\ell \yy}\left(f(\xx) - \frac{1}{\ell^{n}}\sum_{\zz \in B+\ell\zz} f(\zz)\right)^2 + \frac{1}{\ell^n} \sum_{\yy \in \ZZ^n} \left(\sum_{\zz \in B+\ell\yy} f(\zz)\right)^2 \leq \ell \dir(f) + \frac{1}{\ell^{n}} \|f\|_1^2.
\end{equation}
Optimizing over the length scale $\ell$ and rearranging exponents, this bound is equivalent to the more classical form
\begin{equation}
\|f\|_2^{2+\frac{2}{n}} \leq \dir(f) \|f\|_1^{\frac{2}{n}}
\end{equation}
as the Nash inequality.  Lastly, the ultracontractive estimate is obtained by integrating the Nash inequality.  More precisely, to get ultracontractive control on the semigroup $e^{tL}$, start with any positive normalized $f \in L^1(\ZZ^n)$, $\|f\|_1 = 1$, and let $u_t = \|e^{tL} f\|_2^{-2/n}$.  Then
\begin{equation}\label{eq:intro_integrate_Nash}
\partial_t u_t = \frac{1}{n} \|e^{tL}f\|_2^{-2-\frac{2}{n}} \dir(e^{tL} f) \geq \|e^{tL} f\|_1^{-2/n} = 1
\end{equation}
since $e^{tL}$ is a Markov semigroup and hence preserves $L^1$ norms of positive functions.  Therefore, $u_t \geq t$ meaning $\|e^{tL}f\|_2 \leq t^{-n/2}$.  The argument concludes by duality in that $\|e^{tL}\|_{2,\infty} = \|e^{tL}\|_{1,2} \leq t^{-n/2}$.

Another difficulty in moving from this toy example to eigenvector moment flow is that since the semigroup no longer generates a Markov process, it in turn does not conserve the $L^1$ norm.  The key ingredient needed for the step outlined in \eqref{eq:intro_integrate_Nash} is showing that regardless, the $L^1 \rightarrow L^1$ operator norm of the eigenvector moment flow transition semigroup is bounded.

To summarize, the key contributions are listed here.
First, we introduce the idea of using fast $L^2$-mixing to replace the fast $L^\infty$-mixing from \cite{QUE}.  This is necessary because the maximum principle relies on $L^\infty$-positivity of the generator which no longer holds.  However, the generator does admit $L^2$-positivity with respect to an explicit reversible measure.
Second, we introduce an ultra-contractive estimate which implies fast $L^\infty$-mixing after taking $L^1$ or $L^2$ control as an input.  The three sub-ingredients for this step are 
a Poincar\'e inequality proved through path counting and conditioning on particle symmetries,
a Nash inequality proved through optimal dissection of the configuration space,
and $L^1$ boundedness for the non-positivity-preserving dynamics.
 
Returning to applications of the main theorem, a wide class of mean field models including general type Wigner, sparse graphs, and L\'evy matrices all lie in the realm of applicability for the results of the main theorem.  To exemplify specific simple applications of the main dynamical result, we provide complete proofs of joint eigenvector normality for the following three models.
\begin{definition}[Generalized Wigner]\label{defn:genwig}
A \emph{generalized Wigner ensemble} $\dat = \dat_N = (h_{ij})_{i,j=1}^N$ is a sequence of random self adjoint matrices indexed by their size $N$ whose entries independent random variables up to symmetry.  That is, $h_{ij} = h_{ji}$ are mutually independent random variables for $1 \leq i \leq j \leq N$.  Moreover, these entries have mean zero and variance $\sigma_{ij}^2 = \EE{|h_{ij}|^2}$ satisfying:
\begin{enumerate}
\item Normalization: for any $j \in [[1,N]]$, $\sum_{i=1}^N \sigma_{ij}^2 = 1$.
\item Non-degeneracy: there exists $C>0$ such that $C^{-1} \leq N\sigma_{ij}^2 \leq C$ for all $1 \leq i \leq j \leq N$.
\item Finite moments: for any $p \geq 1$, there exists a constant $C_p > 0$ such that $\EE{|\sqrt{N} h_{ij}|^p} < C_p$ for all $i,j,N$.
\end{enumerate}
\end{definition}
\begin{definition}[Sparse graphs]\label{defn:graphs}
Consider the following two graph models with sparsity $p$:
\begin{enumerate}
\item (Erd\H{o}s--R\'enyi Graph model $G(N,p/N)$) Let $\adj$ be the adjacency matrix of the Erd\H{o}s--R\'enyi graph on $N$ vertices, $(v_i)_{i=1}^N$.  That is, for every $i < j$, $(v_i,v_j)$ is an edge with probability $p/N$ independent of all other edges.  Then $\dat = \adj / \sqrt{p(1-p/N)}$ is the normalized adjacency matrix.
\item ($p$-Regular graph model) Let $\adj$ is the adjacency matrix of a $p$-regular graph on $N$ vertices chosen uniformly at random from the set of all such graphs.  Then $\dat = \adj / \sqrt{p-1}$ is the normalized adjacency matrix.
\end{enumerate}
\end{definition}
\begin{definition}[L\'evy matrices]\label{defn:levy}
Fix a parameter $\alpha \in (0,2)$ and let $\sigma > 0$ be a real number.  A random variable $Z$ is a \emph{$(\sigma, \alpha)$-stable law} if it has the characteristic function
\begin{equation}
\EE{e^{itZ}} = \exp\left( -\sigma^\alpha |t|^\alpha \right), 
\quad \mbox{ for all } t \in \RR.
\end{equation}
Fix $Z$, a $(\sigma, \alpha)$-stable law with
\begin{equation}
\sigma = \left( \frac{\pi}{2\sin\left( \frac{\pi\alpha}{2}\right) \Gamma(\alpha)} \right)^{1/\alpha} > 0.
\end{equation}
Let $J$ be a random variable with finite variance $\EE{J^2} < \infty$ such that $J$ and $Z+J$ are symmetric and 
\begin{equation}
\frac{C^{-1}}{\left(|t|+1\right)^\alpha} \leq \prob{|Z+J| \geq t} \leq \frac{C}{\left( |t| + 1 \right)^\alpha}
\quad \mbox{ for all } t\geq 0 \mbox{ for some constant } C>1.
\end{equation}
Now let $\{H_{ij}\}_{1 \leq i \leq j \leq N}$ be independent and identically distributed random variables with the same law as $N^{-1/\alpha}(Z+J)$.  Set $H_{ij} = H_{ji}$ for $1 \leq j < i \leq N$ and define the random symmetric $N \times N$ matrix $\dat = \left(H_{ij}\right)_{i,j=1}^N$ called an \emph{$\alpha$-L\'evy matrix}.
\end{definition}
The specific eigenvector distributions of Generalized Wigner ensembles was first characterized in \cite{QUE}.
Eigenvector distributions of the two sparse graph models were first characterized in \cite{bourgade2017eigenvector}.
The GOE statistics of L\'efy eigenvalues at small energy is proved in \cite{aggarwal2018goe} and the corresponding eigenvector component distributions were first characteriezd in \cite{aggarwal2020eigenvector}.  The corresponding comparison arguments follow the frameworks provided in these three papers, respectively, where their single component analogues are proved. 

Lastly, the authors believe that, as the maximum principle inspired a unified approach to the variety of problems exemplified above, so may our replacement technique, the energy method, to higher dimensional analogues of related flows.

\subsection{Outline of the paper}

After covering our main results, assumptions, and notation in Section \ref{sec:model}, in Section \ref{sec:fc}, we recall results regarding free convolutions, the isotropic local law, and Dyson Brownian motion pertaining to our model which will be necessary inputs to the proofs in Sections \ref{sec:l2} and \ref{sec:proof}.  In Section \ref{sec:jump}, we introduce a particle jump process related to the flow of joint eigenvector moments under eigenvector Dyson Brownian motion.  We then go on to establish the relevant algebraic and positivity properties of this jump process.  In Section \ref{sec:l2} we develope a framework for isolating local particle dynamics near the regular spectral energy interval and show averaged local convergence.  In Section \ref{sec:energy} we prove ultracontractivity of the hydrodynamic limit of this particle jump process initialized with general data.  Finally, in Section \ref{sec:proof}, we apply these results to our matrix model to prove Theorem~\ref{thm:main}, then provide quick comparison arguments to show these consequences persist in generalized Wigner, sparse graph models, and $\alpha$-L\'evy matrices proving Theorems \ref{thm:genwig}, \ref{thm:p-reg}, and \ref{thm:levy} respectively.

\section{Model and main theorem}\label{sec:model}
In this paper, we consider the following family of deterministic matrix models which many random matrix ensembles belong to with overwhelming probability.  The comparison argument for a select few random matrix ensembles is done in Section \ref{sec:proof} allowing the main deterministic result on the regularizing effect of the SEE to carry over to such models.  As discussed in the introduction, the assumptions are in place to imply a local eigenvalue profile and delocalization.

\subsection{Model}\label{ssec:model}
Let $\dat = \dat_N = (h_{ij})_{i,j=1}^N$ always denote a symmetric $(N \times N)$-matrix.
Universally fix $N$-dependent scales $\frac{1}{N} \leq \eta_* = \eta_*(N) \leq r = r(N)$, an energy level $E_0 \in \RR$, a set of $N$-dimensional unit vectors $S \subset S^{N-1}$, and a large constant $\oa > 0$.  For any $z \in \HH = \{x+iy \in \CC : x \in \RR \mbox{ and } 0 < y \in \RR\}$, let $\grn(z) = (\dat-z)^{-1}$ be the Green's function of $\dat$ and $m_N(z) = \frac{1}{N} \grn(z)$ the Stieltjes transform of the empirical spectral distribution of $\dat$.
\begin{assumption}\label{ass:eval}
Assume the following two properties regarding the eigenvalues of $\dat$:
\begin{enumerate}
\item The matrix norm of $\dat$ is polynomially bounded
\begin{equation}\label{eq:eval1}
\|\dat\| \leq N^\oa
\end{equation}
\item The Stieltjes transform is constant order near the fixed energy
\begin{equation}\label{eq:eval2}
\frac{1}{\oa} \leq | \Im m_N(z) |  \leq \oa
\end{equation}
uniformly on $z \in \{ E+i\eta: E \in [E_0 - r, E_0+r], \eta \in [\eta_*,1]\}$.
\end{enumerate}
\end{assumption}

\begin{assumption}\label{ass:evec}
Assume the following property regarding eigenvectors of $\dat$ for every (small) constant $\ob > 0$:
For all $\vv,\vw \in S$,
\begin{equation}\label{eq:evec1}
|\ipr{ \vv, \Im \grn(z) \vw }| \leq N^\ob
\end{equation}
uniformly on $z \in \{ E+i\eta: E \in [E_0 - r, E_0+r], \eta \in [\eta_*,1]\}$.
\end{assumption}

\subsection{Preliminary notation}
For any positive integer $M > 0$, let $[M] = [1,M] \cap \ZZ$ be a set of size $M$, typically used for indexing. 
The standard column basis vectors in $\RR^M$ will be denoted by $\ve_i = \ve_i^{(M)}$.  The superscript is dropped when the dimension in clear from context.  The components are $\ve_i = (e_i^1, \ldots, e_i^M)^\top$ with $e_i^j = \one{i=j}$.  In general, vector components written in the standard basis will always appear in the superscript.

For $N$-dependent (possibly random) quantities $X$ and $Y$, denote $X \ll Y$ to mean there exists a positive constant (independent of $N$) $c>0$ such that $X \leq N^{-c} Y$ for $N$ large enough.  
We also write $X \precsim Y$ to mean that for all $c > 0$ small and all $D > 0$ large, $\prob{X > N^c Y} \leq N^{-D}$ for $N$ large enough.  More generally, we say that an $N$-dependent event $\op$ holds with \emph{overwhelming probability} if for all $D > 0$ large $\prob{\op} \geq 1 - N^{-D}$ for $N$ large enough.

The matrix model appearing in the main result will be a Gaussian perturbed version of a deterministic symmetric matrix satisfying Assumptions \ref{ass:eval} and \ref{ass:evec}.  Let $\GOE$ be a Gaussian orthogonal ensemble.  That is, $\GOE = (Z_{ij})_{i,j=1}^N$ is a symmetric matrix with rescaled entries $\sqrt{N} Z_{ij} / \sqrt{1+\delta_{ij}}$ being mutually independent and identically distributed standard Gaussian random variables for every $1 \leq i \leq j \leq N$.  Define the time $t \geq 0$ Gaussian perturbation of $\dat$ by $\dat(t) = \dat + \sqrt{t} \GOE$.  The Green's function and Stieltjes transform of the perturbed matrix are defined analogously
\begin{equation}
\grn(t;z) = (\dat(t) - z)^{-1} \quad \mbox{and} \quad m_N(t;z) = \frac{1}{N} \tr \grn(t;z)
\end{equation}
for all $z \in \HH$.  For $i \in [N]$, let $\lambda_i(t) \in \RR^N$ and $\vu_i(t) = (u_i^1(t), \ldots, u_i^N(t))^\top \in S^{N-1} \subset \RR^N$ denote the ordered eigenvalues and ($L^2$-normalized) eigenvectors of $\dat(t)$ respectively.  That is, $\{\lambda_i(t), \vu_i(t) | i \in [N], 0 \leq t \in \RR\}$ satisfy
\begin{enumerate}
\item $\dat(t) = \sum_{i=1}^N \lambda_i(t) \vu_i(t) \vu_i(t)^\top$ for all $t \geq 0$,
\item\label{item:ordered} $\lambda_i(t) \leq \lambda_{i+1}(t)$ for all $i \in [N-1]$ and all $t \geq 0$,
\item $\|\vu_i(t)\|_2 = 1$ for all $i \in [N]$ and $t \geq 0$.
\end{enumerate}
For each $t > 0$, the inequalities in item \ref{item:ordered} are almost surely strict and the collection $(\pm \vu_i(t) | i \in [N])$ of eigenvalues and eigenvectors up to $N$ possible sign changes $\vu_i(t) \mapsto - \vu_i(t)$ is almost surely unique because the space of real symmetric matrices admitting an eigenvalue of multiplicity greater than $1$ is not full rank and the distribution for $\dat(t)$ is absolutely continuous with respect to the Lebesgue measure.  The orthogonal matrix of eigenvectors will be referred to by $\EV(t) = (\vu_1(t), \ldots, \vu_N(t)) \in O(N)$ and $\EV = \EV(0)$.
\begin{convention}
The global sign of individual eigenvectors is not of concern in this paper, so each eigenvector is chosen independently uniformly at random from the orbit of the involution $\vu_i(t) \mapsto -\vu_i(t)$ for each $i\in[N]$.  To be precise fix a time $t \geq 0$ and let $(B_i)_{i=1}^N$ be $N$ independent and identically distributed uniform $\{-1,1\}$ Bernoulli random variables.  Then
\begin{equation}
\EE{\EV(t) | \GOE} = \EE{(B_1 \vu_1(t), \ldots, B_N \vu_N(t)) | \GOE}
\end{equation}
after conditioning on the time $t$ randomness induced by $\GOE$, there is additional randomness in the choice of sign for each eigenvector.  This symmetrizes the distribution of $\EV(t)$ to be invariant under the corresponding $(\ZZ/2)^N$ action on $O(N)$.  For our purposes, it forces all mixed multivariate moments with odd multiplicities on any eigenvector to vanish.  For example, if $\vv_1,\vv_2,\vv_3,\vv_4\in\RR^N$ are fixed vectors and $i \neq j \in [N]$ are distinct spectral indices, then $\EE{\ipr{\vu_i(t),\vv_1} \ipr{\vu_i(t),\vv_2} \ipr{\vu_i(t), \vv_3} \ipr{\vu_j(t), \vv_4}}=0$ where as $\EE{\ipr{\vu_i(t),\vv_1} \ipr{\vu_i(t),\vv_2} \ipr{\vu_j(t), \vv_3} \ipr{\vu_j(t), \vv_4}}$ does not necessarily vanish.
\end{convention}

The classical eigenvalue locations $\gamma_i(t)\in\RR$ for all $i \in [N]$ are defined via the Stieltjes transform of the additive free convolution $m_{\fc,t}(z) = m_{\fc,t}^{(N)}(z)$
\begin{equation}\label{eq:defnfc}
\gamma_i(t) = \inf \left\{ \gamma \in \RR : \frac{1}{\pi}\int_{-\infty}^\gamma \lim_{\eta \rightarrow 0^+} \Im m_{\fc,t}(E+i\eta) dE \geq \frac{i-1/2}{N}\right\} \quad \mbox{where} \quad m_{\fc,t}(z) = m_N(z + t m_{\fc,t}(z))
\end{equation}
is defined implicitly.  It is known that there exists a unique analytic solution to \eqref{eq:defnfc} on all of the upper half plane $\HH$ with a continuous extension to $\bar{\HH} = \HH \cup \RR$ for any $t > 0$.  See \cite{biane1997free} for details on the free convolution.  We further define an analogous free convolution analogue of the Green's function to simplify notation: 
\begin{equation}\label{eq:defn_Gfc}
\grn_{\fc,t}(z) = \grn(z + tm_{\fc,t}(z))
\end{equation}
which is also analytic in $z\in\HH$ and extends continuously to $\bar{H}$ for any $t > 0$.

\subsection{Statement of main results}

\begin{definition}
For $0 < \kappa < 1$, define the \emph{$\kappa$-truncated energy interval} by
\begin{equation}
\enint^\kappa = \enint_r^\kappa(E_0) = [E_0 - (1-\kappa)r, E_0 + (1-\kappa)r]
\end{equation}
where $E_0$ is the energy level and $r$ is the regularity scale introduced in Section \ref{ssec:model}.
\end{definition}

\begin{theorem}\label{thm:main}
Suppose $\dat$ is a symmetric $N \times N$ matrix satisfying Assumptions \ref{ass:eval} and \ref{ass:evec}.  Suppose $t = t(N)$ is a scale and $\oc > 0$ is a constant satisfying $\eta_* N^\oc < t < r N^{-\oc}$.  
Fix a constant $0 < \kappa < 1$ and a positive integer $n > 0$.  
Then there exists a (small) constant $\od = \od(\oc, n, \kappa)$ depending on $\oc$, $n$, and $\kappa$ such that the multidimensional eigenvector moments of the Gaussian perturbed matrix $\dat(t)$ satisfy
\begin{equation}
\sup_{ \substack{ i_1, \ldots, i_n \in [N] : \gamma_i(t) \in \enint^\kappa \\ \vv_1, \ldots, \vv_n \in S}} \left| \EE{ \prod_{a=1}^n \ipr{ \sqrt N \vu_{i_a}(t), \vv_k }} - \EE{ \prod_{a=1}^n \ipr{ \vh_{i_a}, \vv_a } } \right| < N^{-\od}
\end{equation}
for $N$ large enough.  The supremum is taken over all $n$-tuples of indices $(i_1,\ldots,i_n)\in[N]^n$ whose corresponding time $t$ classical eigenvalues lie in the $\kappa$-truncated energy interval, $\gamma_{i_a}(t) \in \enint^\kappa$, $a\in[n]$.  On the right hand side, for each $1 \leq i \leq N$, $\vh_i \in \RR^N$ are mutually independent centered Gaussian random vectors, $\vh_i \sim \normal(\zero, \cov_i)$ with covariance
\begin{equation}
\EE{\vh_i\vh_i^\top} = \cov_i := \frac{\Im \grn_{\fc,t}(\gamma_{i_a}(t))}{\Im m_{\fc,t}(\gamma_{i_a}(t))}.
\end{equation}
The terms appearing in the ratio defining the covariance matrix are limits
\begin{equation}
\ipr{\vv, \Im \grn_{\fc,t}(\gamma_i(t)) \vw} = \lim_{\HH \ni z \rightarrow \gamma_i(t)} \ipr{\vv, \Im \grn_{\fc,t}(z) \vw} \quad \mbox{and} \quad m_{\fc,t}(\gamma_i(t)) = \lim_{\HH\ni z\rightarrow \gamma_i(t)} m_{\fc,t}(z)
\end{equation}
which are guaranteed to exist and are finite.  The matrix $\Im \grn_{\fc,t}(\gamma_{i_a}(t))$ is symmetric with entries given by the imaginary parts of the corresponding entries in $\grn_{\fc,t}(\gamma_{i_a}(t))$.
\end{theorem}

\begin{remark}
See \cite{QUE} for the single eigenvector case.
\end{remark}

\begin{definition}
For the remainder of the paper, the time $t$ and truncation parameter $\kappa$ will always be understood to be from Theorem \ref{thm:main}.  Define the \emph{time-$t$ $\kappa$-truncated index interval} as 
\begin{equation}
\indint^\kappa = \indint^\kappa(s) = \{ i \in [N] : \gamma_i(t) \in \enint^\kappa\}.
\end{equation}
\end{definition}

We apply these results to the three popular random matrix models introduced above: generalized Wigner, $p$-sparse random graphs, $\alpha$-Levy random matrices.

\begin{theorem}\label{thm:genwig}
Suppose $\dat$ is a generalized Wigner matrix and fix $\alpha \in (0,1/2)$ small.  Then for every polynomial $P$ of $m$ variables,
\begin{equation}
\sup_{\substack{\alpha N \leq i_a \leq (1-\alpha)N : a \in [m] \\ v_a \in S^{N-1} : a \in [m]}} |\EE{ P((\sqrt{N} \ipr{ \vu_{i_a}, \vv_a })_{a =1 }^m)} - \EE{ P((\ipr{ \vg_{i_a}, \vv_a })_{a=1}^m)}| \leq N^{-\od}
\end{equation}
where $(\vg_i)_{i=1}^N$ are independent and identically distributed standard Gaussian random vectors in $\RR^N$ and $\od > 0$ is a constant depending only on $P$ and $\alpha$.
\end{theorem}

\begin{theorem}\label{thm:p-reg}
Let $\dat$ be the normalized adjacency matrix of a sparse Erd\H{o}s--R\'enyi graph $G(N, p/N)$ with sparsity $N^\delta \leq p \leq N/2$ or the adjacency matrix of a random $p$-regular graph with sparsity $N^\delta \leq p \leq N^{2/3 - \delta}$.  Then 
\begin{equation}
\sup_{\substack{a \in [m] \\ \alpha N \leq i_a \leq (1-\alpha)N \\ v_a \in S^{N-1} \cap \ve^\perp}} |\EE{ P(\sqrt{N} u_{i_a} \cdot v_a)} - \EE{ P(N_{i_a} \cdot v_a)}| \leq N^{-\od}
\end{equation}
where $\ve = (1,1,\ldots,1)^\top = \sum_{i=1}^N \ve_i$ is the constant all-ones vector, $(\vg_i)_{i=1}^N$ are independent and identically distributed standard Gaussian random vectors in $\RR^N$, and $\od > 0$ is  a constant depending only on $P$, $\delta$, and $\alpha$.  
\end{theorem}

The following theorem on eigenvector distributions in L\'evy matrices generalizes the main results \cite[Theorems 2.7 and 2.8]{aggarwal2020eigenvector} by replacing eigenvector moment flow with colored eigenvector moment flow in the dynamical step of the proof.  In particular, Theorem 2.7 computes the distribution for a column in the $(n \times n)$-submatrix $\mat = (N u_{i_a}^{k_b})_{a,b=1}^n$ for some spectral indices $i_1,\ldots,i_n\in[N]$ and some directional indices $k_1,\ldots,k_n\in[N]$.  On the other hand, Theorem 2.8 computes the distribution for a row in the same $(n \times n)$-submatrix $\mat$.  The new dynamics, allows us to compute the distribution of the entire submatrix $\mat$.

\begin{theorem}\label{thm:levy}
Suppose $\dat$ is an $\alpha$-L\'evy matrix as defined in Definition \ref{defn:levy}.  There is a countable set $\countable \subset (0,2)$ such that for every $\alpha \in (0,2) \backslash \countable$ there is a constant $c(\alpha) > 0$ so that the following holds.  Fix a positive integer $n > 0$, a sequence of spectral indices $i_1 < i_2 < \ldots < i_n \in [N]$, and a sequence of test directions $k_1 < k_2 < \ldots < k_n \in [N]$.  Moreover, the spectral indices should satisfy $|i_k-i_1|\leq N^{1/2}$ for every $2 \leq i \leq n$ and $\lim_{N \rightarrow \infty} \gamma_{i_1} = E$ for some $E \in [-c(\alpha), c(\alpha)]$.  Then we have convergence of random $(n \times n)$-matrices
\begin{equation}
(\sqrt N u_{i_a}^{k_b})_{a,b=1}^n \rightarrow (g_{ab} \levlim_b(E))_{a,b=1}^n
\end{equation}
in mixed moments where $\{g_{ab}|a,b\in[n]\}$ are standard Gaussian random variables and $\{\levlim_a(E)|a\in[n]\}, a=1, \ldots n,$ are random variables each distributed with the law denoted by  $\levlim_*(E)$.  All $n^2 + n$ random variables are mutually independent. 
The  probability distribution of  $\levlim_*(E)$,  defined in \cite[Definition 2.5]{aggarwal2020eigenvector} , 
is given by the density of states  of an operator on an infinite tree.
\end{theorem}

\section{Free convolution and local law consequences}\label{sec:fc}

With $t$, $\oc$, and $\kappa$ as in Theorem \ref{thm:main}, this section organizes the local laws and relevant consequences that will be used in computations throughout proofs in Sections \ref{sec:l2} and \ref{sec:proof}.
For eigenvalues, there is tight deterministic control on the free convolution Stieltjes transform and classical positions.  
These approximate the true Stieltjes transform and eigenvalue positions with overwhelming probability.
For eigenvectors, the free convolution Green's function approximates the true Green's function with overwhelming probability.
One consequence is that all eigenvectors are delocalized in the regular set of directions $S$ from Assumption \ref{ass:evec}.
Another consequence is Lipschitz continuity for both the free convolution Stieltjes transform $m_{\fc,s}(z)$ and the free convolution Green's function in the regular set of directions $\ipr{\vv,\grn_{\fc,s}(z)\ww}$, $\vv,\ww \in S$.  This Lipschitz continuity holds in both spectral parameter $z\in\HH$ bounded away from $\RR$ and time $s$ bounded away from $0$.

First, we give deterministic properties for the Stieltjes trarnsform of the free convolution and the classical locations.  These first two results can be found in \cite{landon2017convergence}.

\begin{proposition}[Regularity of the free convolution]\label{prop:reg_fc}
Suppose $\dat$ satisfies Assumption \ref{ass:eval}.  Then there exists a constant $C>0$ depending only on $\oa$ such that, the Stieltjes transform satisfies
\begin{equation}
C^{-1} \leq \Im m_{\fc,s}(z) \leq C
\quad \mbox{and} \quad
|m_{\fc,s}(z)| \leq C \log N
\end{equation}
uniformly for time $t/2 \leq s \leq t$ and spectral parameter $z = E+i\eta$ with $E\in\enint^{\kappa/10}$ and $0 \leq \eta \leq 1-\kappa r$, while the classical locations satisfy
\begin{equation}\label{eq:classicalspeed}
|\partial_s \gamma_i(t)| \leq C \log N
\end{equation}
uniformly for $i \in \indint^{\kappa/10}$ and time $t/2 \leq s \leq t$.
\end{proposition}
\begin{remark}
As a consequence, if $\gamma_i(s) \in \enint^{\kappa'}$ for some $t/2 \leq s \leq t$ and some $\kappa'>\kappa$, then $i\in\gamma\in\indint^\kappa$.
\end{remark}
\begin{definition}
For every small $\varepsilon>0$, define the \emph{$\varepsilon$-truncated spectral domain}
\begin{equation}
\dom_\varepsilon = \{z = E+i\eta \in \HH | E \in \enint^{\varepsilon\kappa}, N^{-1+\varepsilon} \leq \eta \leq 1-\kappa r\} \subset \HH.
\end{equation}
\end{definition}
Next we have overwhelming probability estimates on eigenvalue statistics of the perturbed matrix.

\begin{proposition}[Regularity of eigenvalues]\label{prop:reg_eval}
Suppose $\dat$ satisfies Assumption \ref{ass:eval} and fix positive constants $\varepsilon, \ob > 0$.  Then with overwhelming probability the following two estimates hold: the Stieltjes transform satisfies
\begin{equation}
|m_N(s;z) - m_{\fc,s}(z)| \leq N^\ob (N\eta)^{-1}
\end{equation}
uniformly for $z \in \dom_\varepsilon$ and $t/2 \leq s \leq t$, while individual eigenvalues satisfy
\begin{equation}\label{eq:rigidity}
|\lambda_i(s) - \gamma_i(s)| \leq \frac{N^\ob}{N}
\end{equation}
uniformly for $i \in \indint^{\kappa/10}$ and $t/2 \leq s \leq t$.
\end{proposition}
For a proof, see \cite[Proposition 2.2]{bourgade2017eigenvector}.  We also have overwhelming probability estimates on the quadratic form arising from the Green's function.  The following isotropic law is \cite[Theorem 2.1]{bourgade2017eigenvector}.

\begin{proposition}[Isotropic local law]\label{prop:isotropiclaw}
Suppose $\dat$ satisfies Assumption \ref{ass:eval} and fix two small positive constants $\varepsilon, \ob > 0$.  Then for any $\vv \in S^{N-1}$, with overwhelming probability
\begin{equation}
\left| \ipr{ \vv, (\grn(s;z) - \grn_{\fc,s}(z)) \vv } \right| \leq \frac{N^\ob}{\sqrt{ N\eta }} \Im \ipr{\vv, \grn_{\fc,s}(z) \vv }
\end{equation}
uniformly for $z \in \dom_\varepsilon$ and $t/2 \leq s \leq t$.
\end{proposition}

The first is a classical application to bounds on the Green's functions: delocalization.

\begin{corollary}[Delocalization]\label{cor:deloc}
Suppose $\dat$ satisfies Assumptions \ref{ass:eval} and \ref{ass:evec} and fix $\vv \in S$, where $S$ is the set of regular unit vectors from Assumption \ref{ass:evec}. For every $\ob > 0$, with overwhelming probability
\begin{equation}
|\ipr{\vu_i(s), \vv}|^2 \leq \frac{N^\ob}{N}
\end{equation}
uniformly for $i \in \indint^{\kappa/10}$ and $t/2 \leq s \leq t$.
\end{corollary}

\begin{proof}
Set $z_i = \lambda_i(s) + i N^{-1+\ob/2}$.  By the spectral decomposition of $\grn(t;z)$, Proposition \ref{prop:isotropiclaw}, and Assumption \ref{ass:evec}
\begin{equation}
|\ipr{\vu_i(s), v}|^2 \leq N^{-1+\ob/2} \Im \ipr{ \vv, \grn(s; z_i) \vv } \leq N^{-1+\ob/2}(1 + \frac{N^{\ob/5}}{\sqrt{N^{\ob/2}}}) \Im \ipr{ \vv, \Im \grn_{\fc,s}(z_i) \vv } \leq N^{-1+\ob}
\end{equation}
with overwhelming probability when $z_i \in \dom_{\min(\ob/2,\kappa/10)}$.  This holds with overwhelming probability by the bound on classical eigenvalues \eqref{eq:classicalspeed} and eigenvalue rigidity \eqref{eq:rigidity}.  The second inequality holds with overwhelming probability by Proposition \ref{prop:isotropiclaw} and the last inequality comes from \eqref{eq:evec1}.
\end{proof}

Lastly, the regularity assumptions on the Green's function and Stieltjes transform and time $0$ imply smoothness for $\ipr{\vv,\grn_{\fc,s}(z)\vw}$ and $m_{\fc,s}(z)$ in $(s,z) \in \RR_+ \times \RR_+$ away from the boundary.

\begin{proposition}\label{prop:reg_Gm}
Suppose $\dat$ satisfies Assumptions \ref{ass:eval} and \ref{ass:evec} and $\kappa$ and $t$ are fixed as above.  
There exists a constant $C>0$ such that for all $\ob>0$, the following bounds hold uniformly for time $t/2 \leq s \leq t$ and spectral parameter $z=E+i\eta$ with $E \in \enint^\kappa$ and $0 \leq \eta < s$.
\begin{equation}\label{eq:reg_mfc}
|\partial_z m_{\fc,s}(z)| \leq \frac{C}{s} \quad \mbox{and} \quad |\partial_s m_s(z)| \leq \frac{C\log(N)}{s}
\end{equation}
and similarly
\begin{equation}\label{eq:reg_Gfc}
|\partial_z \ipr{\vv, \grn_{\fc,s}(z) \vw}| \leq \frac{N^\ob}{s} \quad \mbox{and} \quad |\partial_s m_s(z)| \leq \frac{N^\ob}{s}
\end{equation}
for every $\vv,\vw\in S$, the set of regular unit vectors.
\end{proposition}

\begin{proof}
See \cite[Section 7.1]{landon2017convergence} for a proof that
\begin{equation}\label{eq:reg_mfcz}
|\partial_z m_{\fc,s}(z)| \leq \frac{C}{s}.
\end{equation}
Differentiating the definition of $m_{\fc,s}(z)$ in \eqref{eq:defnfc} yields the identity
\begin{equation}
\partial_t m_{\fc,s}(z) = \frac{1}{2} \partial_z\left(m_{\fc,s}(z)\left(m_{\fc,s}(z) + z\right) \right).
\end{equation}
This identity also appears in \cite[Section 7.1]{landon2017convergence}.  Proposition \ref{prop:reg_fc} together with \eqref{eq:reg_mfcz} give the second bound in \eqref{eq:reg_mfc}.  It remains to prove \eqref{eq:reg_Gfc}.  For this, appeal to the algebraic identities on the level of general Green's functions: $\partial_z \grn(z) = \grn(z)^2$ and $\Im \grn(z) = \eta |\grn(z)|^2 := \grn(z)^*\grn(z)$.  Let $\grn'$ denote the complex matrix-valued function $\grn'(z) = \partial_z \grn(z)$.  Then using the identites, we obtain the bound
\begin{equation}\label{eq:reg_Gprime}
\begin{aligned}
|\ipr{\vv, \grn'(z+sm_{\fc,s}(z))\vw}| 
&= |\ipr{\vv, \grn(z+sm_{\fc,s}(z))^2 \vw}| \\
&\leq |\ipr{\vv, \grn(z+sm_{\fc,s}(z))^2 \vv}| + |\ipr{\vw, \grn(z+sm_{\fc,s}(z))^2 \vw}|\\
&\leq \ipr{\vv, |\grn(z+sm_{\fc,s}(z))|^2 \vv} + \ipr{\vw, |\grn(z+sm_{\fc,s}(z))|^2 \vw}\\
&= \frac{1}{\eta + s \Im m_{\fc,s}(z)} \left( \ipr{\vv, \Im \grn(z+sm_{\fc,s}(z)) \vv} + \ipr{\vw, \Im \grn(z+sm_{\fc,s}(z)) \vw} \right) \\
&\leq \frac{N^\ob}{\eta + s \Im m_{\fc,s}(z)} \leq \frac{N^\ob}{s}.
\end{aligned}
\end{equation}
The first inequality is Cauchy-Schwarz applied to both the real and imaginary parts separately. The second inequality is the triangle inequality applied to the spectral decomposition of $\grn(z+s m_{\fc,s}(z))$.  The third inequality is from and Assumption \ref{ass:evec}.  The fourth is from Proposition \ref{prop:reg_fc} reducing $\ob$ if necessary.

Differentiating the definition of $\grn_{\fc,s}(z)$ from \eqref{eq:defn_Gfc} gives
\begin{equation}\label{eq:reg_Gz_diff}
\partial_z \grn_{\fc,s}(z) = G'(z+s\Im m_{\fc,s}(z)) \left(1 + s\partial_z m_{\fc,s}(z) \right).
\end{equation}
Combining \eqref{eq:reg_mfcz}, \eqref{eq:reg_Gprime}, and \eqref{eq:reg_Gz_diff} yields
\begin{equation}
|\partial_z \ipr{\vv, \grn_{\fc,s}(z) \ww}| 
= |\ipr{\vv, G'(z+s\Im m_{\fc,s}(z)) \ww}| \left|1 + s\partial_z m_{\fc,s}(z) \right| 
\leq \frac{N^\ob}{s}
\end{equation}
where again constants are absorbed into the $N^\ob$ factor.  Similarly,
\begin{equation}\label{eq:reg_Gs_diff}
\partial_s \grn_{\fc,s}(z) = G'(z+s\Im m_{\fc,s}(z)) \left(m_{\fc,s}(z) + s\partial_s m_{\fc,s}(z) \right)
\end{equation}
so combining \eqref{eq:reg_mfcz}, \eqref{eq:reg_Gprime}, and \eqref{eq:reg_Gz_diff} yields
\begin{equation}
|\partial_s \ipr{\vv, \grn_{\fc,s}(z) \ww}| 
= |\ipr{\vv, G'(z+s\Im m_{\fc,s}(z)) \ww}| \left|m_{\fc,s}(z) + s\partial_s m_{\fc,s}(z) \right|
\leq \frac{N^\ob}{s}
\end{equation}
by Proposition \ref{prop:reg_fc} and \eqref{eq:reg_mfc}.  This time, the $\log(N)$ factor is absorbed into the $N^\ob$ factor.
\end{proof}

\section{Colored particle jump process}\label{sec:jump}

To prove Theorem \ref{thm:main}, we employ the renormalization strategy outlined in the introduction.  In this section, the \emph{distinguishable particle configurations} are introduced.  The SEE paired against specific multivariate moment test functions induces a the \emph{colored eigenvector moment flow} (CEMF) which generates a (non-stochastic) process on the configuration space.  The remainder of the section is focused on proving a variety of algebraic and positivity properties of the CEMF.

\subsection{Colored eigenvector moment flow}

\begin{definition}\label{def:distinguishable}
A \emph{distinguishable particle configuration} is a lattice vector $\xx = (x_1, \ldots, x_n)^\top \in [N]^n$, interpreted as a collection of $n$ labeled particles, each with a position and a unique label.  Particles are labeled by $a,b,c \in [n] = \{1,\ldots,n\}$ and may be positioned on the sites $i,j,k \in [N] = \{1, \ldots, N\}$ forming a finite integer lattice.  In this interpretation, particle $a$ is positioned at $x_a$ for each $a \in [n]$.  In this paper, labels are depicted as colors.

\emph{Particle number operators} specify how many particles are positioned at a given site.  These operators are defined by
\begin{equation}
n_i(\xx) = |\{a \in [n] | x_a = i\}|
\end{equation}
for all $i \in [N]$.  We say that a distinguishable particle configuration $\xx \in [N]^n$ is \emph{even} if each site is occupied by an even number of particles.  That is, $n_i(\xx)$ is even for all $i \in [N]$.  It will become apparent that our dynamics preserves the parity of all particle numbers, $n_i(\xx) \mod 2$, for all $i \in [N]$. In particular, the set of all even partitions form a closed system which will be denoted by $\Ln \subset [N]^n$
\begin{equation}
\Ln = \{\xx \in [N]^n | n_i(\xx) \mbox{ is even for all } i \in [N]\}.
\end{equation}
\end{definition}

\begin{remark}
Throughout, $N$ denotes matrix size and number of sites while $n$ denotes the degree of eigenvector component moment and total particle number respectively in the matrix and particle configuration settings.  For every $\xx \in \Ln$, $n = \sum_{i=1}^N n_i(\xx)$.
\end{remark}

\begin{remark}\label{rmk:distinguishable}
We use the term \emph{distinguishable} to highlight the contrast between the distinguishable particle configurations introduced in Definition \ref{def:distinguishable} and their \emph{indistinguishable} counterparts introduced in \cite[Section 3.2]{QUE}.  In that context, [indistinguishable] particle configurations are given by $\bet : [N] \rightarrow \NN$ where $\eta_j = \bet(j)$ is interpreted as the number of particles at site $j$.

For future reference, the configuration space of indistinguishable $(n/2)$-particle configurations will be denoted $\On$.  It contains all indistinguishable particle configurations constrained to have a total particle number of $n/2$, $\sum_{j=1}^N \eta_j = n/2$.
When drawing comparisons between the two configuration spaces, the relation will always be through the projection map $\forget : \Ln \rightarrow \On$, called the \emph{colorblind map}, given by $\xx \mapsto \bet$ where $\eta_j = n_j(\xx)/2$ for all $j \in [N]$.  This map essentially forgets the $[n]$-labeling on particles.
\end{remark}

\begin{definition}[Colored eigenvector moment observable]\label{defn:levmo}
Fix $n$ unit vectors $\vv_1, \ldots, \vv_n \in \RR^N$, with $\|\vv_a\|_2 = 1$ for all $a \in [n]$, and denote the entire collection by $\Vecs = (\vv_1, \ldots, \vv_n) \in \RR^{N \times n}$. Fix also an initial symmetric matrix $\dat$ and let $\Bmat(s) = (B_{ij}(s))_{i,j}^N$ be a matrix of $N^2$ independent and identically distributed Brownian motions $B_{ij}(s)$, $0 \leq s \leq t$.  Consider the ordered and normalized spectral decomposition 
\begin{equation}
\dat + \frac{\Bmat(s) + \Bmat(s)}{\sqrt{2N}} = \sum_{i=1}^N \lambda_i(s) \vu_i(s) \vu_i(s)^\top,
\end{equation}
where for every $s \geq 0$, $\lambda_i(s) \leq \lambda_{i+1}(s)$ when $i \in [N-1]$ and $\|\vu_i(s)\|_2=1$ for all $i \in [N]$.  Again, this decomposition is also almost surely unique at all times $s > 0$ (in fact, the left hand side has the same distribution as $\dat + \sqrt{s}\GOE$ for a GOE $\GOE$; see the remarks below for more details).  Denote the collection of full time eigenvalue trajectories by $\bl = (\lambda_i(s) | i \in [N], s \in [0, t])$.
The \emph{colored eigenvector moment observable} $f_s : \Ln \rightarrow \RR$ is defined by integrating out most of the randomness from $\Bmat(s)$
\begin{equation}
f_s(\xx) = f_s(\xx; \dat, \Vecs, \bl) = \left( \prod_{i=1}^N n_i(\xx)!! \right)^{-1} \EE{\left. N^{n/2} \prod_{a=1}^n \ipn{ \vu_{x_a}(s), \vv_a } \right| \bl}
\end{equation}
for every $\xx \in \Ln$.  Here $k!! = (k-1)(k-2)!!$ with $0!!=1$ and $1!! = 0$.
\end{definition}

\begin{remark}
Let $\GOE$ be a GOE and note that the two noise matrices 
\begin{equation}
\sqrt{t} \GOE \overset{d}{\sim} (\Bmat(t)+\Bmat^\top(t))/\sqrt{2N}
\end{equation}
share the same distribution despite admitting different covariance structures through time.  In particular, the two sets of spectral statistics of the corresponding time-$t$ $\dat$-perturbations share the same joint distributions on $(\lambda_i(t), \vu_i(t))_{i=1}^N$.  This means that, after integrating out the $\bl$ randomness, the colored eigenvector moment observable takes a form of interest from the context of Theorem \ref{thm:main}
\begin{equation}
\left( \prod_{i=1}^N n_i(\xx)!! \right) \EE{ f_t(\xx; \dat, \Vecs, \bl) } = \EE{ N^{n/2} \prod_{a=1}^n \ipr{ \vu_{x_a}(t), \vv_a } }.
\end{equation}
\end{remark}

\begin{remark}
The stochastic process describing $\bl$ is well understood and goes by the name of Dyson Brownian motion.  Some relevant properties which hold almost surely are: $\lambda_i(s) < \lambda_{i+1}(s)$ for all $i \in [N-1]$ and all $s > 0$ and $\lambda_i$ is continuous in $s \geq 0$ for every $i \in [N]$.
\end{remark}

\begin{remark}
In the context of the colored eigenvector moment observable, each particle corresponds to a factor in the moment being computed --- an eigenvector component.  The label of the particle $a \in [n]$ corresponds to a direction, namely the unit vector $\vv_a \in \RR^n$, while the position corresponds to a spectral index, namely the eigenvector $\vu_{x_a}$ to be tested.

The normalizing factor $\prod_{i=1}^n n_i(\xx)!!$ would be the corresponding Gaussian moment if each $\sqrt{N}\vu_i$ were independent and identically distributed standard Gaussian random vectors in $\RR^N$ and the test vectors $\vv_1, \ldots \vv_n$ were orthonormal.  Neither of those two assumptions must be even approximately true.  In fact, see Definition \ref{def:ansatz} for the ansatz limiting short time observable with arbitrary initial data $\dat$ and possibly non-orthogonal vectors $\vv_1, \ldots, \vv_n$.
\end{remark}

\begin{theorem}\label{thm:levmf}
As the matrix entries of $\dat$ undergo symmetric matrix-valued Brownian motion and the eigenvalues follow trajectories $\tbl$, the time derivative of the colored eigenvector moment observable satisfies
\begin{equation}
\partial_s f_s(\xx) = \gen_s f_s(\xx) = \sum_{1 \leq i < j \leq j} c_{ij}(s) \gen_{ij} f_s(\xx), \quad c_{ij}(s) = N^{-1} \left(\lambda_i(s) - \lambda_j(s)\right)^2
\end{equation}
where $\gen_{ij} = \move_{ij} - \exch_{ij}$ is decomposed as the difference between the move operator
\begin{equation}
\move_{ij} f(\xx) = \frac{n_j(\xx)+1}{n_i(\xx)-1} \sum_{a \neq b \in [n]} \left( f(m_{ab}^{ij}\xx) - f(\xx) \right) + \frac{n_i(\xx)+1}{n_j(\xx)-1} \sum_{a \neq b \in [n]} \left( f(m_{ab}^{ji}\xx) - f(\xx) \right)
\end{equation}
and the exchange operator
\begin{equation}
\exch_{ij}f(\xx) = 2 \sum_{a \neq b \in [n]} \left( f(s_{ab}^{ij}\xx) - f(\xx) \right)
\end{equation}
respectively.  For each pair of labels $a,b\in[n]$ and each pair of sites $i,j\in[N]$ the two particle jump and swap operators are respectively $m_{ab}^{ij}$ and $s_{ab}^{ij}$ defined as mappings $m_{ab}^{ij}, s_{ab}^{ij} : \Ln \rightarrow \Ln$ of particle configurations by
\begin{equation}
m_{ab}^{ij}\xx = \xx + \one{x_a=x_b=i}(j-i)(\ve_a + \ve_b) \quad \mbox{and} \quad s_{ab}^{ij}\xx = \xx + \one{x_a=x_b=i}(j-i)(\ve_a - \ve_b).
\end{equation}
Moreover the reversible measure for this generator is 
\begin{equation}
\pi(\xx) = \prod_{i=1}^N (n_i(\xx)!!)^2
\end{equation}
meaning that 
\begin{equation}
\sum_{\xx \in \Ln} \pi(\xx) f(\xx) \gen_s g(\xx) = \sum_{xx \in \Ln} \pi(\xx) g(\xx) \gen_s f(\xx)
\end{equation}
for every $s \geq 0$ and test functions $f,g : \Ln \rightarrow \RR$.
\end{theorem}

\begin{remark}
Consider any two distinct labels $a \neq b \in [n]$ and any two sites $i,j \in [N]$.
The two particle jump operator acts by moving particles $a$ and $b$ from site $i$ to site $j$ when possible:
\begin{equation}
m_{ab}^{ij} \xx = (x'_1, \ldots, x'_n)^\top \quad \mbox{where} \quad x'_c = \begin{cases} j & \mbox{ if } c \in \{a,b\} \mbox{ and } x_a = x_b = i \\ x_c &\mbox{ otherwise.} \end{cases}
\end{equation}
Similarly, the two particle swap operator acts by swapping the locations of particle $a$ originally at site $i$ with particle $b$ originally at site $j$ when possible:
\begin{equation}
s_{ab}^{ij} \xx = (x''_1, \ldots, x''_n)^\top \quad \mbox{where} \quad x''_c = \begin{cases} j & \mbox{ if } c=a \mbox{ and } x_a = i \mbox{ and } x_b=j \\ i &\mbox{ if } c=b \mbox{ and } x_a=i \mbox{ and } x_b=j \\ x_c &\mbox{ otherwise.} \end{cases}
\end{equation}
\end{remark}

\begin{proof}
See Appendix \ref{app:derivedynamics} for specific computations and the derivation of the reversible measure.  The colored eigenvector moment flow derivation is outlined here by following a sequence of four steps, each computing a derivative from the previous step / applying It\^o's formula.  Note that Steps 1, 2, and 3 have been standard since the introduction of the SEE in \cite{QUE}, but are included here to provide the complete sequential derivation as certain quantities in these steps will be referred to later.  The key algebraic novelty is identity \eqref{eq:test_poly} in Step 4.
\begin{enumerate}[Step 1]
\item Let $\dat = (h_{ij})_{i,j=1}^N$ be a symmetric matrix with increasing eigenvalues $\lambda_i$ and corresponding eigenvectors $\vu_i = (u_i^\alpha)_{\alpha =1}^N$, $i \in [N]$.  The matrix of eigenvectors is denoted $\EV = (\vu_1, \ldots, \vu_N)$.  Run independent and identically distributed symmetric Brownian motions on $h_{ij}$: $dh_{ij}(s) = dB_{ij}(s)$ with $B_{ij}(s) = B_{ji}(s)$, $\EE{ B_{ij}(s) } = 0$, and $\EE{ B_{ij}(s)^2 } = (1+\delta_{ij})s/N$.
\item Recover the induced stochastic processes on spectral quantities: Dyson Brownian motion 
\begin{equation}
d\lambda_i(s) = \frac{dW_{ii}(s)}{\sqrt{N}}  + \frac{1}{N} \sum_{j \neq i} \frac{1}{\lambda_i(s) - \lambda_j(s)} dt
\end{equation}
and the stochastic eigenstate equation
\begin{equation}\label{eq:edbm}
d u_i^\alpha(s) = \frac{1}{\sqrt{N}} \sum_{j \neq i} \frac{dW_{ij}(s)}{\lambda_i(s) - \lambda_j(s)}u_j^\alpha(s) - \frac{1}{2N} \sum_{j \neq i} \frac{dt}{\left(\lambda_i(s) - \lambda_j(s)\right)^2}u_i^\alpha(s)
\end{equation}
where $W_{ij}(s) = W_{ji}(s)$ and $W_{ij}(s) / \sqrt{1+\delta_{ij}}$ are independent and identically distributed standard Brownian motions for each $i \leq j$.
\item Obtain the associated generator for eigenvector Dyson Brownian motion.  This is a second order translation invariant differential operator on the Lie group $SO(N)$.  The Lie algebra $\mathfrak{spreviouso}(N)$ is the algebra of antisymmetric $N \times N$ matrices.  An orthonormal basis for the Lie algebra $\mathfrak{so}(N)$ is $\{\soB_{ij} = \ve_i \ve_j^\top - \ve_j \ve_i^\top | 1 \leq i < j \leq N\}$.  The associated first order differential operators (or equivalently, left invariant vector fields on $SO(N)$) are $\soB_{ij} = \vu_j \cdot \partial_{\vu_i} - \vu_i \cdot \partial_{\vu_j}$ in the sense that for any smooth $f : SO(N) \rightarrow \RR$, we have
\begin{equation}
\soB_{ij} f(\EV) = \sum_{\alpha=1}^N u^\alpha_j \partial_{u^\alpha_i} f(u) - u^\alpha_i \partial_{u^\alpha_j} f(\EV).
\end{equation}
The generator for eigenvector Dyson Brownian motion is given by the elliptic operator 
\begin{equation}\label{eq:gen_edbm}
\partial_s \EE{f(\EV(s))} = \EE{\sum_{1 \leq i < j \leq N} c_{ij} (\soB_{ij}^2 f)(\EV(s))}
\end{equation}
with coefficients $c_{ij}(s) = N^{-1} \left(\lambda_i(s) - \lambda_j(s)\right)^{-2}$.
\item Applying this generator to a polynomial of eigenvector components generates an action on the polynomials exponents described by $\gen_{ij} = \move_{ij} - \exch_{ij}$.  More specifically, for each $\xx \in \Ln$, define the polynomial over $\EV = (u_{i}^\alpha)_{i,\alpha=1}^N \in SO(N)$ and $\Vecs = (\vv_1, \ldots, \vv_n) \in \RR^{n \times N}$, $\vv_a = (v_a^\alpha)_{\alpha = 1}^N$ for each $a \in [n]$, by
\begin{equation}\label{eq:test_poly}
P(\xx,\EV,\Vecs) = \frac{\prod_{a=1}^n \sum_{\alpha=1}^N v_a^\alpha u_{x_a}^\alpha}{\prod_{i=1}^N n_i(\xx)!!}.
\end{equation}
Then $\soB_{ij}^2 P(\xx,\EV,\Vecs) = \gen_{ij} P(\xx,\EV,\Vecs)$ where $\soB_{ij}$ is acting on the $\EV$ variables and $\gen_{ij}$ is acting on the $\xx$ variables.
\end{enumerate}
\end{proof}

\begin{definition}[Permutation action]\label{def:actn}
Let $S_n$ be the symmetric group on $[n]$.  There is a natural action, denoted by $\actn$, of $S_n$ on the configuration space $\Ln$ of even distinguishable particle configurations given by permuting labels.  That is, for all $\sigma \in S_n$, $\xx \in \Ln$, and $a \in [n]$, we have
\begin{equation}
\sigma \actn \xx = (x_{\sigma(1)}, \ldots, x_{\sigma(n)})^\top.
\end{equation}
A \emph{perfect matching} on $[n]$ is a fixed point free involution in $S_n$.  This means $\sigma \in S^n$ is a perfect matching if and only if $\sigma^2 (a) = a \neq \sigma(a)$ for all $a \in [n]$.  Denote the set of perfect matchings by $M_n \subset S_n$.  Say that a particle configuration $\xx \in \Ln$ is \emph{stabilized} by $\sigma \in S_n$ if $\sigma \actn \xx = \xx$ under this action and let $\Stab_{S_n}(\xx)$ be the subgroup of all permutations that stabilize $\xx$.  Note that the cardinality of the set of perfect matchings stabilizing a specified particle configuration $\xx \in \Ln$ is $|M_n \cap \Stab_{S_n}(\xx)| = \sqrt{\pi(\xx)} = \prod_{i=1}^N n_i(\xx)!!$ for all $\xx \in \Ln$.
\end{definition}

\begin{definition}[Ansatz]\label{def:ansatz}
For every $\yy \in \Ln$, the \emph{ansatz observable} from the perspective of $\yy$ is given by
\begin{equation}
F(\xx; \yy) = \frac{1}{\sqrt{\pi(\xx)}} \sum_{\sigma \in M_n \cap \Stab_{S_n}(\xx)} \sqrt{\prod_{a=1}^n \ipr{ \vv_a, \frac{1}{2}\left( \cov_{y_a} + \cov_{y_{\sigma(a)}} \right)\vv_{\sigma(a)}}}
\end{equation}
where the covariance matrices $\cov_i$, for all $i\in[N]$, are given by 
\begin{equation}
\cov_i = \frac{\Im \grn_{\fc,t}(\gamma_{y_a}(t))}{\Im m_{\fc,t}(\gamma_{y_a}(t))}
\end{equation}
and $t$ is the time specified in Theorem \ref{thm:main}.
\end{definition}

\begin{remark}
We motivate this definition now by evaluating the eigenvector moment observable using several approximations, then continue to outline the proof of Theorem \ref{thm:main}.
The eigenvector moment observable at time $t$ from Theorem \ref{thm:main} is given by
\begin{equation}
f_t(\xx) = \frac{N^{n/2}}{\sqrt{\pi(\xx)}} \EE{ \prod_{a=1}^n \ipr{\vu_{x_a}(t), \vv_a}}.
\end{equation}
The long time equilibrium will be described by the eigenvector matrix $\EV$ approaching Haar distribution.  In particular, let $\orth \sim SO(N)$ be Haar distributed and $\vorth_i$ be column $i$ of $\orth$.  Then the global equilibrium state of the eigenvector moment observable will be 
\begin{equation}
\kproj f_t(\xx) = \pi(\xx)^{-1/2} \EE{\prod_{a=1}^n \ipr{\vorth_{x_a}, \vv_a}}.
\end{equation}
As Haar and spherical distributions are somewhat difficult to work with algebraically, we use a Gaussian approximation for each column individually: $\vorth_i \approx \vg_i$ where $\vg_i \sim \normal(0, \iden_N)$ are independent and identically distributed standard Gaussian vectors.  The local short time equilibrium is not yet completely mixed as Haar on $SO(N)$ but still exhibits the Gaussian behavior, except with nonstandard covariance.  
The local ansatz is that eigenvectors are mutually independent and distributed as $\vu_i(t) \sim \normal(0, N^{-1}\cov_i(t))$.  It is convenient to write this ansatz as a $\vu_i(t) = N^{-1/2} \cov_i(t)^{1/2} \vg_i$.  Refer to \eqref{eq:cov_ansatz} for heuristic reasoning on why this should be the case.  Thus, our \emph{base ansatz observable} will be
\begin{equation}\label{eq:ansatz_base}
\bar{F_t}(\xx) = \frac{1}{\sqrt{\pi(\xx)}} \EE{ \prod_{a=1}^n \ipr{\vg_{x_a}, \cov_{x_a}^{1/2} \vv_a} }.
\end{equation}

Unfortunately, such a simple expression will not be time invariant with the eigenvector moment flow dynamics.  The primary obstruction to invariance are the $x_a$ parameters appearing on the right hand side of the inner products.  If the right side was $\xx$-independent, then the above expression would be invariant as the terms on the left side of the inner product are all standard Gaussians, and hence rotationally invariant, and hence Haar invariant.  Therefore, we generalize this base ansatz observable by introducing a second particle configuration, therein decoupling the left and right sides of the inner product 
\begin{equation}\label{eq:ansatz_decouple}
\bar{F}(\xx,\yy) = \frac{1}{\sqrt{\pi(\xx)}} \EE{ \prod_{a=1}^n \ipr{\vg_{x_a}, \cov_{y_a}^{1/2} \vv_a} }.
\end{equation}
We think of the first variable $\xx$ as specifying the particle color profile and the second variable $\yy$ as specifying the particle locations, or equivalently the covariance structures.
As desired, this \emph{decoupled ansatz observable} is invariant with respect to eigenvector moment flow acting on the $\xx$ variable.  That is, $\bar{F}(\cdot;\yy)\in\ker(\gen(s))$ and to remind the reader of this time invariance, the $t$ subscript is dropped.
It will be necessary to evaluate the integral in \eqref{eq:ansatz_decouple}, for example in the proof of Proposition \ref{prop:l2}.  This can be done by Wick's theorem which computes the high moment as a sum over particle matchings of covariance products 
\begin{align}
\bar{F}(\xx,\yy) &= \frac{1}{\sqrt{\pi(\xx)}} \sum_{\sigma\in M_n \cap \Stab(\xx)} \prod_{a\in[n]/\sigma} \EE{ \ipr{\vg_{x_a}, \cov_{y_a}^{1/2} \vv_a } \ipr{\vg_{x_a}, \cov_{y_{\sigma(a)}}^{1/2}\vv_{\sigma(a)}}} \\
&=\frac{1}{\sqrt{\pi(\xx)}} \sum_{\sigma\in M_n \cap \Stab(\xx)} \prod_{a\in[n]/\sigma} \ipr{\cov_{y_a}^{1/2} \vv_a, \cov_{y_{\sigma(a)}}^{1/2}\vv_{\sigma(a)}} \\
&=\frac{1}{\sqrt{\pi(\xx)}} \sum_{\sigma\in M_n \cap \Stab(\xx)} \sqrt{\prod_{a=1}^n \ipr{\vv_a, \left(\cov_{y_a}^{1/2} \cov_{y_{\sigma(a)}}^{1/2}\right)\vv_{\sigma(a)}}}
\end{align}
where the products in the first two expressions are taken over particle pair representatives $a\in[n]$ so that exactly one label from each involution coset $\{b,\sigma(b)\} \subset [n]$, $b\in[n]$, is chosen.  This can be written in the slightly cleaner product notation appearing in the third expression because each factor appears exactly twice.

To avoid dealing with the matrix square roots appearing in the inner product factors, we make one final simplification to end the construction of our ansatz observable by replacing the matrix geometric mean with a matrix arithmetic mean
\begin{equation}
F_t(\xx; \yy)=\frac{1}{\sqrt{\pi(\xx)}} \sum_{\sigma\in M_n \cap \Stab(\xx)} \sqrt{\prod_{a\in[n]/\sigma} \ipr{\vv_a, \frac{1}{2}\left( \cov_{y_a} + \cov_{y_{\sigma(a)}} \right) \vv_{\sigma(a)}}}.
\end{equation}

Returning to our original focus of approximating short time behavior of the eigenvector moment observable $f_t$, keep in mind that these more general ansatz observables specialize to the base ansatz from \eqref{eq:ansatz_base} when the color profile variable $\xx$ and the location variable $\yy$ coincide
\begin{equation}\label{eq:ansatz}
F(\xx,\xx) = \bar{F}_t(\xx,\xx) = \bar{F}(\xx) \approx f_t(\xx)
\end{equation}
since $x_{\sigma(a)} = x_a$ for all $\sigma\in M_n\cap\Stab(\xx)$.

To prove Theorem \ref{thm:main}, for every $\yy\in\Ln$ which is supported on $\indint^\kappa$, we will consider a local cutoff of $f_{t_0}(\xx) - F(\xx,\yy)$ near $\yy$ at some initial time $t_0 < t$, which we denote $h_{t_0}(\xx;\yy)$.  This local observable will be driven forward to define $h_t(\xx;\yy)$, $t \geq t_0$, via slight variants of the generator $\gen(s)$ acting on the $\xx$ variable.  Using the short range properties of $\gen(s)$, we will see that for all such $\yy$ this local observable relaxes quickly in $L^2$, 
\begin{equation}
\|h_{t_0+t_1}(\cdot,\yy)\|_2^2 < N^{-c}|\supp(h_{t_0}(\cdot,\yy))|.
\end{equation}
Then using the long range properties of $\gen(s)$, we will find the ultracontractive bound
\begin{equation}
\|h_t(\cdot,\yy)\|_\infty^2 < \frac{N^\varepsilon}{(N(t-t_0-t_1))^{n/2}} \|h_{t_0+t_1}(\cdot;\yy)\|_2^2.
\end{equation}
The argument then concludes by carefully choosing times and cutoff parameters so that the making the final comparison at the terminal time $t$ holds
\begin{equation}
|f_t(\xx) - F_t(\xx,\xx)| \approx |h_t(\xx;\xx)| \leq \|h_t(\cdot;\xx)\|_\infty \leq N^{-\od}
\end{equation}
for some small constant $\od>0$ depending only on $n$.
\end{remark}

\subsection{Positivity properties}
Unlike the indistinguishable eigenvector moment flow, the generator $\gen_s$ does not admit a maximum principle despite both the move and exchange components components, $\sum_{i<j} c_{ij} \move_{ij}$ and $\sum_{i < j } c_{ij} \exch_{ij}$ respectively, satisfying one.  While $\gen_s$ lacks positivity in the $L^\infty$ sense, being a lift of the generator for the indistinguishable eigenvector moment flow suggests there might be positivity in some other sense.  This section summarizes the positivity results used throughout the paper.

In a sense, the end goal of this problem is to describe any of the derivatives appearing the four step derivation of the the eigenvector moment flow sufficiently precisely.  The high dimensionality, however, poses some difficulties.

\begin{definition}\label{def:inner_product}
From this point onward, $\Ln$ will refer to the discrete measure space with measure $\pi$ introduced in Theorem \ref{thm:levmf}.  In particular, we will be discussing $L^p$, $1 \leq p \leq \infty$, norms always taken with respect to the reversible measure $\pi$.  That is, the $L^p$ norm of a function $f : \Ln \rightarrow \RR$ is
\begin{equation}\label{eq:p-norm}
\|f\|_p = \left(\sum_{\xx\in\Ln}\pi(\xx)|f(\xx)|^p\right)^{1/p}
\end{equation}
and $f\in L^p(\Ln)$ when this sum converges (which is always the case since $\Ln$ is finite).
For any $f,g \in L^2(\Ln)$, denote their inner product by
\begin{equation}
\ipn{ f, g } = \sum_{\xx \in \Ln} \pi(\xx) f(\xx) g(\xx).
\end{equation}
For any $\xx \in \Ln$, let $\delta_\xx$ be the unique function that satisfies $\ipn{ \delta_\xx, f }= f(\xx)$ for all $f \in L^\infty(\Ln)$.  This leads to
\begin{equation}
\delta_\xx(\yy) = \begin{cases} \pi(\xx)^{-1} &\mbox{if } \yy = \xx \\ 0 &\mbox{else.} \end{cases}
\end{equation}
Having specified a coefficient trajectory $c_{ij}(s) \geq 0$, for all times $s \geq 0$, and start/stop times $0 \leq s_1 \leq s_2$, let $\semi(s_1, s_2) = \semi(s_1, s_2; \{c_{ij}(s) | 1 \leq i < j \leq N, s \geq 0\})$ denote the \emph{propagator} or \emph{transition semigroup} $\semi(s_1, s_2) : L^2(\Ln) \rightarrow L^2(\Ln)$ defined by
\begin{equation}\label{eq:semigroup}
\semi(s_1, s_2) f(\xx) = h_{s_2}(\xx) \quad \mbox{where} \quad h_{s_1}(\xx) = f(\xx) \quad \mbox{and} \quad \partial_s h_s(\xx) = \gen_s h_s(\xx)
\end{equation}
for all $s \in (s_1, s_2)$ and $\xx \in \Ln$.
\end{definition}

\begin{lemma}[$L^1 \rightarrow L^1$ boundedness]\label{lem:L1}
For any coefficient trajectory $c_{ij}(s)=c_{ji}(s) \geq 0$, $1 \leq i \neq j \leq N$, $s \geq 0$ and for any $0 \leq s_1 \leq s_2$, the $L^1 \rightarrow L^1$ operator norm is bounded independent of $N$
\begin{equation}
\|\semi(s_1, s_2)\|_{1,1} \leq n!!.
\end{equation}
\end{lemma}

\begin{proof}
The idea behind this proof is to find a probabilistic representation of $\semi(s_1, s_2)$ applied to $\delta$-functions.  The approach to this is by reverse engineering the four step derivation of eigenvector moment flow outlined above.

Begin by fixing $\EV = (u_i^\alpha)_{i,\alpha=1}^{N} \in SO(N)$ an orthogonal matrix and $\Vecs = (\vv_1, \ldots \vv_n)$ with $\vv_a = (v_a^\alpha)_{\alpha=1}^N \in \RR^N$, $\|\vv_a\|_2$ a collection of unit vectors in $\RR^n$. 

We now describe an analogue of the stochastic process appearing in Step 2 with arbitrary coefficients, \eqref{eq:edbm}, describing eigenvector Dyson Brownian motion.  This process must then inherit the generator described in Step 3 \eqref{eq:gen_edbm} with arbitrary nonnegative coefficients $c_{ij}$.
Consider the stochastic process which is the scaling limit of the translation invariant (but not time invariant) random walk on the Lie group $SO(N)$ with rate $\sqrt{c_{ij}(s)}$ in the $\soB_{ij}$ direction.   Define the process explicitly by 
\begin{equation}\label{eq:lie_rw}
u_i^\alpha(s) = \sum_{k=1}^N u_k^\alpha \texp\left(\int_{s_1}^s \sqrt{c(s')} \odot d\fh(s')\right)_{ki}
\end{equation}
where $\texp(\int \cdot)$ is the time ordered exponential.  Here $\fh = (\fh_{ij})_{i,j=1}^N$ is a Lie algrebra valued Brownian motion.   In particular, $\fh_{ij} = -\fh_{ji}$ are independent and identically distributed standard Brownian motions for $1 \leq i < j \leq N$.  The $\mathfrak{so}(N)$-valued differential $\sqrt{c(s)} \odot d\fh(s)$ is defined by 
\begin{equation}
\left(\sqrt{c(s)} \odot d\fh(s)\right)_{ij} = \begin{cases} \sqrt{c_{ij}(s)} d\fh_{ij}(s) & \mbox{if } 1 \leq i < j \leq N \\
0 &\mbox{if } 1 \leq i = j \leq N \\
-\sqrt{c_{ij}(s)} d\fh_{ji}(s) & \mbox{if } 1 \leq j < i \leq N \end{cases}
\end{equation}
for all $i,j \in [N]$.

It\^o's lemma applied to the exponential map \eqref{eq:lie_rw} gives
\begin{equation}\label{eq:ito_texp}
d\EV(s) = \EV(s) \left( \sqrt{c(s)} \odot d\fh(s) + \frac{1}{2} (\sqrt{c(s)} \odot d\fh(s))^2 \right)
\end{equation}
where the square in the second term includes a matrix product.  In particular,
\begin{equation}
[(\sqrt{c(s)} \odot d\fh(s))^2]_{ij} = \sum_{k \neq i,j} \left(\sqrt{c_{ik}(s)} d\fh_{ik}(s) \right)\left( - \sqrt{c_{kj}(s)} d\fh_{kj}(s)  \right) = -\delta_{ij} \sum_{k \neq i} c_{ik}(s) dt
\end{equation}
is a diagonal matrix.  Expanding the expression entry-wise, \eqref{eq:ito_texp} becomes
\begin{equation}
d u^\alpha_i(s) = \sum_{k \neq i} u^\alpha_k(s) \sqrt{c_{ki}(s)} d\fh_{ki}(s) - u^\alpha_i(s) \sum_{k \neq i} c_{ik}(s) dt.
\end{equation}
Note that this expression exactly matches Step 2, \eqref{eq:edbm}, in distribution after specializing coefficients to $c_{ij}(s) = N^{-1} \left(\lambda_i(s) - \lambda_j(s)\right)^2$.  By the same computations that lead us from Step 2 to Steps 3 and 4, we learn that the generator for $\EV$ is given by $\sum_{1 \leq i < j \leq N} c_{ij}(s)\soB_{ij}^2$ and that $\soB_{ij}^2 P(\xx,\EV,\Vecs) = \gen_{ij} P(\xx,\EV,\Vecs)$ where $P$ is the polynomial defined in \eqref{eq:test_poly}.  In particular, we have just shown that for every $\xx \in \Ln$, $\EV \in SO(N)$, and $\Vecs \in \RR^{N \times n}$, we have
\begin{equation}\label{eq:stoch_propagator}
\semi(s_1, s_2)P(\xx,\EV,\Vecs) = \EE{P(\xx,\EV(s_2),\Vecs)}
\end{equation}
where the expectation is taken over the stochastic process $\EV(t)$ as defined in \eqref{eq:lie_rw} or equivalently in \eqref{eq:ito_texp}.  Taking an $L^1$ norm gives
\begin{equation}
\|P(\cdot,\EV,\Vecs)\|_1 = \sum_{\xx \in \Ln} \left( \prod_{i=1}^N n_i(\xx)!! \right) \left| \prod_{a=1}^n \sum_{\alpha=1}^N v_a^\alpha u^\alpha_{x_a} \right|.
\end{equation}
Now we rearrange the sum, grouping all configurations $\xx$ which share a common perfect matching stabilizing $\xx$.  There are exactly $\prod_{i=1}^N n_i(\xx)!!$ such perfect matchings for each $\xx \in \Ln$.  The sum over $\xx$ now becomes a double sum -- first over the perfect matching $\sigma \in M_n$, then over the site $i$ at which the matched pair $\{a, \sigma(a)\}$ is positioned.  This second sum is done independently for all matched pairs $\{a, \sigma(a)\}$ ranging over all orbit representatives $a \in [n] / \sigma$.  Therefore,
\begin{equation}
\|P(\cdot, \EV, \Vecs)\|_1 = \sum_{\sigma \in M_n} \prod_{a \in [n]/\sigma}  \left( \sum_{\alpha=1}^N \sum_{i = 1}^N |v_a^\alpha u^\alpha_i | |v_{\sigma(a)}^\alpha u^\alpha_i| \right).
\end{equation}
To bound the inner sum, use the fact that $\EV \in SO(N)$ is an orthogonal matrix and that the test vectors are $L^2$ normalized: $\|\vv_a\|=1$ for all $a \in [n]$.
\begin{equation}
\sum_{\alpha=1}^N \sum_{i =1}^N |v_a^\alpha u^\alpha_i v_{\sigma a}^\alpha u^\alpha_i| = \sum_{\alpha=1}^N |v_a^\alpha v_{\sigma a}^\alpha| \leq \sum_{\alpha=1}^N \frac{1}{2} \left( |v_a^\alpha|^2 + |v_{\sigma a}^\alpha|^2 \right) = 1
\end{equation}
where the the inequality is Schwarz or AM-GM.  This proves that 
\begin{equation}\label{eq:deterministic_testpoly}
\|P(\cdot,\EV,\Vecs)\|_1 \leq |M_n| = n!!
\end{equation}
for every $\EV \in SO(N)$ and $\Vecs \in \RR^{n \times N}$, $\|v_a\|_2 = 1$.  Since $\sqrt{c(s)} \odot d\fh(s)$ is an $\mathfrak{so}(N)$-valued differential, the stochastic process $\EV(t) \in SO(N)$ remains in the Lie group almost surely.  Hence,
\begin{equation}\label{eq:jensen_poly}
\|\semi(s_1, s_2) P(\cdot, \EV, \Vecs)\|_1 = \|\EE{P(\cdot, \EV(s_2), \Vecs)}\|_1 \leq \EE{\|P(\cdot, \EV(s_2), \Vecs)\|_1} \leq n!!
\end{equation}
where the first inequality is Jensen since the $L^1$ norm is convex and the second inequality is \eqref{eq:deterministic_testpoly}.  To conclude the proof, simply note that every $\delta$-function $\delta_\xx$, $\xx \in \Ln$, can be written as a scalar multiple of some $P(\cdot, \EV, \Vecs)$:
\begin{equation}\label{eq:d_as_P}
\delta_\xx(\yy) = \left(\prod_{i=1}^N n_i(\xx)!!\right)^{-1} P(\yy, \iden_N, \Vecs^\xx)
\end{equation}
where $\iden_N$ is the $N \times N$ identity matrix and $\Vecs^\xx$ is the collection of standard basis vectors specified by the distinguishable particle configuration $\xx$: $v^\xx_a = \ve_{x_a}$ for all $a \in [n]$.  From \eqref{eq:jensen_poly} and \eqref{eq:d_as_P} along with the facts that for all $\xx \in \Ln$, $\prod_{i=1}^N n_i(\xx)!! \geq 1$, we get
\begin{equation}\label{eq:d_11}
\|\semi(s_1,s_2)\delta_\xx\|_1 \leq n!!
\end{equation}
for every $\delta$-function.  This $L^1$ bound on $\delta$-functions translates to a bound on the $L^1 \rightarrow L^1$ operator norm by recovering the $L^1$ norm through the triangle inequality.  Proceed by expanding any $f \in L^1(\Ln)$ according to its representation as a linear combination of $\delta$-functions: $f = \sum_{\xx \in \Ln} f(\xx) \pi(\xx) \delta_\xx$.
\begin{multline}
\|f\|_1 = \left\|\sum_{\xx \in \Ln} \pi(\xx) f(\xx) \semi(s_1,s_2)\delta_\xx\right\|_1 \leq \sum_{\xx \in \Ln} \pi(\xx) |f(\xx)| \|\semi(s_1,s_2)\delta_\xx\|_1 \leq \|f\|_1 \sup_{\xx \in \Ln} \|\semi(s_1,s_2)\delta_\xx\|_1 \\
\leq n!! \|f\|_1
\end{multline}  
where the last inequality is \eqref{eq:d_11}.
\end{proof}

\begin{remark}
The idea behind this proof is that all dynamics generated by linear combinations of $\gen_{ij}$, $1 \leq i < j \leq N$, on the configuration space $\Ln$ admit a description from an underlying eigenvector evolution.  Compactness (and more specifically, $L^2$ boundedness) of the unit sphere $S^{N-1}$, or more generally the special orthogonal group $SO(N)$, translates to $L^1$ boundedness for the renormalized random walk on the configuration space.
\end{remark}

\begin{lemma}[Negative semidefinite]\label{lem:definite}
For all $i \neq j \in \ZZ$, $\move_{ij} \leq \exch_{ij} \leq 0$ in the sense that for all test functions $f\in L^2(\Ln)$,
\begin{equation}
\ipn{ f, (-\exch_{ij}) f} \geq 0 \quad\mbox{and}\quad 
\ipn{ f, (\exch_{ij}-\move_{ij}) f} \geq 0
\end{equation}
for all $i \neq j \in [N]$.
\end{lemma}

\begin{proof}
Both $\move_{ij}$ and $\exch_{ij}$ satisfy the maximum principle and hence are negative semidefinite.  Let $f$ be an eigenfunction of $\move_{ij} - \exch_{ij}$ with eigenvalue $\lambda$.  By Lemma \ref{lem:L1}, $n!! \|f\|_1 \geq \| e^{t(\move_{ij} - \exch_{ij})} f \|_1 = e^{t \lambda} \|f\|_1$ is bounded independent of $t$, so we must have  $\lambda \leq 0$ and $\move_{ij}-\exch_{ij}$ must also be negative semidefinite.
\end{proof}

\begin{lemma}[Kernel Projection]\label{lem:kernel}
Let $\kproj = \kproj_N : L^2(\Ln) \rightarrow L^2(\Ln)$ be the orthogonal projection onto the \emph{global kernel}, $\cap_{1 \leq i<j \leq N} \ker(\gen_{ij})$.  Then for any $\xx, \yy \in \Ln$,
\begin{equation}
\ipn{ \delta_\xx, \kproj \delta_\yy } = \frac{1}{\sqrt{\pi(\xx)\pi(\yy)}} \EE{\prod_{a=1}^n O_{x_a y_a}}
\end{equation}
where $\orth = (O_{ij})_{i,j=1}^N \in SO(N)$ is a Haar distributed random matrix.
\end{lemma}

\begin{proof}
Borrowing the $\delta$-function representation from \eqref{eq:d_as_P} and the classifying property of $\delta$-functions from Definition \ref{def:inner_product}, the left hand side can be written as
\begin{equation}
\langle \delta_\xx, \kproj \delta_\yy \rangle = \pi(\yy)^{-1/2} \kproj P(\xx, \iden_N, \Vecs^\yy).
\end{equation}
Running the dynamics generated by $\sum_{1 \leq i < j \leq N} \gen_{ij}$ for long times, all components orthogonal to the kernel become arbitrarily small.  Therefore, the orthogonal projection onto the kernel has the representation $\kproj = \lim_{s \rightarrow \infty} e^{s\sum_{1 \leq i < j \leq N} \gen_{ij}}$.  Moreover, as $\gen_{ij}$ are symmetric negative semidefinite, $\ker(\sum_{1 \leq i < j \leq N} \gen_{ij}) = \cap_{1 \leq i < j \leq N} \ker(\gen_{ij})$. Now we can borrow the stochastic process interpretation of the propagator from \eqref{eq:stoch_propagator}
\begin{equation}
\langle \delta_\xx, \kproj \delta_\yy \rangle = \pi(\yy)^{-1/2} \lim_{s \rightarrow \infty} e^{s\sum_{1 \leq i < j \leq N} \gen_{ij}} P(\xx, \iden_N, \Vecs^\yy) = \pi(\yy)^{-1/2} \lim_{s \rightarrow \infty} \EE{P(\xx, \EV(s), \Vecs^\yy)}
\end{equation}
where $\EV(s) = \texp(\int_0^s d\fh)$ is a standard Brownian motion on $SO(N)$ generated by the classical Laplace-Beltrami operator on $SO(N)$.  Here $\fh$ is a standard Brownian motion on $\mathfrak{so}(N)$ as described in the proof of Lemma \ref{lem:L1}.  Recalling the definition of $P(\xx,\EV,\Vecs)$ from \eqref{eq:test_poly} and the choice $\Vecs^\yy = (\ve_{y_1}, \ldots, \ve_{y_n})$, the right hand side becomes
\begin{equation}
\pi(\yy)^{-1/2} \lim_{s \rightarrow \infty} \EE{P(\xx, \EV(s), \Vecs^\yy)} = \frac{1}{\sqrt{\pi(\xx)\pi(\yy)}} \lim_{s \rightarrow \infty} \EE{ \prod_{a=1}^n \langle \vu_{x_a}(s), \ve_{y_a} \rangle } = \frac{1}{\sqrt{\pi(\xx)\pi(\yy)}} \EE{ \prod_{a=1}^n O_{x_a y_a} }
\end{equation}
where the last equality is a consequence of the limiting distribution of $\EV(s)$ being Haar measure on $SO(N)$.
\end{proof}

\begin{remark}
For explicit computations for such Haar integrals of matrix entries involving Weingarten functions, see \cite{collins2003moments}.  For our purposes, an elementary $L^\infty$ bound will suffice.  For convenience, we use the subgaussian techniques from \cite{vershynin2018high} for a quick proof.
\end{remark}

\begin{lemma}\label{lem:subgaussian}
There exists a universal constant $\sgconst>0$ such that for all $\xx,\yy\in\Ln$,
\begin{equation}
\left|\EE{\prod_{a=1}^n O_{x_ay_a}}\right| \leq \sgconst^n n^{n/2} N^{-n/2}
\end{equation}
uniformly in particle number $n$ and length of the configuration space $N$.
\end{lemma}

\begin{proof}
Bound the lefthand side by Jensen's inequality and AM-GM.
\begin{equation}
\left|\EE{\prod_{a=1}^n O_{x_ay_a}}\right| \leq \frac{1}{n}\sum_{a=1}^n \EE{|O_{x_ay_a}|^n}
\end{equation}
Note that although not independent, for all $i,j\in[N]$ the marginal distribution of $O_{ij}$ matches that of the first coordinate from the uniform spherical distribution on $S^N$.  By \cite[Proposition 3.4.6]{vershynin2018high}, $\sqrt{N}O_{ij}$ is subgaussian with Olicz 2-norm $\|\sqrt{N}O_{ij}\|_{\psi_2} \leq \sgconst'$ bounded by a universal constant $\sgconst'>0$ independent of $N$.  Then \cite[Proposition 2.5.2]{vershynin2018high} says that 
\begin{equation}
\EE{|\sqrt{N}O_{ij}|^n}^{1/n} \leq \sgconst \sqrt{n}
\end{equation}
for another universal constant $\sgconst>0$ independent of $N$ and $n$.
\end{proof}

\begin{corollary}\label{cor:kproj_bound}
The global kernel projection operator on $L^2(\Ln)$ satisfies the following strong $L^1 \rightarrow L^\infty$ operator norm bound
\begin{equation}
\|\kproj\|_{1,\infty} \leq \sgconst^n n^{n/2} N^{-n/2}
\end{equation}
where $\sgconst$ is the universal constant from Lemma \ref{lem:subgaussian}.
\end{corollary}

\begin{proof}
Suppose $f\in L^1(\Ln)$.  Then for all $\xx\in\Ln$,
\begin{equation}
|\kproj f(\xx)| = \left| \sum_{\yy\in\Ln} \pi(\yy)f(\yy)\ipn{\delta_\xx, \kproj \delta_\yy} \right| \leq \|f\|_1 \sup_{\xx,\yy\in\Ln} \frac{|\EE{\prod_{a=1}^n O_{x_ay_a}}|}{\sqrt{\pi(\xx)\pi(\yy)}} \leq \sgconst^n n^{n/2} N^{-n/2} \|f\|_1
\end{equation}
by Lemma \ref{lem:subgaussian} and the fact that $\pi(\xx) \geq 1$ for all $\xx\in\Ln$.
\end{proof}

\begin{remark}
In addition to bounds on kernel projection entries, it will be helpful to have some exact algebraic relations for functions within the kernel.
\end{remark}

\begin{definition}\label{defn:actN}
Let $S_N$ be the group of permutations of $[N]$.  Just as $S_n$ acts on $\Ln$ by permuting labels, there is a similar permutation action, denoted by $\actN$, of $S_N$ on $\Ln$ by permuting sites.  For every $\tau \in S_N$, $\xx \in \Ln$, and $a \in [n]$ the action satisfies
\begin{equation}
\tau \actN \xx = (\tau(x_1), \ldots, \tau(x_n))^\top.
\end{equation}
\end{definition}

\begin{corollary}[Spatial invariance in the kernel]\label{cor:spatial_kernel}
For any site permutation $\tau \in S_N$ and any function in the global kernel $f \in \cap_{1 \leq i < j \leq N} \ker(\gen_{ij})$, we have $f(\tau \actN \xx) = f(\xx)$ for every $\xx \in \Ln$.
\end{corollary}

\begin{proof}
Let $\lperm^\tau = (a^\tau_{ij})_{i,j=1}^N \in SO(N)$ be the orthogonal matrix with entries $a^\tau_{ij} = \one{j = \tau(i)}$.  This is the (left) permutation matrix corresponding to $\tau$ in the sense that for any $N \times N$ matrix $\dat=(h_{ij})_{i,j=1}^N$, we have $[\lperm^\tau \dat]_{ij} = h_{\tau(i)j}$.  I claim that 
\begin{equation}\label{eq:kernel_permute}
\kproj \delta_\xx = \kproj \delta_{\tau \actN \xx}
\end{equation}
for every $\xx \in \Ln$.  To see this, appeal to the Lemma \ref{lem:kernel} and use duality.  Indeed, when paired against any other $\delta$-function, $\delta_\yy$, $\yy \in \Ln$, we get
\begin{equation}
\langle \kproj \delta_{\tau \actN \xx}, \delta_\yy \rangle = \EE{\frac{\prod_{a=1}^n O_{\tau(x_a) y_a}}{\sqrt{\pi(\xx)\pi(\yy)}}} =\EE{\frac{\prod_{a=1}^n [\lperm^\tau \orth]_{x_a y_a}}{\sqrt{\pi(\xx)\pi(\yy)}}} = \EE{\frac{\prod_{a=1}^n O_{x_a y_a}}{\sqrt{\pi(\xx)\pi(\yy)}}} = \langle \kproj \delta_\xx, \delta_\yy \rangle
\end{equation}
where the third inequality holds because Haar measure is left-translation invariant.  Equipped with \eqref{eq:kernel_permute}, we can directly compute
\begin{equation}
f(\xx) = \langle \delta_\xx, f \rangle = \langle \delta_\xx, \kproj f \rangle = \langle \kproj \delta_\xx, f \rangle = \langle \kproj \delta_{\tau \actN \xx}, f \rangle = \langle \delta_{\tau \actN \xx}, \kproj f \rangle = \langle \delta_{\tau \actN \xx}, f \rangle = f(\tau \actN \xx).
\end{equation}
since $f \in \cap_{1 \leq i < j \leq N} \ker(\gen_{ij})$ if and only if $f = \kproj f$.
\end{proof}

The previous result gave a weak upper bound for the global kernel of the generator for the distinguishable eigenvector moment flow.  This will be used in the last step of the proof of Proposition \ref{prop:FSP} as well as in a path counting argument in Proposition \ref{prop:poincare}.  The following result provides a lower bound for the global kernel.  This will be used in Proposition \ref{prop:l2} through knowing that the ansatz is time invariant.

\begin{lemma}[Kernel classification] \label{lem:globalker}
For any  perfect matching $\sigma \in M_n$, define the \emph{stratum indicator} $\chi_\sigma : \Ln \rightarrow \RR$ by
\begin{equation}
\chi_\sigma(\xx) = \frac{\one{\sigma \actn \xx = \xx}}{\sum_{\sigma' \in M_n} \one{\sigma' \actn \xx = \xx}} = \frac{\one{\sigma \actn \xx = \xx}}{\sqrt{\pi(\xx)}}
\end{equation}
Then the global kernel $\cap_{1 \leq i < j \leq N} \ker(\gen_{ij})$ contains the $n!!$ dimensional subspace of $L^2(\Ln)$ spanned by the eigenbasis $\chi_\sigma$, $\sigma \in M_n$.  That is
\begin{equation}
\gen_{ij} \chi_\sigma = 0
\end{equation}
for all $1 \leq i < j \leq N$, $\sigma \in M_n$.
\end{lemma}

\begin{proof}
For any $\sigma \in M_n$, $\xx \in \Ln$, and $i \neq j \in [N]$, there are three possibilities.  Either
\begin{itemize}
\item $\sigma \actn \xx = \xx$ in which case 
\begin{equation}
(\move_{ij} \chi_\sigma)(\xx) = (\exch_{ij}\chi_\sigma)(\xx) = -2n_i(\xx)n_j(\xx)\pi(\xx)^{-1/2}
\end{equation}
\item $\sigma \actn \xx \neq \xx$ but there exists indices $a \neq b \in [n]$ such that the matching $\sigma$ satisfies and there exists $a \neq b \in [n]$ such that $x_a=x_b=i$, $x_{\sigma(a)} = x_{\sigma(b)} = j$ and $x_{\sigma(c)} = x_c$ for all $c \in [n] \backslash \{a,b,\sigma(a), \sigma(b)\}$, in which case
\begin{equation}
(\move_{ij} \chi_\sigma ) (\xx) = ( \exch_{ij} \chi_\sigma )(\xx) = 4\pi(\xx)^{-1/2}
\end{equation}
\item $\sigma \actn \xx \neq \xx$ and $\chi_\sigma$ vanishes on $\xx$ and every configuration formed by any two particle jump or swap originating at $\xx$.  In this case, we have  
\begin{equation}
(\move_{ij} \chi_\sigma)(\xx) = (\exch_{ij} \chi_\sigma)(\xx) = 0.
\end{equation}
\end{itemize}
In all three cases, $(\gen_{ij} \chi_\sigma)(\xx) = (\move_{ij} \chi_\sigma)(\xx) - (\exch_{ij}\chi_\sigma)(\xx) = 0$. This implies that 
\begin{equation}
\vspan \{\chi_\sigma | \sigma \in M_n\} \subset \cap_{ij} \ker(\gen_{ij}).
\end{equation}
\end{proof}

\begin{remark}
In fact, the converse containment holds as well so the global kernel is precisely the span of all $\chi_\sigma$, $\sigma \in M_n$.  Is not needed throughout the paper so we do not provide the proof here.  It can however be deduced as an immediate consequence of the Poincar\'e inequality, Proposition \ref{prop:poincare}.  For completeness, a proof is included in Section \ref{sec:energy} after the Nash inequality where some convenient notation is introduced.  See Corollary \ref{cor:globalker_complete} for the complete proof.
\end{remark}

\begin{corollary}\label{cor:ansatzker}
The ansatz observable is time invariant with respect to colored eigenvector moment flow
\begin{equation}
\gen_s F_t(\xx; \yy) = 0
\end{equation}
for all $\xx, \yy \in \Ln$ and $s \geq 0$.
\end{corollary}

\begin{proof}
Appealing to Lemma \ref{lem:globalker}, $F_t$ is a linear combination of $\chi_\sigma$, $\sigma \in M_n$.  Indeed
\begin{equation}
F_t(\xx; \yy) = \sum_{\sigma \in M_n} \chi_\sigma(\xx) \left( \one{\sigma \actn \yy = \yy}\prod_{a=1}^n \sqrt{\ipr{ \vv_a, \frac{\Im \grn_{\fc,t}(\gamma_{y_a}(t)+it))}{\Im m_{\fc,t}(\gamma_{y_a}(t)+it)}  \vv_{\sigma(a)} }} \right).
\end{equation}
\end{proof}

We conclude this section by providing a simple yet convenient algebraic representations for the quadratic form induced by the colored eigenvector moment flow dynamics which will be useful in Sections \ref{sec:l2relax} and \ref{sec:energy}.
\begin{lemma}[Integration by-parts]\label{lem:ibp}
For any $f, g \in L^2(\Ln)$ and any $s \geq 0$,
\begin{equation}
\ipn{ f, \gen_s g } = - \frac{1}{2} \sum_{\xx \neq \yy} \pi(\xx) \pi(\yy) \gen_{\xx\yy}(s) \left( f(\xx) - f(\yy) \right) \left( g(\xx) - g(\yy) \right)
\end{equation}
where $\gen_{\xx\yy}(s) = \ipn{ \delta_\xx, \gen_s \delta_\yy }$ is symmetric.
\end{lemma}

\begin{proof}
Constant functions are in the kernel of $\gen_s$.  For instance, the vector of all ones is the sum of all stratum indicators: $\sum_{\xx \in \Ln} \pi(\xx) \delta_\xx = \sum_{\sigma \in M_n} \chi_\sigma \in \ker(\gen_s)$.  In particular, 
\begin{equation}
0 = \ipn{ \delta_\xx, \gen_s \sum_{\yy \in \Ln} \pi(\yy) \delta_\yy } = \sum_{\yy \in \Ln} \pi(\yy) \gen_{\xx\yy}.
\end{equation}
Replace the diagonal terms in the left hand side inner product with the relation $\pi(\xx) \gen_{\xx\xx} = -\sum_{\yy \neq \xx} \pi(\yy)\gen_{\xx\yy}$ to get
\begin{equation}
\ipn{ f, \gen_s g } = \sum_{\xx,\yy \in \Ln} \pi(\xx) \pi(\yy) \gen_{\xx\yy} f(\xx) g(\yy) = \sum_{\xx \neq \yy} \pi(\xx) \pi(\yy) \gen_{\xx\yy} f(\xx) \left(g(\yy) - g(\xx) \right).
\end{equation}
Similirly, swapping the roles of $\xx$ and $\yy$ gives
\begin{equation}
\ipn{ f, \gen_s g } = \sum_{\xx \neq \yy} \pi(\xx) \pi(\yy) \gen_{\xx\yy} f(\yy) \left(g(\xx) - g(\yy)\right)
\end{equation}
by reversibility.  Taking the average of the previous two identities gives the desired result.
\end{proof}

\begin{definition}
The Dirichlet form is a positive semidefinite quadratic form given by
\begin{equation}\label{def:dirichlet}
\dir_s(f) = \ipn{ f, (-\gen_s) f } = \frac{1}{2} \sum_{\xx \neq \yy} \pi(\xx) \pi(\yy) \gen_{\xx\yy}(s) |f(\xx) - f(\yy)|^2
\end{equation}
and describes the change in $L^2$ norm for the colored eigenvector moment observable, 
\begin{equation}\label{eq:dir_derivative}
\partial_s \|f_s\|_2^2 = \partial_s \|f_s - \kproj f_s\|_2^2 = -2 \dir_s(f_s)
\end{equation}
for all $s > 0$.
\end{definition}

\subsection{Particle configurations --- distinguishable, indistinguishable, and everything in between}\label{sec:colorlattice}

The \emph{colorblind map} $\forget : \Ln \rightarrow \On$ defined in Definition \ref{def:distinguishable} is compatible with the eigenvector moment flow dynamics.  To make this precise, define the functional pullback and pushforward operators by
\begin{equation}\label{eq:pushpull}
\forget^* f(\xx) = f(\forget(\xx)) \quad \mbox{and} \quad \forget_* g(\bet) = \pi(\forget^{-1}(\bet))^{-1} \sum_{\xx \in \forget^{-1}(\bet)} \pi(\xx) g(\xx)
\end{equation}
respectively for all $f \in L^2(\On)$ and $g \in L^2(\Ln)$.  It can be shown that $\gen_s \forget^* f  = \forget^* \oldgen_s f$ and $\forget_* \gen_s g = \oldgen_s \forget_* g$ where $\oldgen_s$ is the generator from \cite[Theorem 3.1]{QUE}.  In fact, a special case of Lemma \ref{lem:comm_genep} is that $[\gen_s, \forget^*\forget_*] = 0$.  Furthermore, the reversible measure of $\oldgen_s$ also introduced in \cite[Theorem 3.1]{QUE} is, up to a constant factor, the pushforward measure $\forget_* \pi$.  See Appendix \ref{app:rev_meas} for details.

These observations may be extended to a lattice, coined the \emph{coloring lattice}, where only specified particles become indistinguishable from one another corresponding to some equivalence relations. For the sake of this generalization, reinterperet the indistinguishable particle configuration space as the quotient $\On = \Ln / S_n$ of the distinguishable particle configuration space by the $S_n$ action, $\actn$.  Moreover, the colorblind map can be interpereted as the quotient map $\forget : \Ln \rightarrow \Ln / S_n$.

\begin{definition}[Partition, lattice, and groups]\label{defn:partition}
A \emph{partition of $[n]$} is a covering set of disjoint subsets of $[n]$.  That is, $\pP = \{P_1, \ldots, P_m\}$ for some $1 \leq m \leq n$, where $P_i \subset [n]$ for each $i \in [m]$, is a partition of $[n]$ if and only if $P_i \cap P_j = \varnothing$ when $i \neq j \in [m]$ and $\cup_{i=1}^m P_i = [n]$.  Endow the set of partitions of $[n]$ with a lattice structure ordered by refinement: $\pP \leq \pQ$ if and only if for all $P \in \pP$ there exists $Q \in \pQ$ such that $P \subset Q$.  Say that two labels $a,b \in [n]$ belong to the same $\pP$-part and write $a \eqp b$ if there exists $i \in [m]$ such that $a \in P_i$ and $b \in P_i$.  Lastly, for any partition $\pP$, let $\symP = \{\sigma \in S_n | \sigma \actn a \eqp a \mbox{ for all } a \in [n]\}$ denote the set of permutations \emph{compatible} with $\pP$.
\end{definition}

The colored configuration space $(\Lambda^\pP, \pi_\pP, \gen_s^\pP) = (\Ln, \pi, \gen_s) / \sym_\pP$ is the measure space which can now be interpereted as the space of particle configurations where two particles $a,b \in [n]$ are indistinguishable if and only if they have the same color $a \eqp b$.  The quotient map $\forget_\pP : \Ln \rightarrow \Lambda^\pP = \Ln / \symP$ remembers particle numbers of each color at every site, but not their original labels.  The colored measure and dynamics are given by the pushforward measure $\pi_\pP = (\forget_\pP)_* \pi$ and the functional pullback and pushforward from \eqref{eq:pushpull} $\gen_s^\pP = (\forget_\pP)_* \gen_s \forget_\pP^*$ of the operator $\gen_s$, respectively.

The indistinguishable particle configuration space can now be thought of as the colored particle configuration space with coloring $\pP = \{[n]\}$ for all $a \in [n]$ whereas the indistinguishable particle configuration space is identified with the colored configuration space with coloring $[\{n\}]$.

The use of the intermediate colored particle configuration spaces will play a central role in our proof of the Poincar\'e inequality.

\section{$L^2$ decay}\label{sec:l2}

In this section we prove an averaged form of convergence taking the form of fast decay in $L^2$ deviation to equilibrium for a local neighborhood.

For this section, fix time $t$ from Theorem \ref{thm:main} and an initial time $t/2 < t_0 < t$ which hence satisfies the conditions of the results from Section \ref{sec:fc}.
\begin{definition}\label{defn:short_range}
Fix $N$-dependent length scales $1 \ll \ell \ll K \ll Nt$, to be specified later.  Consider the \emph{averaging operator} on scale $K$ centered at $\yy\in\Ln$, $\Av(K, \yy)$.  The averaging operator serves as a mollified indicator for the $K$-neighborhood of $\yy \in \Lambda^n$.  It is a diagonal operator given by
\begin{equation}
\Av(K, \yy) f(\xx) = \Av(\xx; K,\yy) f(\xx) \quad \mbox{where} \quad \Av(\xx; K, \yy) = \frac{1}{K} \sum_{\alpha = K}^{2K-1} \one{\|\xx-\yy\|_1 < \alpha}
\end{equation}
The $L^1$ difference between particle configurations appearing in the definition of the coefficients $\Av(\xx;K,\yy)$ is taken to be $\|\xx-\yy\|_1 = \sum_{a=1}^n |x_a - y_a|$.  Consider the \emph{local short range cutoff coefficients} on scale $\ell$ defined by
\begin{equation}
\cshort_{ij}(s) = \cshort_{ij}(s; \ell) = \begin{cases} \creg_{ij}(s) &\mbox{if } i,j\in\indint^{\kappa/10} \mbox{ and } |i-j| \leq \ell \\ 0 &\mbox{otherwise} \end{cases}
\end{equation}
where $\creg_{ij}(s)$ are the coefficients defined in Theorem \ref{thm:levmf}.  Define the \emph{short range generator} on scale $\ell$ by 
\begin{equation}
\short(s) = \short(s;\ell) = \sum_{i<j} \cshort_{ij}(s; \ell)\gen_{ij}
\end{equation}
where $\gen_{ij}=\move_{ij}-\exch_{ij}$ is from Theorem \ref{thm:levmf}.  For all times $s \geq t_0$, the \emph{short range observable} $g_s(\xx;\ell,K,\yy)$ centered at $\yy$, initialized on scale $K$, and evolving on scale $\ell$ is defined as the unique solution to the partial differential equation initialized at the locally mollified difference between the moment observable and the ansatz observable and driven forward dyanamically by the short range generator
\begin{equation}
\begin{cases}
g_{t_0}(\xx; \ell,K,\yy) = \Av(\xx;K,\yy) \left( f_{t_0}(\xx) - F_t(\xx; \yy) \right) \\
\partial_s g_s(\xx; \ell,K,\yy) = \short(s;\ell) g_s(\xx; \ell,K,\yy)
\end{cases}
\end{equation}
for all $s \geq t_0$.  Here $f_{t_0}$ is the colore eigenvector moment observable introduced in Theorem \ref{thm:levmf}.  Lastly, let $\semi_\short(s_1, s_2) = \semi_\short(s_1,s_2;\ell)$ be the transition semigroup associated with the short range generator on length scale $\ell$.  That is,
\begin{equation}
\partial_s \semi_\short(s_1,s;\ell) f = \short(s;\ell) \semi_\short(s_1,s;\ell) f
\end{equation}
for every $f \in L^2(\Ln)$ and times $t_0 \leq s_1 < s$.
\end{definition}

In this section, it is shown that for $\yy \in \Ln$ with sites supported on $\indint^\kappa$, $y_a\in\indint^\kappa$ for all $a\in[n]$, the short range observable $g_s(\xx; \ell,K,\yy)$ is essentially supported on the $K$-neighborhood of $\yy$ up until time $t$ and that its $L^2$ norm decays quickly relative to the support volume.

\subsection{Finite speed of propogation}\label{sec:fsp}

Define the \emph{regular configuration distance} between two particle configurations $\xx,\yy\in\Ln$ to be the maximal difference in positions between two corresponding particles in the two configurations accounting only for the part of the difference appearing over the regular sites $\indint^\kappa$.
\begin{equation}\label{eq:dist}
\dist(\xx,\yy) = \sup_{a \in [n]} | \indint^\kappa \cap [\min(x_a, y_a), \max(x_a, y_a)) |
\end{equation}
Note that while $\dist$ is not a metric as it is degenerate, $\dist$ is still symmetric and satisfies the triangle inequality.

The following finite speed estimate has been used in related literature several times.  The first introduction of this method was in \cite{EY15} which identified the optimal speed and probability scales for the bound.  Later, \cite{bourgade2017eigenvector} used the idea to counter eigenvalue fluctuations with using individual hops in the short term operator.
The novel argument in the proof provided below is to more abstractly counter eigenvalue fluctuations with the Dirichlet form using properties of the global kernel from Proposition \ref{lem:globalker}.

\begin{proposition}\label{prop:FSP}
Let $0 < \kappa < 1$, $\varepsilon > 0$, and $\ell \geq N^\varepsilon$.  Then for any $\xx,\yy \in \Lambda^n$ with $\dist(\xx,\yy) > N^\varepsilon \ell$, 
\begin{equation}
\sup_{t_0 \leq s_1 \leq s_2 \leq s_1 + \ell/N \leq t} |\semi_\short(s_1,s_2;\ell)_{\xx\yy}| \leq e^{-N^{\varepsilon/2}}
\end{equation}
with overwhelming probability.  
\end{proposition}

\begin{proof}
Fix the scale $\nu = N/\ell$.  Let $h_w$ be the sequence of continuous test functions for each $w \in \RR$ satisfying $\inf_x h_w(x) = 0$ and $h_w'(x) = \one{x \in I^{\kappa/2}}\sign(x-w)$.  Let $\chi(x)$ be any smooth nonnegative function supported on $[-1,1]$ with $\int \chi(x) dx = 1$ and $\|\chi\|_\infty \leq 1$.  Then, for all $i \in [N]$, let $\psi_i(x) = \int_\RR h_{\gamma_i(s_1)}(x-y)\nu\chi(\nu y) dy$. Record here the following bounds on $\psi_i$ and its derivatives
\begin{equation}\label{eq:fsp_psicontrol}
\|\psi_i - h_{\gamma_i(s_1)}\|_\infty \leq \frac{1}{\nu}, \quad \|\psi_i'\|_\infty \leq 1, \quad \mbox{and} \quad \|\psi_i''\|_\infty \leq \nu.
\end{equation}
Consider the stopping time \begin{equation}
\tau = \max\big(\inf\{s \geq s_1 | \mbox{ the conclusion of Proposition \ref{prop:reg_eval} fails} \}, s_1\big).
\end{equation}
By Proposition \ref{prop:reg_eval}, $\tau > s_2$ with overwhelming probability.  Define the following preliminary functions on configuration space
\begin{equation}
\varphi_s(\zz) = \sum_{a=1}^n \psi_{y_a}( \lambda_{z_a}(s \wedge \tau)) \quad \mbox{ and } \quad r_s(\zz) = \semi_\short(s_1, s \wedge \tau; \ell) \delta_\yy(\zz)
\end{equation}
for any time $s \geq s_1$ and configuration $\zz \in \Ln$ where $\yy$ is from the proposition statement.  Further define the following quantities building off of the preliminary functions
\begin{equation}
\phi_s(\zz) = e^{\nu \varphi_s(\zz)}, \quad v_s(\zz) = \phi_s(\zz) r_s(\zz), \quad X_s = \sum_{\zz\in\Ln}\pi(\zz)v_s(\zz)^2
\end{equation}
so that $\phi_s, v_s \in L^2(\Ln)$ for all $s \geq s_1$ and $X_s = \|v_s\|_2^2$ is the $L^2$ norm of $v_s$.  Note that $\phi_s$ and $X_s$ are positive.
The bulk of the proof consists of showing that for all $s_1 < s < s_2$
\begin{equation}\label{eq:FSPgoal}
\partial_s \EE{X_{s}} \leq C_1 \nu \log(N) \EE{X_{s}}
\end{equation}
for some constant $C_1 >0$ which depends only on $\varepsilon$ and $n$.

Suppose \eqref{eq:FSPgoal} holds.  The primary implication is that, as the initial condition is a $\delta$-function, the $L^2$ norm of $v_{s_1}$ satisfies $X_{s_1} = 1$, and therefore
\begin{equation}\label{eq:fsp_step1a}
\EE{X_{s_2}} \leq \exp(C_1\nu(s_2-s_1)\log(N)).
\end{equation}
By the assumption that $\dist(\xx,\yy) > N^\varepsilon\ell$ there must exist some label $b \in [n]$ such that $|[x_b, y_b) \cap \indint^\kappa| \geq N^\varepsilon \ell / n$.  In this case,
\begin{equation}\label{eq:fsp_philower}
\varphi_{s_2}(\xx) = \sum_{a=1}^n \psi_{y_a}(\lambda_{x_a}(s_2 \wedge \tau)) \geq \psi_{y_b}(\lambda_{x_b}(s_2 \wedge \tau))
\end{equation}
by the positivity of $\psi_i$ for all $i\in[N]$.  This is then lower bounded by decomposing the right hand side into four terms
\begin{multline}\label{eq:fsp_psilower}
\psi_{y_b}(\lambda_{x_b}(s_2\wedge\tau))
\geq \left| \int_{\gamma_{y_b}(s_1)}^{\gamma_{x_b}(s_1)} h_{\gamma_{y_b}}'(x) dx \right| 
- \left| h_{\gamma_{y_b}(s_1)}(\gamma_{y_b}(s_1)) - \psi_{y_b}(\gamma_{y_b}(s_1))\right| \\
- \left| h_{\gamma_{y_b}(s_1)}(\gamma_{x_b}(s_1)) - \psi_{y_b}(\gamma_{x_b}(s_1)) \right| 
- \left| \int_{\gamma_{y_b}(s_1)}^{\lambda_{x_b}(s_2\wedge\tau)}\psi_{y_b}'(x) dx \right|
\end{multline}
by the fundamental theorem of calculus and the triangle inequality.  
The first term in \eqref{eq:fsp_philower} is first computed with
\begin{equation}
\left| \int_{\gamma_{y_b}(s_1)}^{\gamma_{x_b}(s_1)} h_{\gamma_{y_b}}'(x) dx \right|  = |[\gamma_{x_b}(s_1), \gamma_{y_b}(s_1)) \cap \enint^{\kappa/2}|
\end{equation}
by the construction of $h$.  
By Proposition \ref{prop:reg_fc} and the definitions of $\enint^\kappa$ and $\indint^\kappa$, $[\gamma_{x_b}(s_1), \gamma_{y_b}(s_1)) \subset \enint^{\kappa/2}$.  Proposition \ref{prop:reg_fc} also implies that $|\gamma_{x_b}(s_1)-\gamma_{y_b}(s_1)| \geq 2c\ell N^{-1+\varepsilon}$ for some constant $c>0$ depending only on $\oa$.  See Lemma \ref{lem:reg_meso} below for a more general statement.  This implies
\begin{equation}\label{eq:fsp_psilower1}
\left| \int_{\gamma_{y_b}(s_1)}^{\gamma_{x_b}(s_1)} h_{\gamma_{y_b}}'(x)dx\right| \geq 2c\ell N^{-1+\varepsilon}.
\end{equation}
The second and third terms in \eqref{eq:fsp_philower} are controlled by the first bound in \eqref{eq:fsp_psicontrol}.  
\begin{equation}\label{eq:fsp_psilower23}
|h_{\gamma_{y_b}(s_1)}(\gamma_{x_a}(s_1)) - \psi_{y_b}(\gamma_{y_a}(s_1))| + |h_{\gamma_{y_b}(s_1)}(\gamma_{y_b}(s_1)) - \psi_{y_b}(\gamma_{y_b}(s_1))| \leq \frac{2}{\nu}
\end{equation}
The fourth term in \eqref{eq:fsp_philower} is controlled by the second bound in \eqref{eq:fsp_psicontrol}.
\begin{equation}\label{eq:fsp_psilower4}
\int_{\gamma_{y_b}(s_1)}^{\lambda_{x_b}(s_2\wedge\tau)}|\psi_{y_b}'(x)| dx \leq |\gamma_{x_b}(s_2\wedge\tau) - \gamma_{y_b}(s_1)| + |\gamma_{x_b}(s_2\wedge\tau) - \lambda_{x_b}(s_2\wedge\tau)| \leq (s_2-s_1)\log(N) + N^{\oc-1} \leq \frac{\ell \log(N)}{N}
\end{equation}
by Proposition \ref{prop:reg_fc} and Proposition \ref{prop:reg_eval}, respectively.  Combining \eqref{eq:fsp_philower}, \eqref{eq:fsp_psilower}, \eqref{eq:fsp_psilower1}, \eqref{eq:fsp_psilower23}, and \eqref{eq:fsp_psilower4} gives
\begin{equation}\label{eq:fsp_step1b}
\varphi_{s_2}(\xx) \geq c \ell N^{-1+\varepsilon}
\end{equation}
deterministically.
Putting \eqref{eq:fsp_step1a} and \eqref{eq:fsp_step1b} together,
\begin{equation}
\EE{r_{s_2}^2(\xx)} \leq \EE{ X_{s_2} \phi_{s_2}(\xx)^{-2}} \leq \exp(C_1 \nu(s_2-s_1) \log(N) - 2\nu c\ell N^{-1+\varepsilon}) \leq \exp(-c N^{\varepsilon}).
\end{equation}
The result now follows from Markov's inequality.
\begin{equation}
\prob{|\semi'(s_1,s_2)_{\xx\yy}| \geq e^{-N^{\varepsilon/2}}} \leq \prob{|r_{s_2}(\xx)| \geq e^{-N^{\varepsilon/2}}} + \prob{\tau < s_2} \leq \exp(-cN^\varepsilon + 2 N^{\varepsilon/2}) + \prob{\tau < s_2} < N^{-D}
\end{equation}

It remains to prove \eqref{eq:FSPgoal}.  This is done by applying It\^o's lemma to obtain three drift terms corresponding to colored eigenvector moment flow, drift from DBM, and quadratic variation of DBM, respectively.
\begin{equation}
dX_s = 2 \ipn{ \phi_s^2, r_s dr_s } + 2\ipn{ \phi_s d\phi_s, r_s^2 } + \ipn{ r_s^2, (d\phi_s)^2 }
\end{equation}
The first term is understood via the integration by-parts in Lemma \ref{lem:ibp}
\begin{align}
2&\ipn{ \phi_s^2 r_s, \short(s) r_s} \one{s < \tau} \\
&= -\sum_{\xx \neq \zz} \short_{\xx\zz} (\phi_s(\xx)^2r_s(\xx) - \phi_s(\zz)^2 r_s(\zz)) (r_s(\xx) - r_s(\zz))\pi(\xx)\pi(\zz)\one{s < \tau} \\
&=-\sum_{\xx \neq \zz} \pi(\xx)\pi(\zz) \short_{\xx\zz} \left( (v_s(\xx) - v_s(\zz))^2 - v_s(\xx)v_s(\zz) \left(\frac{\phi_s(\xx)}{\phi_s(\zz)} + \frac{\phi_s(\zz)}{\phi_s(\xx)} - 2\right) \right)\one{s < \tau} \\
&=-2\dir_\short(s;v)\one{s < \tau} + \sum_{\xx \neq \zz} \pi(\xx)\pi(\zz)\short_{\xx\zz}v_s(\xx)v_s(\zz) \left(\frac{\phi_s(\xx)}{\phi_s(\zz)} + \frac{\phi_s(\zz)}{\phi_s(\xx)} - 2\right)\one{s < \tau}
\end{align}
The Dirichlet term will be saved for later to dominate the DBM drift.  The final sum is treated as an error term of size
\begin{equation}
\left|\sum_{\xx \neq \zz} \pi(\xx)\pi(\zz)\short_{\xx\zz}v_s(\xx)v_s(\zz) \left(\frac{\phi_s(\xx)}{\phi_s(\zz)} + \frac{\phi_s(\zz)}{\phi_s(\xx)} - 2\right)\one{s < \tau}\right| \leq C \frac{\nu^2}{N} \sum_{\xx \neq \zz} (v_s(\xx)^2 + v_s(\zz)^2) \one{\short_{xy} \neq 0} \leq C' \nu X_s
\end{equation}
where $C, C'>0$ depend only on $n$ and $\oa$.  To obtain the first inequality, suppose $\short_{xy} \neq 0$ and assume $s < \tau$.  Then there exist sites $i \neq j \in \indint^{\kappa/10}$, $|i-j| < \ell$, and labels $a,b \in [n]$ such that $\zz \in \{m_{ab}^{ij}\xx, s_{ab}^{ij}\}\backslash\{\xx\}$.  In this case,
\begin{equation}
|\short_{\xx\zz}| \asymp_n \frac{1}{N\left(\lambda_i(s \wedge \tau) - \lambda_j(s \wedge \tau)\right)^2}
\end{equation}
whereas the factor with the ratio is bounded by
\begin{align}
0 \leq \frac{\phi_s(\xx)}{\phi_s(\zz)} + \frac{\phi_s(\zz)}{\phi_s(\xx)} - 2 
&= 2 \cosh( \nu (\varphi_s(\xx) - \varphi_s(\yy))) - 2 \\
&\leq \nu^2 |\varphi_s(\xx)-\varphi(\yy)|^2 + \frac{\nu^4}{12} |\varphi(\xx) - \varphi(\zz)|^4 \partial_\xi^4 \cosh(\xi) \leq C \nu^2 \left(\lambda_i(s \wedge \tau) - \lambda_j(s \wedge \tau)\right)^2
\end{align}
for some $|\xi| \leq \nu |\varphi_s(\xx)-\varphi_s(\zz)| \leq \nu^2\left(\lambda_i(s\wedge\tau) - \lambda_j(s\wedge\tau)\right)^2 \leq C'$ for some constants $C,C'>0$ depending on $\oa$.
To get the second inequality, simply note that for each $\xx$, $|\{\zz \in \Lambda^n | \short_{\xx\zz} \neq 0\}| = O(\ell)$ independent of $x$.

The DBM terms require differentiating $\phi_s$, so some preliminary derivatives are recorded now.
\begin{equation}
d\varphi_s(\xx) = \sum_{a=1}^n \psi_{y_a}'(\lambda_{x_a}(s\wedge\tau)) d\lambda_{x_a}(s\wedge\tau) + \frac{1}{2} \psi_{y_a}''(\lambda_{x_a}(s\wedge\tau))(d\lambda_{x_a}(s\wedge\tau))^2
\end{equation}
where the eigenvalues are driven by DBM
\begin{equation}
d\lambda_k(s) = \sqrt{\frac{2}{N}} dB_k(s) + \frac{1}{N} \sum_{j \neq k} \frac{1}{\lambda_k(s) - \lambda_j(s)}ds
\end{equation}
with $B_k(s)$ independent standard Brownian motions.  This differential satisfies
\begin{equation}
\frac{(d\varphi_s(\zz))^2}{ds} = \frac{1}{N} \sum_{a,b} \psi_{y_a}'(\lambda_{x_a}(s\wedge\tau))\psi_{y_a}'(\lambda_{x_b}(s\wedge\tau)) \one{x_a=x_b}\one{s<\tau} \leq \frac{n}{N}
\end{equation}
and
\begin{equation}
\phi_s^{-1}(\zz) d\phi_s(\zz) = \nu d\varphi_s(\zz) + \frac{1}{2} \nu^2 (d\varphi_s(\zz))^2
\end{equation}
The DBM quadratic variation term is lower order
\begin{equation}
\frac{\ipn{ r_s^2, (d\phi_s)^2 }}{ds} = \nu ^2 \ipn{v_s^2, \frac{(d\varphi_s)^2}{ds} } \leq n\frac{\nu}{\ell} X_s.
\end{equation}
Lastly, the DBM drift term requires more subtle analysis.  The $\varphi_s$-QV term is of the same order as the $\phi_s$-QV contribution.  The $\varphi_s$-drift term will be split into a short range and long range contribution.  The long range contribution is bounded by the logarithmic decay in eigenvalue interactions (dyadic decomposition).  The short range part is large but explicitly controlled by the short range Dirichlet form.
\begin{align}
\frac{\EE{\ipn{r_s^2, \phi_s d\phi_s }}}{ds} 
=& \frac{\EE{\nu \ipn{ v_s^2, d\varphi_s }}}{ds} + \frac{1}{2}\EE{\nu\ipn{ v_s^2, \frac{(d\varphi_s)^2}{ds} }\one{s \leq \tau}} \\
=&\nu \EE{\sum_{\xx\in\Ln} \pi(\xx) v_s(\xx)^2 \sum_{a =1}^n \psi_{y_a}'(\lambda_{x_a}(s\wedge\tau)) \frac{1}{N}\sum_{\substack{j \neq x_a \\ |j-x_a|\leq \ell}} \frac{\one{s \leq \tau}}{\lambda_{x_a}(s\wedge\tau) - \lambda_j(s\wedge\tau)}} \label{eq:fspshort}\\
&+ \nu \EE{\sum_{\xx\in\Ln} \pi(\xx) v_s(\xx)^2 \sum_{a =1}^n \psi_{y_a}'(\lambda_{x_a}(s\wedge\tau)) \frac{1}{N}\sum_{|j-x_a| > \ell} \frac{\one{s \leq \tau}}{\lambda_{x_a}(s\wedge\tau) - \lambda_j(s\wedge\tau)}} \label{eq:fsplong} \\
&+ \nu \EE{\sum_{\xx \in \Ln} \pi(\xx) v_s(\xx)^2 \frac{1}{N} \sum_{a=1}^n \psi_{y_a}''(\lambda_{x_a}(s \wedge \tau))}\one{s \leq \tau} \label{eq:fsppsipp} \\
&+ \frac{\nu^2}{2}\EE{\ipn{ v_s^2, \frac{(d\varphi_s)^2}{ds}}} \label{eq:fspvarqv}
\end{align}
Lines \eqref{eq:fsppsipp} and \eqref{eq:fspvarqv} are both lower order
\begin{equation}
\nu \sum_{\xx \in \Ln} \pi(\xx) v_s(\xx)^2 \frac{1}{N} \sum_{a=1}^n \psi_{y_a}''(\lambda_{x_a}(s \wedge \tau)) \leq n\frac{\nu}{\ell} X_s \quad \mbox{ and } \quad \frac{\nu^2}{2} \ipn{v_s^2, \frac{(d\varphi_s)^2}{ds} } \leq \frac{n}{2}\frac{\nu}{\ell} X_s
\end{equation}
Line \eqref{eq:fsplong} will be the leading order contribution.  If $\lambda_{x_a}(s\wedge\tau) \not \in \enint^{\kappa/2}$, then $\psi'_{y_a}(\lambda_{x_a}(s\wedge\tau))=0$ and the term contributes nothing.  Otherwise, by a dyadic decomposition like that from \eqref{eq:dyadic} and control on the Stieltjes transform, there exists a constant $C_2$ depending only on $\oa$ such that
\begin{equation}
\left| \frac{1}{N} \sum_{|j-x_a|>\ell} \frac{1}{\lambda_{x_a}(s\wedge\tau) - \lambda_j(s\wedge\tau)} \right| \leq C_2 \log(N)
\end{equation}
so \eqref{eq:fsplong} is bounded by $nC_2 \log(N) \nu\EE{X_s}$.  It remains to bound the short range contribution \eqref{eq:fspshort} by $\dir_\short(s; v_s) + C_3 \nu X_s$.  This is done by reversing the order of summation and symmetrizing twice -- first over pairs of sites, then over conjugate pairs of configurations.
\begin{multline}
\sum_{\xx \in \Ln} \pi(\xx) v_s(\xx)^2 \sum_{a =1}^n \psi_{y_a}'(\lambda_{x_a}(s\wedge\tau)) \frac{1}{N}\sum_{\substack{j \neq x_a \\ |j-x_a| \leq \ell}} \frac{1}{\lambda_{x_a}(s\wedge\tau) - \lambda_j(s\wedge\tau)} \\
= \sum_{j-\ell \leq i < j} \frac{1}{\lambda_i(s\wedge\tau) - \lambda_j(s\wedge\tau)} \sum_{x \in \Lambda^n} \pi(\xx)v_s(\xx)^2 \sum_{a=1}^n (\one{x_a=i} - \one{x_a=j})\psi_{y_a}'(\lambda_{x_a}(s\wedge\tau))
\end{multline}
Note that if $x_a \in \{i,j\} \backslash \indint^\kappa$, then the $\psi'$ term will vanish.  Therefore, we may restrict the $i,j$ summation to $\indint$.  Moreover, for every fixed $i,j$ the coefficient on $v_s(\xx)$ is the negation of the coefficient on $v_s((ij)\actN \xx)$, where $(ij)\in S_N$ is the position transposition swapping all contents on sites $i$ and $j$.  Recall $\actN$ is the action defined in Definition \ref{defn:actN}.
\begin{equation}
\begin{aligned}
&\sum_{\xx \in \Ln} \pi(\xx) v_s(\xx)^2 \sum_{a=1}^n \psi_{y_a}'(\lambda_{x_a}(s\wedge\tau)) \frac{1}{N}\sum_{\substack{j \neq x_a \\ |j-x_a|\leq \ell}} \frac{1}{\lambda_{x_a}(s\wedge\tau) - \lambda_j(s\wedge\tau)} \\
=& \frac{1}{4}\sum_{\substack{i,j \in J \\ 0 < |i-j| \leq \ell}} \frac{1}{\lambda_i(s\wedge\tau) - \lambda_j(s\wedge\tau)} \sum_{\xx \in \Ln} \pi(\xx)(v_s(\xx)^2 - v_s((ij)\actN\xx)^2) \sum_{a=1}^n (\one{x_a=i} - \one{x_a=j})\psi_{y_a}'(\lambda_{x_a}(s\wedge\tau)) \\
\leq& \sum_{ij\xx} \pi(\xx) \left( 
\begin{aligned}
&\frac{|v_s(\xx) - v_s((ij)\actN\xx)|^2}{\nu N\left(\lambda_i(s\wedge\tau) - \lambda_j(s\wedge\tau)\right)^2} \\
&+ C_3' \ell^{-1} |v_s(\xx) + v_s((ij)\actN\xx)|^2 \left|\sum_{a=1}^n (\one{x_a=i} - \one{x_a=j}) \psi_{y_a}'(\lambda_{x_a}(s\wedge\tau))\right|^2
\end{aligned}
\right) \\
\leq& \nu^{-1} \dir_\short(s;v_s) + C_3 X_s
\end{aligned}
\end{equation}
where the second to last line used the AM-GM inequality.  The bound on the first term in the last inequality is a consequence of Lemma \ref{lem:kernel} and the second term comes from the fact that the sum over $a$ is $O(1)$ and $v_s(\xx)$ appears in $O(\ell)$ swaps.
In conclusion,
\begin{equation}
\partial_s \EE{X_s} \leq \EE{-2\dir_\short(s;v_s) + C_1 \nu X_s + \frac{5n}{2} \frac{\nu}{\ell}X_s + C_2 \log(N) \nu X_s + \dir_\short(s;v_s) + C_3 \nu X_s } \leq C \nu \log(N) \EE{X_s}
\end{equation}
proving \eqref{eq:FSPgoal} and finishing the proof of FSP.
\end{proof}

In practice, finite speed of propogation is used through the following corollary.
\begin{corollary}\label{cor:FSPcomm}
For any times satisfying $t/2 \leq s_1 \leq s_2 \leq s_1 + \ell/N \leq t$, any configuration $\yy \in \Ln$ supported on $\indint^\kappa$ (that is, $y_a\in\indint^\kappa$ for all $a\in[n]$), we have that the $L^\infty \rightarrow L^\infty$ operator norm of the commutator satisfies
\begin{equation}
\left\|[\semi_\short(s_1,s_2; \ell), \Av(\yy, K)]\right\|_{\infty, \infty} \leq 5n!!\frac{N^\varepsilon\ell}{K}
\end{equation}
for every $\varepsilon>0$.
\end{corollary}

\begin{remark}
Note that both operators $\Av(K,\yy)$ and $\semi_\short(s_1,s_2;\ell)$ have $L^2(\{\xx\in\Ln | \supp(\xx) \indint^{\kappa/10}\})$ as an invariant subspace and $\Av(K,\yy)$ vanishes off this subspace.  Therefore, it is true that for any test function $\phi \in L^\infty(\Ln)$, 
\begin{equation}\label{rmk:fspcomm}
\|[\semi_\short(s_1,s_2;\ell), \Av(K,\yy)]\phi\|_\infty \leq \|\one{\Ln(\indint^{\kappa/10})}\phi\|_\infty \frac{N^\varepsilon\ell}{K}
\end{equation}
\end{remark}

\begin{proof}
Let $\phi\in L^1(\Ln)$.
Begin by expanding the commutator in terms of matrix entries, expanding the averaging operator entries, and swapping the order of summation.  For all $\xx\in\Ln$, 
\begin{equation}
\begin{aligned}
\label{eq:fsp_cor}
\semi_\short(s_1,s_2), \Av(K,\yy)] \phi(\xx)
=& \sum_{\zz\in\Ln} \semi_\short(s_1,s_2)_{\xx\zz}\pi(\zz)\Av(\zz;K,\yy)\phi(\zz) - \Av(\xx;K,\yy)\semi_\short(s_1,s_2)_{\xx\zz}\pi(\zz)\phi(\zz) \\
=& \frac{1}{K}\sum_{\alpha=K}^{2K-1} \sum_{\zz\in\Ln} \semi_\short(s_1,s_2)_{\xx\zz} \pi(\zz) \phi(\zz)\left( \one{\|\zz-\yy\|_1 \leq \alpha} - \one{\|\xx-\yy\|_1 \leq \alpha} \right)
\end{aligned}
\end{equation}
where we used the expansion of $\Av(\xx; K,\yy)$ from Definition \ref{defn:short_range}.  Each $K \leq \alpha \leq 2K-1$ falls into one of three categories and each category will be dealt with separately.  

First, suppose $\dist(\xx,\yy) \leq \alpha - N^\varepsilon\ell$.  In this case, $\xx\in\indint^{\kappa/2}$.  If $\dist(\zz,\xx)\leq N^\varepsilon\ell$, then $\zz\in\indint$ as well, so $\|\xx-\yy\|_1 = \dist(\xx,\yy) \leq \alpha$ and $\|\zz-\yy\|_1=\dist(\zz,\yy) \leq \dist(\zz,\xx) + \dist(\xx,\yy) \leq \alpha$.  Therefore, for such $\alpha$, the $\zz$-sum may be restricted to configurations satisfying $\dist(\xx,\zz) > N^\varepsilon\ell$ on which $|\semi_\short(s_1,s_2)_{\xx\zz}| \leq \exp^{-N^{\varepsilon/2}}$
\begin{equation}\label{eq:fsp_cor_far}
\left|\sum_{\zz\in\Ln} \semi_\short(s_1,s_2)_{\xx\zz} \pi(\zz) \phi(\zz)\left( \one{\|\zz-\yy\|_1 \leq \alpha} - \one{\|\xx-\yy\|_1 \leq \alpha} \right)\right| 
\leq 2 e^{-N^{\varepsilon/2}} \|\phi\|_1
\end{equation}
by Proposition \ref{prop:FSP}.

Second, suppose $\dist(\xx,\yy) \geq \alpha + N^\varepsilon\ell$.  In this case, if $\dist(\zz,\xx) \leq N^\varepsilon$, then $\|\xx-\yy\|_1 \geq \dist(\xx,\yy) \geq \alpha$ and $\|\zz-\yy\|_1 \geq \dist(\zz,\yy) \geq \dist(\xx,\yy) - \dist(\zz,\xx) \geq \alpha$.  Therefore, for such $\alpha$, the $\zz$-sum may again be restricted to configurations satisfying $\dist(\xx,\zz) > N^\varepsilon\ell$ and again the bound \eqref{eq:fsp_cor_far} is satisfied.

Third, all other values of $\alpha$ must satisfy $\alpha-N^\varepsilon\ell \leq \dist(\xx,\yy) \leq \alpha + N^\varepsilon$.  For this case, we use the dual of Lemma \ref{lem:L1}, that $\|\semi_\short(s_1, s_2)\|_{\infty,\infty} \leq n!!$.  Indeed, to see this, note that the adjoint of $\semi_\short(s_1,s_2)$ is the transition semigroup for the time reversed dynamics in the sense that $\semi_\short(s_1, s_2)^*f = h_{s_1}$ where $h$ is the solution to $\partial_s h_s = \short(s)$ with terminal condition $h_{s_2} = f$.  This transition semigroup satisfies all conditions of Lemma \ref{lem:L1}, so $\|\semi_\short(s_1,s_2)\|_{\infty,\infty} = \|\semi_\short(s_1,s_2)^*\|_{\infty\infty} \leq n!!$.  To use this observation,
\begin{multline}\label{eq:fsp_cor_near}
\left|\sum_{\zz\in\Ln} \semi_\short(s_1,s_2)_{\xx\zz} \pi(\zz) \phi(\zz)\left( \one{\|\zz-\yy\|_1 \leq \alpha} - \one{\|\xx-\yy\|_1 \leq \alpha} \right)\right| \leq n!! \|\one{\|\cdot - \yy\|_1\leq\alpha}\phi\|_\infty + \|\semi_\short(s_1,s_2)\phi\|_\infty \\
\leq 2 n!! \|\phi\|_\infty
\end{multline}

The first and second cases will occur for at most $K$ choices of $\alpha$ combined.  The third case will occur for at most $2N^\varepsilon\ell$ choices of $\alpha$.  Therefore, bounding \eqref{eq:fsp_cor} by \eqref{eq:fsp_cor_far} and \eqref{eq:fsp_cor_near} gives
\begin{equation}
\|[\semi_\short(s_1, s_2), \Av(K,\yy)] \phi\|_\infty \leq 2e^{-N^{\varepsilon/2}}\|\phi\|_1 + 4n!! \frac{\ell}{K}N^\varepsilon \|\phi\|_\infty \leq 5n!!\frac{\ell}{K}N^\varepsilon\|\phi\|_\infty
\end{equation}
where the last inequality comes from $\|\phi\|_1 \leq \pi(\Ln) \|\phi\|_\infty$ and $\pi(\Ln)$ is only polynomially large in $N$.
\end{proof}

\subsection{Duhamel expansion of long range perturbation}\label{sec:duhamel}

\begin{lemma}\label{lem:reg_meso}
As a consequence of the local laws, there exists some universal constant $C$ such that with overwhelming probability for any $t_0 \leq s \leq t$ and any interval $I$ centered in $\enint^\kappa$ with $1 > |I| \geq N^{-1+\oc}$
\begin{equation}
C^{-1}|I|N \leq |\{i | \gamma_i(s) \in I\}|, |\{i| \lambda_i(s) \in I\}| \leq C|I|N
 \quad \mbox{and} \quad 
\sum_{i: \lambda_i(s) \in I} \langle \vv, \vu_i(s) \rangle^2 \leq C|I| N^\ob
\end{equation}
for any $\vv \in S$ where $S$ is the set of regular unit vectors from Assumption \ref{ass:eval}.
\end{lemma}

\begin{proof}
This is a consequence of the control on $m_N(s;z)$ and $\grn(s;z)$ implied by Proposition \ref{prop:reg_fc}, Proposition \ref{prop:reg_eval}, Proposition \ref{prop:isotropiclaw}, and Assumption \ref{ass:evec}.  Indeed, $m_N(s;z)$ is the Stieltjes transform of the empirical spectral distribution $N^{-1} \sum_{i=1}^N \delta_{\lambda_i(s)}$ while $\ipr{\vv, \grn(s;z) \vv}$ is the Stieltjes transform of the weighted empirical spectral distribution $\sum_{i=1}^N |\ipr{\vv, \vu_i(s)}|^2 \delta_{\lambda_i(s)}$.  The result now follows from properties of the Stieltjes transform.
\end{proof}

\begin{lemma}\label{lem:dyadic}
There exists a constant $c>0$ depending only on $\oa$ such that for any $\ell \geq N^{-1+\oc}$, any site $i \in \indint^\kappa$, any time $t_0 \leq s \leq t$, and any order parameter $\ob > 0$, 
\begin{equation}\label{eq:dyadic}
\sum_{|j-i| > \ell} \frac{1}{N\left(\lambda_j(s) - \lambda_i(s)\right)^2} \leq c\frac{N}{\ell}
\quad\mbox{and}\quad
\sum_{|j-i| > \ell} \frac{\ipr{ \vv, \vu_j(s)} \ipr{\vu_j(s), \vw}}{N\left(\lambda_i(s) - \lambda_j(s)\right)^2} \leq N^\ob \frac{N}{\ell}
\end{equation}
with overwhelming probability.
\end{lemma}

\begin{proof}
Use a dyadic decomposition
\begin{equation}
\sum_{|j-i| > \ell} \frac{1}{N\left(\lambda_j(s) - \lambda_i(s)\right)^2} \leq \sum_{q=1}^{\lceil\log_2 N/\ell \rceil} \sum_{j : C^{-1} \ell 2^{q-1} \leq |j-i| \leq C^{-1} \ell 2^q} \frac{N}{2^{2q}\ell^2} \leq  C^{-1} \frac{N}{\ell} \sum_{q=1}^{\lceil \log_2 N/\ell\rceil} 2^{-q} \leq c\frac{N}{\ell}
\end{equation}
where $c = C^{-1}$ from Lemma \ref{lem:dyadic}. Similarly,
\begin{equation}
\sum_{|j-i| > \ell} \frac{\ipr{\vv, \vu_j(s)}\ipr{\vu_j(s), \vw}}{N\left(\lambda_i(s) - \lambda_j(s)\right)^2} \leq \sum_{q=1}^{\log_2 N/\ell} \frac{N}{2^{2q} \ell^2} \sum_{j : \frac{\ell 2^{q-1}}{C} \leq |j-i| \leq \frac{\ell 2^q}{C}} |\ipr{\vv, \vu_j(s)}|^2 + |\ipr{\vu_j(s), \vw}|^2 \leq N^\ob \frac{N}{\ell}
\end{equation}
In this equation, the constant factors are absorbed into the $N^\ob$ error.
\end{proof}

\begin{proposition}\label{prop:duhamel}
Let $t_0 \leq s_1 \leq s_2 \leq s_1+\ell/N \leq t$.  Then for all $\xx \in \Ln$ supported on $\indint^\kappa$, that is $x_a\in\indint^\kappa$ for all $a\in[n]$,
\begin{equation}
|(\semi(s_1,s_2) - \semi_\short(s_1,s_2;\ell))f_s(\xx)| \leq N^{1+\frac{n}{2}\ob}(s_2-s_1)/\ell
\end{equation}
\end{proposition}

\begin{proof}
Appeal to Duhamel's principle
\begin{equation}
(\semi(s_1, s_2) - \semi_\short(s_1, s_2;\ell)) f_{s_1}(\xx) = \int_{s_1}^{s_2} \semi_\short(s, s_2;\ell) (\gen(s)-\short(s)) f_s(\xx) ds
\end{equation}
For any $\yy\in\Ln$ supported on $\indint^{\kappa/2}$,
\begin{equation}
(\gen(s)-\short(s))f_s(\yy) = \sum_{|i-j|> \ell} \frac{\move_{ij} - \exch_{ij}}{N\left(\lambda_i(s) - \lambda_j(s)\right)^2} f_s(\yy)
\end{equation}
First consider the exchange term.  This is lower order because $\exch_{ij} f(\yy) = 0$ unless $n_i(\yy) n_j(\yy) > 0$ in which case $|\exch_{ij} f_s(\yy)| < N^{\frac{n}{2}\ob}$ by delocalization in $\indint^{\kappa/2}$, see Corollary \ref{cor:deloc}.  Therefore,
\begin{equation}
\sum_{|i-j| > \ell} \frac{1}{N\left(\lambda_i(s) - \lambda_j(s)\right)^2} |\exch_{ij} f_s(\yy) | \leq \frac{N}{\ell^2}N^{\frac{n}{2}\ob}
\end{equation}
Next consider the move term.  Manipulate the expression by pivoting on a single site in the support of $\yy$.  Then for each pair of $a,b \in [n]$, $y_a=y_b=i$, the paired jump gives the second expression in \eqref{eq:dyadic} whereas the stationary terms give the first expression in \eqref{eq:dyadic}.
\begin{multline}
\sum_{|i-j|> \ell} \frac{1}{N\left(\lambda_i(s) - \lambda_j(s)\right)^2} |\move_{ij} f_s(\yy)| = \sum_{i : n_i(\yy) > 0} \sum_{|j-i|>\ell} \sum_{a \neq b} \frac{|f(m_{ab}^{ij}\yy) - f(\yy)|}{N\left(\lambda_i(s) - \lambda_j(s)\right)^2} \\ \leq \sum_{i : n_i(\yy) > 0} \sum_{|j-i|>\ell} \sum_{a<b} \one{y_a=y_b=i} \frac{N^{(\frac{n}{2}-1)\ob} f^{ab}_s(j) + N^{\frac{n}{2}\ob}}{N\left(\lambda_i - \lambda_j\right)^2} \leq \frac{N}{\ell} N^{\frac{n}{2}\ob}
\end{multline}
where the second line follows from delocalization, Corollary \ref{cor:deloc}.  In the last line, $f^{ab}$ is the $\Lambda^2$ observable with test vectors $\vv_a$ and $\vv_b$ and the input $j$ denotes the particle configuration with both $a$ and $b$ particles located at site $j\in[N]$.  The final inequality uses \eqref{eq:dyadic} and any potential constant factors are absorbed into the $\ob$ control parameter to ease notation.  To conclude the proof, use finite speed of propagation
\begin{multline}
\semi_\short(s_1,s;\ell) (\gen(s)-\short(s)) f_s(\xx) = \sum_{\yy\in\Ln} \semi_\short(s_1, s; \ell) \one{\supp(\yy)\subset \indint^{\kappa/2}}(\gen(s)-\short(s)) f_s)(\xx) + O(\pi(\Ln)e^{-N^{\varepsilon/2}}) \\
< \frac{N}{\ell}N^{\frac{n}{2}\ob}
\end{multline}
by Propostion \ref{prop:FSP}.  The indicator is over the event that $\yy$ is supported on $\indint^{\kappa/2}$.  The long range error term is lower order as $\pi(\Ln)$ is only polynomially large in $N$.  The last inequality uses the dual of Proposition \ref{lem:L1}, that the transition semigroup $\semi_\short$ is an $L^\infty$ contraction up to a constant.  See the paragraph preceeding \eqref{eq:fsp_cor_near} for an explanation.
\end{proof}

\subsection{Averaged local relaxtion}\label{sec:l2relax}
The idea here is to use the $L^2$ positivity of the generator to smooth coefficients.  As with the Maximum Principle from (\cite{QUE}, \cite{bourgade2018distribution}, ), this argument then inducts on particle number and uses Gronwall's inequality to prove short time convergence.  Unlike the maximum principle, since we can only smooth coefficients on the $L^2$ norm and not the $L^\infty$ norm, the short time convergence only holds for the $L^2$ average in our $K$-neighborhood.  In Section \ref{sec:energy}, it is shown that the CEMF  dynamics exhibit sufficiently fast local mixing properties so that the following $L^2$ convergence implies pointwise convergence.

\begin{proposition}\label{prop:l2}
Assume Theorem \ref{thm:main} holds for moments of degree $n-2$ with exponent $\od(n-2)$.  For any scales satisfying $N^{-1} \ll \eta \ll T_1 \ll \ell/N \ll K/N \ll t_0$, there exists a large constant $C>0$ depending only on $\oa$ and $n$ such that for any $\varepsilon, \ob > 0$ 
\begin{equation}
\sum_{\xx \in \Ln} \pi(\xx) |g_{t_0 + T_1}(\xx; \yy, K)|^2 \leq K^{n/2} \err^2
\end{equation}
where $\err$ is the scale given by
\begin{equation}\label{eq:l2_defn_err}
\err = N^{\frac{n}{2}\ob} \left( \frac{N^\varepsilon\ell}{K} + \frac{NT_1}{\ell} + \frac{N^\varepsilon}{N\eta} + \frac{N\eta}{\ell} + \frac{N^{\varepsilon-\ob}}{\sqrt{N\eta}} + \frac{N^\varepsilon + K + (t-t_0)\log N}{N t_0} + N^{-\od(n-2)} \right)
\end{equation}
uniformly for particle configuration $\yy \in \Ln$ supported on $\indint^\kappa$ and eigenvalue trajectory $\bl$ satisfying the conclusion of Proposition \ref{prop:reg_eval} and Proposition \ref{prop:isotropiclaw} on the time interval $[t_0, t]$.
\end{proposition}

\begin{proof}
The estimates in the following argument hold uniformly for $\yy\in\Ln$ supported on $\indint^\kappa$, so fix such a particle configuration $\yy$.  To unclutter notation in this proof, drop the $\ell$, $K$, and $\yy$ in the short range notation and simply write $g_s(\xx) = g_s(\xx; \ell, K, \yy)$ for the short range observable, $\Av = \Av(K,\yy)$ for the averaging operator, and $\Av(\xx) = \Av(\xx; K, \yy)$ for its coefficients.
By Lemma \ref{lem:definite}, each $\gen_{ij}$ is negative definite so the Dirichlet form is bounded by reduced coefficients
\begin{multline}\label{eq:l2_decomp}
\partial_s \|g_s\|_2^2 = \sum_{\substack{i,j \in \indint \\ j-\ell \leq i < j}} c_{ij}(s) \ipn{g_s, \gen_{ij} g_s} \leq \frac{1}{N\eta} \sum_{\substack{i,j \in \indint \\ j-\ell \leq i < j}} \sum_{\xx \in \Ln} \frac{\pi(\xx)\eta}{\left(\lambda_i(s) - \lambda_j(s)\right)^2 + \eta^2} g_s(\xx)(\move_{ij} - \exch_{ij}) g_s(\xx) \\
= \frac{I + II + III}{\eta}
\end{multline}
Split up the sum into three parts: exchange terms, move-to-occupied-site terms, and move-to-unoccupied-site terms.  The first two summands will be lower order by volume considerations.  The vast majority of terms are of the third type and that summand will be most difficult to bound.  The first part is guaranteed to be positive as $\exch_{ij}$ is positive definite.
\begin{equation}\label{eq:l2_I}
I = -\frac{1}{N} \sum_{\substack{i,j \in \indint \\ j-\ell \leq i < j}} \sum_{\xx \in \Ln}\pi(\xx) \frac{\eta}{\left(\lambda_i(s) - \lambda_j(s)\right)^2 + \eta^2} g_s(\xx) \exch_{ij} g_s(\xx) \leq \frac{1}{N\eta} \|g_s\|_2^2
\end{equation}
The inequality was obtained by expanding $\exch_{ij} g_s(\xx) \asymp g_s(s_{ab}^{ij}\xx)-g_s(\xx)$, using Schwarz, and noting that only finitely many $i,j,a,b$ contribute to the sum for each $\xx$.
The second part is dealt with similarly.
\begin{equation}\label{eq:l2_II}
II = \frac{1}{N} \sum_{\substack{i,j \in \indint \\ j-\ell \leq i < j}} \sum_{\xx \in \Ln}\pi(\xx) \frac{\eta}{\left(\lambda_i(s) - \lambda_j(s)\right)^2 + \eta^2} g_s(\xx) \move_{ij} g_s(\xx) \one{n_i(\xx) n_j(\xx) > 0} \leq \frac{1}{N\eta} \|g_s\|_2^2
\end{equation}
again by expanding $\move_{ij} g_s(\xx) \asymp g_s(m_{ab}^{ij}\xx)-g_s(\xx)$, using Schwarz, and noting that only finitely many $i,j,a,b$ contribute to the sum for each $\xx$.
The third part will be the leading contribution.
\begin{multline}\label{eq:l2_IIIa}
III = \frac{1}{N} \sum_{\substack{i \neq j \in J \\ |i-j| \leq \ell}} \sum_{\xx \in \Ln}\pi(\xx) \Im \frac{g_s(\xx) \move_{ij}g_s(\xx)}{\lambda_j(s) - z_i(s)}\one{n_i(\xx) > n_j(\xx) = 0} \\
= \sum_{\xx \in \Lambda^n}\pi(\xx) \frac{g_s(\xx)}{N} \sum_{\substack{i \neq j \in \indint \\ |i-j| \leq \ell}} \frac{\one{n_i(\xx) > n_j(\xx) = 0}}{n_i(\xx)-1} \Im \frac{1}{\lambda_j(s) - z_i(s)} \sum_{a \neq b} \one{x_a=x_b=i}\left(g_s(m_{ab}^{ij}\xx) - g_s(\xx)\right)
\end{multline}
In the first line, we use the notation $z_i(s) = \lambda_i(s) + i \eta$. 

I claim that for every $\xx \in \Ln$ supported on $\indint^{\kappa/2}$ and every time $t_0 \leq s \leq t_0 + T_1$,
\begin{equation}\label{eq:l2stationary}
\frac{1}{N} \sum_{\substack{i \neq j \in \indint \\ |i-j| \leq \ell}} \frac{\one{n_i(\xx) > n_j(\xx) = 0}}{n_i(\xx)-1}\Im \frac{1}{\lambda_j(s) - z_i(s)} \sum_{a \neq b} \one{x_a=x_b=i} = \sum_{a=1}^n \Im m_N(s; z_{x_a}(s)) + \err_1(s, \xx)
\end{equation}
and
\begin{equation}\label{eq:l2dynamic}
\frac{1}{N} \sum_{\substack{i \neq j \in \indint \\ |i-j| \leq \ell}}\frac{\one{n_i(\xx) > n_j(\xx) = 0}}{n_i(\xx)-1} \sum_{a \neq b} \one{x_a=x_b=i} \Im \frac{g_s(m_{ab}^{ij}\xx)}{\lambda_j(s) - z_i(s)} = \err_2(s, \xx)
\end{equation}
where $\err_1(s, \xx)$ and $\err_2(s, \xx)$ are error terms which satisfy the following bounds.  For all $\xx\in\Ln$,
\begin{align}
|\err_1(s, \xx)| &\leq N \label{eq:l2_err1_far} \\
|\err_2(s, \xx)| &\leq N \label{eq:l2_err2_far}
\end{align}
and when $\xx$ is supported on $\indint^{\kappa/2}$ stronger bounds hold
\begin{align}
|\err_1(\xx)| &\leq \frac{N\eta}{\ell} + \frac{1}{N\eta} \label{eq:l2_err1_near} \\
|\err_2(\xx)| &\leq  \err \label{eq:l2_err2_near}
\end{align}
uniformly for all times $t_0 \leq s \leq t$ where $\err$ is the error scale defined in Equation \eqref{eq:l2_defn_err}.  If this is the case, then plugging \eqref{eq:l2stationary} and \eqref{eq:l2dynamic} into \eqref{eq:l2_IIIa} yields
\begin{equation}\label{eq:l2_IIIb}
III = \sum_{\xx\in\Ln}\pi(\xx)g_s(\xx)\left[ -g_s(\xx)\left( \sum_{a=1}^n \Im m_N(s; z_{x_a}(s)) + \err_1(s,\xx) \right) + \err_2(s,\xx) \right]
\end{equation}
For all $\xx\in\Ln$ such that $\dist(\xx, \yy) > 3K$, we have
\begin{equation}\label{eq:l2_farterms}
g_s(\xx) = \ipn{\delta_\xx, \semi_\short(t_0,s;\ell) f_{t_0}} = \sum_{\zz\in\Ln} \pi(\zz) \semi_\short(t_0,s;\ell)_{\xx\zz} \Av(\zz) f_{t_0}(\zz)
\end{equation}
If $\Av(\zz) \neq 0$, then $\dist(\yy,\zz) \leq 2K$ which means that $\zz$ is supported on $\indint^{\kappa/2}$ and that $\dist(\xx,\zz) > \dist(\yy,\xx) - \dist(\yy,\zz) > K > \ell N^\varepsilon$ for some small enough $\varepsilon>0$.  The finite speed of propagation bound from Proposition \ref{prop:FSP} tells us that $|\semi_\short(t_0,s;\ell)_{\xx\zz}|\leq e^{-N^{\varepsilon/2}}$.  For such configurations $\zz$, $|\Av(\zz)|\leq 1$ by the definition of $\Av$ in Definition \ref{defn:short_range}, $|f_{t_0}(\zz)| \leq N^{\frac{n}{2}\ob}$ by delocalization in Corollary \ref{cor:deloc} and the definition of the moment observable in Definition \ref{defn:levmo}.  By the triangle inequality \eqref{eq:l2_farterms} is bounded by
\begin{equation}\label{eq:l2_expbound}
|g_s(\xx)| \leq \pi(\{\zz\in\Ln|\dist(\yy,\zz) \leq 3K\}) e^{-N^{\varepsilon/2}} N^{\frac{n}{2}\ob}
\end{equation}
In addition, we know that $\|g_s\|_\infty \leq n!! \|\Av f_{t_0}\|_\infty \leq N^{\frac{n}{2}\ob}$ by the dual to Lemma \ref{lem:L1} and delocalization.  Moreover, for all $i,j\in[N]$, $\Im \frac{1}{\lambda_j(s) - z_i(s)} \leq \frac{1}{\eta}$ and taking an average over $j\in[N]$ gives $\Im m_N(s;z_i(s)) \leq \frac{1}{\eta}\leq N$.  These observations along with \eqref{eq:l2_err1_far} and \eqref{eq:l2_err2_far} allow us to restrict the $\xx$-sum in \eqref{eq:l2_IIIb} to $\xx\in\Ln$ such that $\dist(\xx,\yy)\leq 3K$
\begin{equation}\label{eq:l2_IIIc}
III \leq 
\begin{aligned}[t]
&\sum_{\dist(\xx,\yy)\leq 3K} \pi(\xx)g_s(\xx)\left[ -g_s(\xx)\left( \sum_{a=1}^n \Im m_N(s; z_{x_a}(s)) + \err_1(s,\xx) \right) + \err_2(s,\xx) \right] \\
+ &\pi(\Lambda^n) e^{-N^{\varepsilon/2}}\left[ e^{-N^{\varepsilon/2}} 2N + N \right]
\end{aligned}
\end{equation}
For the near sum, when $\xx\in\Ln$ with $\dist(\xx,\yy) \leq 3K$, it must be the case that $\xx$ is supported on $\indint^{\kappa/2}$.  Therefore, by Proposition \ref{prop:reg_fc} and \ref{prop:reg_eval} and the local error bounds \eqref{eq:l2_err1_near},
\begin{equation}\label{eq:l2_III_err1}
\sum_{\dist(\xx,\yy)\leq 3K} \pi(\xx)|g_s(\xx)|^2 \left( \sum_{a=1}^n \Im m_N(s; z_{x_a}(s)) + \err_1(s,\xx) \right) \geq C_1 \sum_{\dist(\xx,\yy) \leq 3K} \pi(\xx)|g_s(\xx)|^2
\end{equation}
for some constant $C_1 > 0$ depending only on $n$ and $\oa$ originating from Proposition \ref{prop:reg_fc}.  The $\err_2(s,\xx)$ sums are bounded with AM-GM
\begin{equation}\label{eq:l2_III_err2}
\sum_{\dist(\xx,\yy)\leq 3K} \pi(\xx)g_s(\xx)\err_2(s,\xx) \leq 2 \sum_{\dist(\xx,\yy)\leq 3K} \pi(\xx) \left( \frac{C_1}{4}|g_s(\xx)|^2 + \frac{4}{C_1}\err^2\right)
\end{equation}
by \eqref{eq:l2_err2_near}.  There exists a constant $C_2 > 0$ depending only on $n$ such that
\begin{equation}\label{eq:l2_counts}
\pi(\Ln) \leq C_2 N^{n/2} \quad \mbox{and} \quad \pi\{\xx\in\Ln | \dist(\xx,\yy)\leq 3K\} \leq C_2 K^{n/2}
\end{equation}
Therefore, combining \eqref{eq:l2_IIIc}, \eqref{eq:l2_III_err1}, \eqref{eq:l2_III_err2}, and \eqref{eq:l2_counts}
\begin{equation}\label{eq:l2_IIId}
III \leq \frac{C_1}{2} \sum_{\dist(\xx,\yy)\leq 3K} \pi(\xx)|g_s(\xx)|^2 + K^{n/2} \err^2
\end{equation}
Finally, one more application of \eqref{eq:l2_expbound} and \eqref{eq:l2_counts} gives 
\begin{equation}\label{eq:l2_III_long}
\|g_s\|_2^2 - \sum_{\dist(\xx,\yy)\leq 3K} \pi(\xx) |g_s(\xx)|^2 \leq C_2^2 N^{n/2} K^{n/2} e^{-N^{\varepsilon/2}} N^{\frac{n}{2}\ob}
\end{equation}
which tends to $0$ as $N$ grows.  At last, \eqref{eq:l2_IIId} and \eqref{eq:l2_III_long} give our desired result for $III$
\begin{equation}\label{eq:l2_III}
III \leq \frac{C_1}{3} \|g_s\|_2^2 + K^{\frac{n}{2}} \err^2
\end{equation}
We conclude from \eqref{eq:l2_decomp}, \eqref{eq:l2_I}, \eqref{eq:l2_II}, and \eqref{eq:l2_III} that the Dirichlet form of our short range observable satisfies
\begin{equation}
\partial_s \|g_s\|_2^2 \leq -\frac{C_1}{4} \frac{\|g_s\|_2^2}{\eta} + \frac{K^{n/2} \err^2}{\eta}
\end{equation}
for all $t_0 \leq s \leq t$.    Moreover, by delocalization we know that $\|g_{t_0}\|_2^2 \leq C_2K^{n/2} N^{(n/2)\ob}$.  By Gronwall's inequality, for any $t_1 \geq t_0 + N^\varepsilon \eta$.
\begin{equation}
\|g_{t_1}\|_2^2 \leq K^{n/2} \err^2
\end{equation}
where throughout $\err$ has been absorbing constant factors.

It remains to complete the proof of \eqref{eq:l2_III} by proving the control bounds \eqref{eq:l2_err1_far}, \eqref{eq:l2_err2_far}, \eqref{eq:l2_err1_near}, and \eqref{eq:l2_err2_near}.  Let's start with $\err_1(s,\xx)$ terms
\begin{multline}
\err_1(s,\xx) = \sum_{a=1}^n \Im m_N(z_{x_a}(s)) - \frac{1}{N} \sum_{\substack{i \neq j \in \indint \\ |i-j| \leq \ell}} \frac{\one{n_i(\xx) > n_j(\xx) = 0}}{n_i(\xx)-1}\Im \frac{1}{\lambda_j(s) - z_i(s)} \sum_{a \neq b} \one{x_a=x_b=i} \\
= \frac{1}{N} \Im \sum_{a=1}^n \sum_{|j-x_a| > \ell} \frac{1}{\lambda_j(s) - z_{x_a}(s)} + \frac{1}{N} \sum_{a = 1}^n \Im \frac{1}{\lambda_{x_a}(s) - z_{x_a}(s)}
\end{multline}
for every $t_0 \leq s \leq t$ and $\xx\in\Ln$.  For arbitrary $\xx$, each term satisfies $\Im (\lambda_j(s) - z_i(s))^{-1} \leq \eta^{-1}$ for all $i,j\in[N]$ and there are at most $n N$ such terms.  This proves \eqref{eq:l2_err1_far}.  On the other hand, when $\xx$ is supported on $\indint^{\kappa/2}$, the first term is bounded by $N\eta/\ell$ by \eqref{eq:dyadic} and the second term is bounded by $(N\eta)^{-1}$ using the same naive bound as in the long range case.  This proves \eqref{eq:l2_err1_near}.

The long range estimate for $\err_2(s,\xx)$ in \eqref{eq:l2_err2_far} also follows from the $\Im (\lambda_j(s)-z_i(s))^{-1} \leq \eta^{-1}$ bound along with delocalization from Corollary \ref{cor:deloc} and $L^\infty$ contraction of the short range semigroup from the dual of Lemma \ref{lem:L1} on the terms in \eqref{eq:l2dynamic}.

Equation \eqref{eq:l2_err2_near} requires more work.  For this, fix a configuration $\xx\in\Ln$ and swap the order of summation and split the $g_s$ factor into two parts corresponding to the decomposition $g_s = \semi_\short(t_0,s;\ell)\Av f_s - \semi_\short(t_0, s; \ell) \Av F_t$ into an eigenvector moment observable part and an ansatz observable part.  Then both parts proceed in a similar fashion -- fill in missing terms of a tracial expression.
\begin{multline}\label{eq:l2jsum}
\frac{1}{N} \sum_{\substack{i \neq j \in \indint \\ |i-j| \leq \ell}} \frac{\one{n_i(\xx) > n_j(\xx) = 0}}{n_i(\xx)-1} \sum_{a \neq b} \one{x_a=x_b=i} \Im \frac{g_s(m_{ab}^{ij}\xx)}{\lambda_j(s) - z_i(s)}\\ 
=
\sum_{i \in \indint} \frac{\one{n_i(\xx) > 0}}{n_i(\xx)-1} \sum_{a \neq b} \frac{\one{x_a=x_b=i}}{N}\Im \sum_{0<|i-j|\leq \ell} \one{n_j(\xx) = 0} \frac{\semi_\short(t_0,s) (\Av f_{t_0})(m_{ab}^{ij}\xx)-\semi_\short(t_0,s) (\Av F_t)(m_{ab}^{ij}\xx)}{\lambda_j(s) - z_i(s)} 
\end{multline}

Now fix a site $i$ in the support of $\xx$ and particle labels $a,b \in [n]$ with $x_a=x_b=i$.  We compare the corresponding $j$-sums for both the $\semi_\short(t_0,s;\ell) \Av f_s$ and $\semi_\short(t_0,s;\ell) \Av F_t$ parts seperately and relate them to the ansatz observable.
For the $\semi_\short(t_0,s;\ell) \Av f_s$ part, begin with the estimate
\begin{align}
\semi_\short(t_0,s;\ell) (\Av f_s) (m_{ab}^{ij}\xx) &= \Av(m_{ab}^{ij}\xx) \semi_\short(t_0,s;\ell) f_{t_0}(m_{ab}^{ij}\xx) + O(\frac{N^\varepsilon\ell}{K}N^{\frac{n}{2}\ob}) \\
&= (\Av(\xx) + O(\frac{\ell}{K})) (f_s(m_{ab}^{ij}\xx) + O(\frac{N(s-t_0)}{\ell}N^{\frac{n}{2}\ob})) + O(\frac{N^\varepsilon\ell}{K}N^{\frac{n}{2}\ob}) \\
&= \Av(\xx) f_s(m_{ab}^{ij}\xx) + O(\frac{N^\varepsilon \ell}{K} + \frac{N(s-t_0)}{\ell} ) N^{\frac{n}{2}\ob}
\end{align}
where the first line used Corollary \ref{cor:FSPcomm}, the second line used Proposition \ref{prop:duhamel} and that $|\Av(\xx) - \Av(\zz)| \leq \|\xx-\zz\|_1/K$ for any $\xx,\zz \in \Ln$ while $\|\xx - m_{ab}^{ij}\xx\|_1 \leq 2\ell$, and the third line used delocalization Corollary \ref{cor:deloc} to say $|f_s(m_{ab}^{ij}\xx)| \leq N^{\frac{n}{2}\ob}$ and $\|\Av\|_\infty \leq 1$.  Using this approximation in the $j$-sum of \eqref{eq:l2jsum} gives
\begin{equation}
\frac{1}{N}\Im \sum_{\substack{0<|i-j|\leq \ell \\ n_j(\xx) = 0}} \frac{\semi_\short(t_0,s;\ell) (\Av f_{t_0})(m_{ab}^{ij}\xx)}{\lambda_j(s) - z_i(s)}
= \Av(\xx) \frac{1}{N}\sum_{\substack{0<|i-j| \leq \ell \\ n_j(\xx) = 0}} \Im \frac{f_s(m_{ab}^{ij}\xx)}{\lambda_j(s)-z_i(s)} + O(\frac{N^\varepsilon\ell}{K} + \frac{N(s-t_0)}{\ell})N^{\frac{n}{2}\ob}
\end{equation}
where we used that the Stieltjes transform is order 1 from Propositions \ref{prop:reg_fc} and \ref{prop:reg_eval} to factor the error term outside of the sum.  From here, fix $i,a,b$ and consider the sum over $j$.  The cost for extending the sum over all spectral indices $j \in [N]$ is
\begin{equation}
\left|\frac{1}{N}\sum_{j=1}^N \Im \frac{f_s(m_{ab}^{ij}\xx)}{\lambda_j(s) - z_i(s)} - \frac{1}{N}\sum_{0<|i-j|\leq \ell} \one{n_j(\xx)=0} \Im \frac{f_s(m_{ab}^{ij}\xx)}{\lambda_j(s) - z_i(s)}\right| < (\frac{1}{N\eta} + \frac{N\eta}{\ell})N^{\frac{n}{2}\ob}
\end{equation}
by Lemma \ref{lem:dyadic} and delocalization. Up to the delocalization contribution, this is the same bound as in \eqref{eq:l2_err2_near}.  However, using the definition of the colored eigenvector moment flow, this full sum can be approximated by a roughly deterministic Green's function entry along with an $(n-2)$-particle configuration.
\begin{multline}
\frac{1}{N}\sum_{j=1}^N \Im \frac{f_s(m_{ab}^{ij}\xx)}{\lambda_j(s) - z_i(s)} = \EE{\left. \frac{1}{N} \Im \sum_{j=1}^N \frac{\ipr{ \vv_a, \vu_j(s)} \ipr{\vu_j(s), \vv_b }}{\lambda_j(s) - z_i(s)} \pi(\xx \backslash ab)^{-1/2} \prod_{c\neq a,b} (\sqrt{N} \ipr{\vv_c, \vu_{x_c}}) \right| \bl,\dat} \\
= (\ipr{\vv_a, \Im \grn_{\fc,s}(z_i(s)) \vv_b} + O(\frac{N^\varepsilon}{\sqrt{N\eta}})) f_s(\xx \backslash ab; \Vecs^{(ab)}, \dat, \bl)
\end{multline}
where $\xx\backslash ab \in \Lambda^{n-2}$ is the particle configuration obtained by removing particles labeled $a$ and $b$.  The factor containing the $j$-sum inside the expectation is exactly the Green's function $\grn(s; z_i(s))$ at time $s$.  By the isotropic law in Proposition \ref{prop:isotropiclaw}, this factor is estimated by the free convolution analogue down to optimal scale with overwhelming probability and as the product factor grows only polynomially fast, we can ignore the complementary event.  The $\Lambda^{n-2}$ observable $f_s(\xx\backslash ab; \Vecs^{(ab)}, \dat, \bl)$ is the eigenvector moment observable on $\Lambda^{n-2}$ with test vectors $\Vecs^{(ab)} = (\vv_c)_{c\in[n]\backslash\{a,b\}}$.

Now consider the $\semi_\short(t_0,s;\ell)\Av F_t$ term.  Again, we start with the estimate
\begin{align}
\semi_\short(t_0,s;\ell) (\Av F_t) (m_{ab}^{ij}\xx) &= \Av(m_{ab}^{ij}\xx) \semi_\short(t_0,s;\ell)F_t(m_{ab}^{ij}\xx) + O(\frac{N^\varepsilon\ell}{K}N^{\frac{n}{2}\ob}) \\
&=(\Av(\xx) + O(\frac{\ell}{K})) F_t(m_{ab}^{ij}\xx) + O(\frac{N^\varepsilon\ell}{K}N^{\frac{n}{2}\ob}) \\
&= \Av(\xx) F_t(m_{ab}^{ij}\xx) + O(\frac{N^{\frac{n}{2}\ob + \varepsilon}\ell}{K})
\end{align}
where the first line used Corollary \ref{cor:FSPcomm}, the second line used that the ansatz observable lies in the kernel of $\gen_{ij}$ and $|\Av(\xx) - \Av(m_{ab}^{ij}\xx)| < 2\ell / K$ as before, and the third line uses $\|F_t\|_\infty \leq N^{\frac{n}{2}\ob}$.  Now using this approximation on the ansatz $j$-sum from \eqref{eq:l2jsum} gives
\begin{multline}
\frac{1}{N}\Im \sum_{0<|i-j|\leq \ell} \one{n_j(\xx) = 0} \frac{\semi_\short(t_0,s;\ell) (\Av F_t)(m_{ab}^{ij}\xx)}{\lambda_j - z_i} \\
= \Av(\xx) \frac{1}{N}\sum_{0<|i-j| \leq \ell} \one{n_j(\xx) = 0} \Im \frac{F_t(m_{ab}^{ij}\xx)}{\lambda_j-z_i} + O(\frac{\ell}{K} N^{\frac{n}{2}\ob+\varepsilon})
\end{multline}
where we used that the Stieltjes transform is order 1 to factor the error term outside of the sum.  From here, fix $i,a,b$ as we did with the moment observable term and consider the sum over $j$.  The cost for extending the sum over all spectral indices $j \in [N]$ is
\begin{equation}
\left|\frac{1}{N}\sum_{j=1}^N \Im \frac{F_t(m_{ab}^{ij_0}\xx)}{\lambda_j(s) - z_i(s)} - \frac{1}{N}\sum_{0<|i-j|\leq \ell} \one{n_j(\xx)=0} \Im \frac{F_t(m_{ab}^{ij}\xx)}{\lambda_j(s) - z_i(s)}\right| < (\frac{1}{N\eta} + \frac{N\eta}{\ell})N^{\frac{n}{2}\ob}
\end{equation}
where $j_0$ is any site with $n_j(\xx) = 0$.  This is because $F_t(m_{ab}^{ij}\xx)$ is invariant as $j$ varies over all such sites.  It also uses Lemma \ref{lem:dyadic} and that $Im (\lambda_j(s)-z_i(s))^{-1} \leq \eta^{-1}$ for all $i,j\in[N]$. This time however, the full sum is the Stieltjes transform along with the $n$-particle ansatz.
\begin{equation}
\frac{1}{N}\sum_{j=1}^N \Im \frac{F_t(m_{ab}^{ij_0}\xx)}{\lambda_j(s) - z_i(s)} = F_t(m_{ab}^{ij_0}\xx) \Im m_N(s; z_i(s)) = F_t(m_{ab}^{ij_0}\xx) \Im m_{\fc,s}(z_i(s)) + O(\frac{N^{\frac{n}{2}\ob + \varepsilon}}{N\eta})
\end{equation}
by Proposition \ref{prop:reg_eval} and Assumption \ref{ass:evec} implying $\|F_t\|\leq N^{\frac{n}{2}\ob}$.

Putting all these estimates together, the entire $g_s$ $j$-sum becomes
\begin{equation}\label{eq:l2dynamicend}
\begin{aligned}
&\frac{1}{N}\Im \sum_{0<|i-j|\leq \ell} \one{n_j(\xx) = 0} \frac{g_s(m_{ab}^{ij}\xx)}{\lambda_j(s) - z_i(s)} \\
=& \frac{1}{N}\Im \sum_{0<|i-j|\leq \ell} \one{n_j(\xx) = 0} \frac{\semi_\short(t_0,s;\ell) (\Av f_{t_0})(m_{ab}^{ij}\xx)-\semi_\short(t_0,s;\ell) (\Av F_t)(m_{ab}^{ij}\xx)}{\lambda_j(s) - z_i(s)} \\
=& \Av(\xx) \left( \ipr{\vv_a, \Im \grn_{\fc,s}(z_i(s)) \vv_b} f_s(\xx \backslash ab) - F_t(m_{ab}^{ij_0}\xx) \Im m_{\fc,s}(z_i) \right) \\
&+ O(\frac{N^\varepsilon\ell}{K} + \frac{N(s-t_0)}{\ell} + \frac{N^\varepsilon}{N\eta} +\frac{N\eta}{\ell} + \frac{N^{\varepsilon-\ob}}{\sqrt{N\eta}})N^{\frac{n}{2}\ob}
\end{aligned}
\end{equation}
The free convolution of the Green's function and Stieltjes transforms can be compared to their analogues at time $t$ using the deterministic bounds from Proposition \ref{prop:reg_Gm} and the fundamental theorem of calculus
\begin{equation}
|\ipr{\vv_a, (\Im (\grn_{\fc,s}(z_i(s)) - \grn_{\fc,t}(\gamma_i(t))) \vv_b}| \leq \frac{N^\ob}{s} (|t-s| + |z_i(s) - \gamma_i(t)|)
\end{equation}
by representing the difference as an integral over the line segment connecting $(s,z_i(s))$ and $(t,\gamma_i(t))$.  Moreover, the difference between spectral parameters is controlled with
\begin{equation}
|z_i(s) - \gamma_i(t)| \leq \lambda_i(s) - \gamma_i(s)| + |\gamma_i(s) - \gamma_i(t)| + \eta \leq \frac{N^\oc}{N} + \log(N)(t-s) + \eta
\end{equation}
by Propositions \ref{prop:reg_fc} and \ref{prop:reg_eval}, giving a final bound of
\begin{equation}\label{eq:l2_grn_perturb}
|\ipr{\vv_a, (\Im (\grn_{\fc,s}(z_i(s)) - \grn_{\fc,t}(\gamma_i(t))) \vv_b}| \leq \frac{N^\ob}{t_0}((t-t_0) + \eta + \frac{N^\oc}{N})
\end{equation}
for the Green's function terms.  The Stieltjes transform terms are estimated similarly.
\begin{equation}\label{eq:l2_mfc_perturb}
|m_{\fc,s}(z_i(s)) - m_{\fc,t}(\gamma_i(t))| \leq \frac{N^\ob}{s}(|t-s| + |z_i(s) -\gamma_i(t)|) \leq \frac{N^\ob}{t_0}((t-t_0) + \eta + \frac{N^\oc}{N})
\end{equation}
again accounting for the drift in the classical locations at speed $\log N$, single eigenvalue fluctuations of order $N^{\oc-1}$, and imaginary spectral shift from Propositions \ref{prop:reg_fc}, \ref{prop:reg_eval}, and \ref{prop:reg_Gm}. 
By induction, the $(n-2)$-particle observable converges to the ansatz observable
\begin{equation}\label{eq:l2_induct}
|f_s(\xx \backslash ab) - F_s(\xx \backslash ab; \xx\backslash ab) | \leq N^{-\od(n-2)}
\end{equation}
and the ansatz observable is sufficiently robust to perturbations in time and center
\begin{equation}\label{eq:l2_F_perturb}
|F_s(\xx \backslash ab; \xx\backslash ab)  - F_t(\xx\backslash ab; \yy\backslash ab)| = N^{(\frac{n}{2}-2)\ob} \frac{|\gamma_{x_a}(s) - \gamma_{y_a}(t)| + |s-t|}{t} \leq N^{(\frac{n}{2}-2)\ob} \left(\frac{K}{Nt} + \frac{t-t_0}{t}\log(N)\right)
\end{equation}
by interpolating between the two ansatz observables replacing one factor at a time in Definition \ref{def:ansatz} and using Proposition \ref{prop:reg_Gm}.
Therefore, by \eqref{eq:l2_grn_perturb}, \eqref{eq:l2_mfc_perturb}, \eqref{eq:l2_induct}, and \eqref{eq:l2_F_perturb}
\begin{equation}\label{eq:l2inductansatz}
\begin{aligned}
&\ipr{\vv_a, \Im \grn_{\fc,s}(z_i) \vv_b} f_s(\xx \backslash ab) - F_t(m_{ab}^{ij_0}\xx) \Im m_{\fc,s}(z_i) \\
=& \ipr{\vv_a, \Im \grn_{\fc,t}(\gamma_i(t)) \vv_b} F_t(\xx \backslash ab) - F_t(m_{ab}^{ij_0}\xx)\Im m_{\fc,t}(\gamma_i(t)) \\
&+ O(\frac{K}{Nt} + \frac{t-t_0}{t} + \eta + \frac{N^\oc}{N})N^{\frac{n}{2}\ob} + O(N^{-\od(n-2)+\ob}) \\
=& O(\frac{K}{Nt} + \frac{t-t_0}{t} + \eta + \frac{N^\oc}{N})N^{\frac{n}{2}\ob} + O(N^{-\od(n-2)+\ob})
\end{aligned}
\end{equation}
where the first two terms in the second line cancel by the definition of the ansatz observable from Definition \ref{def:ansatz}.
Combining \eqref{eq:l2inductansatz} and \eqref{eq:l2dynamicend} and noting that only finitely many triples $i,a,b$ satisfy $n_i(\xx) > 0$ and $x_a=x_b=i$ for each $\xx\in\Lambda^n$, finishes the proof of \eqref{eq:l2_err2_near}.
\end{proof}

\section{Energy method}\label{sec:energy}
The end goal of the energy method is to prove that the colored eigenvector moment flow dynamics satisfies an ultracontractive property $L^2 \rightarrow L^\infty$ with sufficient decay so that the bound from Proposition \ref{prop:l2} implies pointwise convergence of the colored eigenvector moment observable.  This is accomplished in three steps.

Step one is the Poincar\'e inequality showing that the global mixing time of the colored eigenvector moment flow is proportional to the side length of the configuration space.  This is accomplished via careful combinatorics and colored particle bookkeeping on the configuration space.  The $L^2$ deviation from equilibrium is controlled by the total energy divided by the ground state energy for our system.  The ground state energy grows as the side length grows so this bound is sharper for smaller systems.

Step two is the Nash inequality which utilizes a fine dissection of the configuration space.  The total $L^2$ deviation is now bounded by the sharper energy bound from the Poincar\'e inequality offset by an $L^1$ cost from dissecting.

Step three involves integrating the Nash inequality and using duality to obtain the desired ultracontractive bound.

\subsection{Poincar\'e Inequality}

In a sense, the Poincar\'e inequality is a converse to the Finite Speed of Propagation estimate.  Recall that Proposition \ref{prop:FSP} shows that spectral information travels between spatially separated eigenvalues with speed at most one.  On the other hand, the Poincar\'e inequality suggests that spectral information travels with at speed at least one.

Under this heuristic, it would take order one time for all regular eigenvector components to mix and reach \textit{global} equilibrium, because the regular eigenvalues all lie in some compact interval with length of order one since their typical spacing is order $N^{-1}$ by eigenvalue rigidity.  This is too slow for our purposes, however the heuristic also predicts faster mixing time to reach \textit{local} equilibrium.  In particular, $\ell$ nearby regular eigenvalues will share their spectral information amongst each other in time $t \geq \ell/N + N^\oc/N$, the length scale of the smallest interval containing all of these eigenvalues.  In the renormalized picture of configuration space, the colored eigenvector moment flow is a heavy tailed random walk on $\Ln$ which diffuses on a local neighborhood in time proportional to the radius of the neighborhood.  To make this precise, we first define quantities relevant to a local neighborhood.

\begin{definition}\label{defn:localnbh}
For every distinguishable configuration $\yy \in \Ln$ and every length scale $\ell > 0$, define the set of position pairs
\begin{equation}
R_0(\yy,\ell) = \{(i,j) \in [N]^2 | \mbox{ there exists } a \in [n] \mbox{ such that } |i-y_a| \leq \ell \mbox{ and } |j-y_a| \leq \ell\}
\end{equation}
and let $R(\yy,\ell) = \langle R_0(\yy,\ell) \rangle$ be the equivalence relation on $[N]$ generated by $R_0(\yy,\ell)$.  The symbol $\eql$ will often be used to replace $R(\yy,\ell)$ meaning $i \eql j$ if and only if $(i,j) \in R(\yy,\ell)$.

Define the \emph{local $\ell$-neighborhood centered at $\yy$} by
\begin{equation}
\Lyl = \{ \xx \in \Ln |  x_a \eql y_a \mbox{ for all } a \in [n]\}
\end{equation}
and the \emph{local inner product} by
\begin{equation}
\ipyl{ f, g } = \sum_{\xx \in \Lyl} \pi(\xx) f(\xx) g(\xx)
\end{equation}
for all $f,g \in L^2(\Lyl)$.
\end{definition}

\begin{remark}
To reduce notational confusion, we use separate conventions for two equivalence relations and point out contextual indicators for which is being referred to here.  Recall from Definition \ref{defn:partition}, the partition equivalence relation is a relation on $a,b\in[n]$ and the partition data is placed above the $\sim$, as in $a \eqp b$.
The local neighborhood equivalence relation is a relation on $i,j\in[N]$ and the neighborhood data is placed below the $\sim$, as in $i \eql j$.
\end{remark}

\begin{remark}
Note that if $\xx \in \Lyl$ and $i \eql j$, then for all $a \neq b \in [n]$, the two particle jump and swap operators satisfy $m_{ab}^{ij} \xx \in \Lyl$ and $s_{ab}^{ij}\xx \in \Lyl$.
Indeed, let $(x'_1, \ldots, x'_n)^\top \in \{m_{ab}^{ij}\xx, s_{ab}^{ij}\xx\}$.  
Then for all $a \in [n]$, $x'_a \eql x_a$ as either $x'_a = x_a$ or $\{x'_a, x_a\} = \{i,j\}$.  Also, $x_a \eql y_a$ since $\xx \in \Lyl$.  Transitivity of $\eql$ implies $x'_a \eql y_a$.
In particular, this allows us to make the following definition.
\end{remark}

\begin{definition}
The \emph{local Dirichlet form} is the quadratic form on $L^2(\Lyl)$ defined by the two following equivalent expressions
\begin{equation}
\dyl(s; f) = \sum_{i \eql j} c_{ij}(s) \ipyl{ f, (-\gen_{ij}) f } = \frac{1}{2} \sum_{\xx \neq \zz \in \Lyl} \gen_{\xx\zz}(s) |f(\xx) - f(\zz)|^2
\end{equation}
for any observable $f \in L^2(\Lyl)$.
\end{definition}

An equivalent interpretation is that the \emph{spectral gap} or \emph{ground state energy} of the system is inversely proportional to the the side length of the system.  The Poincar\'e inequality is stated as $L^2$ deviation from equilibrium is upper bounded by the multiplicity of ground state energies belonging to the system (total energy divided by ground state energy).  However, the equilibrium is often difficult to work with (Lemma \ref{lem:kernel}).  For this reason, we introduce the following fake projection operator using an extended language of partitions alluded to in Section \ref{sec:colorlattice}.

\begin{definition}[Position partitions]\label{defn:verticalpartitions}
Given a distinguishable particle configuration $\xx \in \Ln$, let the \emph{position partition} associated with $\xx$ be $\px = \{\{a \in [n] | x_a = i\} | i \in [N]\}$.  That is, two labels $a,b \in [n]$ belong to the same part of $\px$ if and only if particles $a$ and $b$ lie at the same position: $a \eqpx b$ if and only if $x_a = x_b$.
\end{definition}

Equiped with the partition notation from Definitions \ref{defn:partition} and \ref{defn:verticalpartitions}, we are ready to define the makeshift projection operator.

\begin{definition}\label{defn:fakeproj}
Recalling the ordering on partitions from Definition \ref{defn:partition}, define the \emph{local projection operator} by
\begin{equation}
\pyl f(\xx) = \frac{\sum_{\zz \in \Lyl} \pi(\zz) \one{\pz \leq \px} f(\xx)}{\sum_{\zz \in \Lyl}  \pi(\zz) \one{\pz \leq \px}}
\end{equation}
\end{definition}

\begin{remark}
Unlike the other local definitions which share most analogous properties with their global counterparts, the operator $\pyl$ is not the orthogonal projection onto $\bigcap_{i \eql j} \ker(\gen_{ij}|_{L^2(\Lyl)})$, although it is our desire for $\pyl$ to approximate some operator of that form.  In fact $\pyl$ is not even self-adjoint.  For some indication of the relevance of $\pyl$, it will become apparent that $\pyl$ and $\kproj$ share the same $1$-eigenspace (restricted to the local neighborhood).  It is not hard to check that 
\begin{equation}\ker(1-\pyl) = \vspan \{\chi_\sigma|_{\Lyl} | \sigma \in M_n\}\end{equation} where the stratum indicators $\chi_\sigma$, $\sigma \in M_n$ were introduced in \ref{lem:globalker}.
For instance, suppose $\pyl f = f \in L^2(\Lyl)$. For every partition $\pP$ whose parts are all size two, $f$ must be constant on the stratum $\{\xx \in \Lyl | \px = \pP\}$.  Then recurse up the lattice of partitions.
\end{remark}

For the mixing time estimate proposed by the Poincar\'e inequality to hold, we require only the following structure on generator coefficients $\{c_{ij}(s) | 1 \leq i < j \leq \ell\}$ at time $s \geq 0$.
\begin{assumption}\label{ass:heavytail}
Assume the coefficients $c_{ij}(s)$ in the operator $\gen_s = \sum_{i<j} c_{ij}(s) \gen_{ij}$ satisfy subquadratic decay with rate $\upsilon > 0$,  $c_{ij}(s) > \upsilon |i-j|^{-2}$ for all $i < j \in [N]$.
\end{assumption}

At last, we are prepared to state the  main result for this section.  Throughout the proof of the Poincar\'e inequality, there are several minor computations and facts which are not difficult to derive but in our opinion clutter the argument.  For this organizational purpose, these proofs are postponed to the next section.
\begin{proposition}[Poincar\'e lemma]\label{prop:poincare}
If $\gen_s$ satisfies Assumption \ref{ass:heavytail} with rate $\upsilon > 0$, then there exists a constant $\pconst=\pconst(n) > 0$ depending only on $n$ such that
\begin{equation}\label{eq:poincare}
\sum_{\xx \in \Lyl} \pi(\xx) |f(\xx) - \pyl f(\xx)|^2 \leq \pconst \frac{\ell}{\upsilon} \dyl(f)
\end{equation}
uniformly in $\yy \in \Ln$, $\ell > 0$, and $f \in L^2(\Lyl)$.
\end{proposition}

\begin{proof}
For any partition $\pP$ of $[n]$, consider the following \emph{$\pP$-conditional expectation operator} $\ep : L^2(\Ln) \rightarrow L^2(\Ln)$ defined by
\begin{equation}\label{eq:defn_ep}
\ep f(\xx) = \frac{1}{|\symP|} \sum_{\sigma \in \symP} f(\sigma \actn \xx)
\end{equation}
where $\symP \leq S_n$ is the subgroup of $\pP$-compatible permutations introduced in Definition \ref{defn:partition}.
According to Lemma \ref{lem:EPidentity}, the proof of which is deferred to the next section, the $\pP$-conditional expectation operators satisfy the following two envelope identities.  For every $\xx \in \Lyl$, 
\begin{equation}\label{eq:envelope}
\epx f(\xx) = f(\xx) \quad \mbox{and} \quad (\pyl\epx)f(\xx) = \pyl f(\xx)
\end{equation}
the proof of which we postpone to Lemma \ref{lem:EPidentity}.
Apply these identities to each term appearing in the left hand side of \eqref{eq:poincare}.
\begin{equation}\label{eq:app_envelope}
\sum_{\xx \in \Lyl} \pi(\xx) |f(\xx) - \pyl f(\xx)|^2 = \sum_{\xx \in \Lyl} \pi(\xx) |\epx f(\xx) - (\pyl \epx)f(\xx)|^2
\end{equation}
Treating the local projection operator $\pyl$ as a weighted averge over $\zz \in \Lyl$ satisfying $\pz \leq \px$, apply Jensen's inequality to bound the left hand side of the Poincar\'e inequality resembling variance by the average squared deviation of $\pP$-expectations between pairs of certain sites.
\begin{multline}\label{eq:poincarepairs}
\sum_{\xx \in \Lyl} \pi(\xx) |\epx f(\xx) - (\pyl \epx)f(\xx)|^2 
\leq \ell^{-n/2} \sum_{\xx \in \Lyl} \sum_{\zz \in \Lyl} \one{\pz \leq \px} |\epx f(\xx) - \epx f(\zz)|^2
\end{multline}
Here we used that $|\{\zz \in \Lyl | \pz \leq \px\}| \asymp \ell^{n/2}$ uniformly in $\xx$.  

To simplify the argument, introduce some new notation.  First is $\LylP$ as the configuration subspace in the local neighborhood consisting of all distinguishable configurations whose particle configurations are sufficiently compatible with the partition $\pP$, $\LylP = \{\xx \in \Lyl | \px \leq \pP\}$.  This can be thought of as the closure of an open neighborhood of the $\pP$ stratum in distinguishable configuration space.  

Second, consider the \emph{maximal local partition} $\pl$ given by its defining property $a \eqpl b$ if and only if $y_a \eql y_b$.  The name is inspired by the fact that for all $\xx \in \Lyl$, $\px \leq \pl$.  To see this, suppose $x_a = x_b$ for some $a,b\in[n]$.  Then $y_a \eql x_a = x_b \eql y_b$ so $y_a \eql y_b$ by transitivity.

Using these new terms, relax and symmetrize \eqref{eq:poincarepairs} by summing over all local pairs $(\xx,\zz)$ compatible with a common partition that is a refinement of $\pl$.
\begin{equation}\label{eq:sympairs}
\sum_{\xx \in \Lyl} \sum_{\zz \in \Lyl} \one{\pz \leq \px} |\epx f(\xx) - \epx f(\zz)|^2 \leq \sum_{\pP \leq \pl} \sum_{\xx,\zz \in \LylP} |\ep f(\xx) - \ep f(\zz)|^2
\end{equation}
where the first summation is taken over all refinements $\pP \leq \pl$.  To manipulate the expression on the right hand side into the Dirichlet form, we employ a path counting argument.

For each $\pP \leq \pl$, further enrich $\LylP$ with a graph structure by defining the edge sets $\edgel = \edgem \cup \edgee$ where pairs $(\ww,\ww') \in \LylP$ are edges when the following conditions are met.
\begin{itemize}
\item $(\ww,\ww') \in \edgem$ if and only if there exist $i \eql j \in [N]$ and $a \eqp b \in [n]$ such that $\ww' = m_{ab}^{ij}\ww \neq \ww$ and
\item $(\ww,\ww') \in \edgee$ if and only if there exist $i \eql j \in [N]$ such that $\ww' = (ij)\actN\ww \neq \ww$.  Here $(ij)$ is the $(i,j)$-transposition in $S_N$ and $\actN$ is the $S_N$ action defined in Definition \ref{defn:actN}.
\end{itemize}

For any fixed partition $\pP \leq \pl$, the permutation action $\symP$ on $\Ln$ respects the graph structure on $\LylP$ in the following sense.  For all local $\pP$-configurations $\ww, \ww' \in \LylP$ with $(\ww,\ww') \in \edgel$ and $\pP$-compatible permutations $\sigma \in \symP$, we have $\sigma \actn \ww \in \LylP$ and $(\sigma \actn \ww, \sigma \actn \ww') \in \edgel$.  Again, the proof is simple after following definitions, but is postponed to Lemma \ref{lem:lpgraph} to avoid clutter.

Therefore, $\symP$ acts on the graph $\LylP$.  The content of the next lemma is that the quotient graph $\LylP/\symP$ is sufficiently expanding, has diameter $n/2$, and is nearly regular with all vertices having degree on order $\ell$.  See Lemma \ref{lem:pathcounting} for a restatement and proof.

\begin{lemma}[Path counting]
There exists a constant $C_1=C_1(n)>0$ depending only on $n$ for which every partition $\pP \leq \pl$ and local $\pP$-configurations $\xx,\zz \in \LylP$ admit a positive integer $0 \leq \len \leq C$ and a path $\family(\xx,\zz) = (\family^{\xx\zz}_0, \ldots, \family^{\xx\zz}_{\len}) \in (\LylP)^{\len}$ of length $\len$ such that the resulting family of paths $\family = (\family(\xx,\zz) | \xx,\zz \in \LylP)$ satisfy the following properties.
\begin{enumerate}
\item For all $\xx,\zz \in \LylP$, the path $\family(\xx,\zz)$ satisfies $\family^{\xx\zz}_0 = \xx$, $\family^{\xx\zz}_{\len} = \zz$, and for every $i \in [\len]$ there exists $\sigma(\xx,\zz,i) \in \symP$ such that $(\sigma(\xx,\zz,i) \actn \family^{\xx\zz}_i, \family^{\xx\zz}_{i-1}) \in \edgel$.
\item There are $O(\ell^{n/2-1})$ paths incident to each edge:
\begin{equation}
|\{(\xx,\zz) \in \LylP | \{\family^{\xx\zz}_{i-1}, \family^{\xx\zz}_i\} = \{\sigma \actn \ww, \sigma' \actn \ww'\} \mbox{ for some } \sigma, \sigma' \in \symP, i \in [\len]\}| \leq C_1 \ell^{n/2-1}
\end{equation}
uniformly over all edges $(\ww,\ww') \in \edgel$.
\item There are at $O(1)$ edges incident to each path:
\begin{equation}
|\{(\ww,\ww') \in \edgel | \{\family^{\xx\zz}_{i-1}, \family^{\xx\zz}_i\} = \{\sigma \actn \ww, \sigma' \actn \ww'\} \mbox{ for some } \sigma, \sigma' \in \symP, i \in [\len]\}| \leq C_1
\end{equation}
uniformly over all pairs of points $\xx, \zz \in \LylP$.
\end{enumerate}
\end{lemma}

Taking the path counting lemma for granted, the squared differences across entire paths appearing in \eqref{eq:sympairs} can be bounded by the telescoping sum of squared differences along each edge in the path for all origin-destination pairs $\xx, \zz \in \LylP$
\begin{equation}
|\ep f(\xx) - \ep f(\zz)|^2 \leq \len \sum_{i=1}^{\len} | \ep f(\family^{\xx\zz}_{i-1}) - \ep f(\family^{\xx\zz}_i)|^2
\end{equation}
by the Schwarz inequality.  Taking this one step further, $\ep$ equalizes $f$ along all $\symP$-orbits so the right hand side can be rewritten in a more convoluted manner using the replacement
\begin{equation}
\ep f(\family^{\xx\zz}_i) = \ep f(\sigma(\xx,\zz,i)\actn\family^{\xx\zz}_i)
\end{equation}
for all $\xx,\zz\in\LylP$ and $i \in [\len]$.
Since $\len \leq C_1$ and each edge appears in at most $C_1 \ell^{n/2-1}$ paths in $\family$ (connecting all pairs $\xx,\zz \in \LylP$ for each $\pP \leq \pl)$, the summation over pairs in \eqref{eq:sympairs} can be replaced by a summation over edges
\begin{multline}\label{eq:edgesum}
\sum_{\xx,\zz\in\LylP} | \ep f(\xx) - \ep f(\zz)|^2 
\leq \sum_{\xx,\zz\in\LylP} \len \sum_{i=1}^\len | \ep f(\family^{\xx\zz}_{i-1}) - \ep f(\sigma(\xx,\zz,i)\actn\family^{\xx\zz}_i)|^2 \\
\leq C_1^2 \ell^{n/2-1} \sum_{(\ww,\ww') \in \edgel} | \ep f(\ww) - \ep f(\ww')|^2
\end{multline}
uniformly over all refinements $\pP \leq \pl$.

We pause for a moment here to record our progress as well as breifly describe the plan forward.  Bounds from \eqref{eq:app_envelope}, \eqref{eq:poincarepairs}, \eqref{eq:sympairs}, and \eqref{eq:edgesum} imply
\begin{equation}\label{eq:Poincare_pause}
\sum_{\xx\in\Lyl} \pi(\xx)|f(\xx)-\pyl f(\xx)|^2 \leq \ell^{-1} \sum_{\pP\leq\pl}\sum_{(\ww,\ww')\in\edgel}|\ep f(\ww)-\ep f(\ww')|^2
\end{equation}
It remains to recover the local Dirichlet form from these squared differences across edges.  As there are only finitely many partitions of $[n]$ independent of $N$ (trivially bounded by $n^n$ since functions $[n] \rightarrow [n]$ induce all partitions of $[n]$ via preimages), in order to deduce the Poincar\'e inequality with constant $\pconst = n^n C_1^2 C_2$, it suffices to show that there exists a constant $C_2=C_2(n) > 0$ depending only on $n$ such that
\begin{equation}\label{eq:Poincare_goal}
\sum_{(\ww,\ww')\edgel}|\ep f(\ww)-\ep f(\ww')|^2\leq C_2\frac{\ell^2}{\upsilon}\dyl(s;f)
\end{equation}
uniformly over all refinements $\pP\leq\pl$.
This is done by decomposing the local Dirichlet form as the sum of many two-site inner products which we introduce now.  The argument requires insight to the kernel and finite dimensionality of $\gen_{ij}$.  For the remainder of the proof, fix a refinement $\pP \leq \pl$.

For every pair of sites $i,j\in[N]$ and pair of distinguishable configurations $\xx,\yy\in\Ln$, write $\xx \eqij \yy$ to mean that for all $a\in[n]$ either $x_a=y_a$ or $\{x_a,y_a\}=\{i,j\}$.  Note that $\eqij$ is an equivalence so the parts $\Lij = \{\yy \in \Lambda^n | \yy \eqij \xx\}$ partition $\Ln$.  These parts will be referred to as \emph{two-site subspaces} and we will see that three relevant properties of these two-site subspaces push the proof forward.

The first property of interest for two-site subspaces is that each edge of type 1 is contained in a two-site subspace which is contained in the local $\pP$-neighborhood.  That is, for all $(\ww,\ww')\in\edgem$, there exists a local pair of sites $i \eql j$ and a configuration $\xx\in\LylP$ such that $\ww,\ww'\in\Lij$ and $\Lij\subset\LylP$.  Also, for every $i \eql j$, the two-site subspaces $\Lij$ partition $\Lyl$ in the sense that $\Lij\subset\Lyl$ whenever $\xx\in\Lyl$, which implies 
\begin{equation}
\Lyl = \bigsqcup_{\xx\in\Lyl/\eqij} \Lij
\end{equation}
where $\xx$ is taken over a set of representatives, one for each equivalence class in $\Lyl/\eqij$.
See Lemma \ref{lem:two-site} in the next section for proofs.

This observation allows us to group the summation over type 1 edges in \eqref{eq:Poincare_goal} according to which two-site subspace the edge belongs.  It further allows us to require the two-site subspace be entirely contained in $\LylP$.
\begin{multline}\label{eq:two-site-complete}
\sum_{(\ww,\ww') \in \edgem} |\ep f(\ww) - \ep f(\ww')|^2 \\
= \sum_{i \eql j} \sum_{\xx \in \Lyl/\eqij} \one{\Lij \subset \LylP} \sum_{\ww,\ww'\in\Lij} \one{(\ww,\ww')\in\edgem} |\ep f(\ww) - \ep f(\ww')|^2
\end{multline}
where the second summation is taken over representatives $\xx \in \LylP$, one for each equivalence class in $\LylP/\eqij$.

The second important property of two-site subspaces we will use is that they are invariant under the generators corresponding to their two sites.
That is, for all $i < j \in [N]$, the operators $\move_{ij}$, $\exch_{ij}$, and hence $\gen_{ij}=\move_{ij}+\exch_{ij}$ each split along the orthogonal decomposition of invariant subspaces 
\begin{equation}
L^2(\Ln) = \bigoplus_{\xx\in\Ln/\eqij} L^2(\Lij)
\end{equation}
where the direct sum is taken over a set of representatives $\xx \in \Ln$, one for each equivalence class in $\Ln/\eqij$.  Here, $L^2(\Lij) \subset L^2(\Ln)$ is the space of functions $f \in L^2(\Ln)$ which vanish off of $\Lij$, $f(\xx)=0$ when $\xx \in \Ln\backslash\Lij$.  This is proved in Lemma \ref{lem:two-site-split}.  Denote the inner-product over the two-site subspace by 
\begin{equation}
\ipij{f,g} = \sum_{\xx\in\Lij} \pi(\xx)f(\xx)g(\xx)
\end{equation}
for any $f,g\in L^2(\Lij)$.

This property is utilized as follows.  Consider the move operator restricted to a corresponding two-site subspace, $\move_{ij}|_{L^2(\Lij)}$ for some $\xx\in\Ln$ and $i<j\in[N]$ satisfying $\Lij \subset \LylP$.  This operator generates a random walk on the induced subgraph with edges taken only from the edge set $\edgem$, $(\Lij, \edgem|_{\Lij})$.  Since this graph is finite and connected, the kernel of $\move_{ij}|_{L^2(\Lij)}$ consists precisely of the constant functions.  Since the graph is finite with size independent of $N$, there exist constants $C_2^{(1)}(\xx,i,j) > 0$ such that for all configurations $\xx \in \Ln$ and sites $i < j \in [N]$
\begin{equation}\label{eq:two-site-move}
\sum_{\ww,\ww'\in\Lij} \one{(\ww,\ww')\in\edgem} |g(\ww) - g(\ww')|^2 \leq C_2^{(1)}(\xx,i,j) \ipij{g, (-\move_{ij}) g}
\end{equation}
for any $g \in L^2(\Lij)$.  For completeness, \eqref{eq:two-site-move} is a consequence of Lemma \ref{lem:kernel_form_bound}.  Moreover, as $\move_{ij}$ acts isometrically on each $\Lij$ up to a finite choice of $\xx,i,j$ (whenever the graphs $(\Lij, \edgem|_{\Lij})$ are isomorphic), there exists an upper bound 
\begin{equation}\label{eq:isomorphism}
0 < \sup_{\substack{\xx\in\Ln\\i<j\in[N]}} C_2^{(1)}(\xx,i,j) = C_2^{(1)} < \infty.
\end{equation}

The third property of two-site subspaces is that the composition of the exchange operator with the $\pP$-conditional expectation acts only between two-site subspaces and not within an individual two-site subspace.  That is, if $\Lij \subset \LylP$, then $(\exch_{ij}\ep)|_{L^2(\Lij)} = 0$.  This is proved in Lemma \ref{lem:two-site-exch}.  Combining this observation with \eqref{eq:two-site-move} and \eqref{eq:isomorphism} implies that
for all $\xx \in \Lyl$ and $i \eql j$, 
\begin{equation}
\one{\Lij \subset \LylP} \sum_{\ww,\ww'\in\Lij} |\ep f(\ww) - \ep f(\ww')|^2 \leq C_2^{(1)} \ipij{\ep f, (-\gen_{ij}) \ep f}
\end{equation}
where the positive definiteness of $-\gen_{ij}$ from Lemma \ref{lem:definite} deals with the case $\Lij\not\subset\LylP$.  Plugging this bound into \eqref{eq:two-site-complete} yields
\begin{equation}\label{eq:pairs_to_gen1}
\sum_{(\ww,\ww')\in\edgem} |\ep f(\ww) - \ep f(\ww')|^2 \leq C_2^{(1)} \sum_{i \eql j} \sum_{\xx \in \LylP/\eqij} \ipij{\ep f, (-\gen_{ij}) \ep f}
\end{equation}
uniformly over all refinements $\pP\leq\pl$ which concludes our analysis of type 1 edges.

The goal for type 2 edges will be analogous to \eqref{eq:pairs_to_gen1}, but the argument will be slightly different.  This time, if $(\ww,\ww')\in\edgee$, then there exist a local pair of sites $i \eql j$ and a configuration $\xx\in\Lyl$ such that $\ww,\ww'\in\Lij$.  Again, see Lemma \ref{lem:two-site} for the proof.  Note that $\Lij\subset\LylP$ is not guaranteed this time around.  Nevertheless, this observation still allows us to group the type 2 edge terms in \eqref{eq:Poincare_goal} according to which two-site subspace the edge belongs.
\begin{equation}\label{eq:two-site-complete2}
\sum_{(\ww,\ww') \in \edgee} |\ep f(\ww) - \ep f(\ww')|^2
= \sum_{i \eql j} \sum_{\xx \in \Lyl/\eqij} \sum_{\ww,\ww'\in\Lij} \one{(\ww,\ww')\in\edgee} |\ep f(\ww) - \ep f(\ww')|^2
\end{equation}
Rather than splitting the generator into move and exchange terms $\gen_{ij}=\move_{ij}+\exch_{ij}$ and analyzing the kernels of both components individually as done in the type 1 edge case, the relevant kernel information is already prepared for us from Section \ref{sec:jump}.  If a type 2 edge $(\ww,\ww')\in\edgee$ is contained in a two-site subspace $\ww,\ww'\in\Lij$, then $\ww' = (ij)\actN\ww$ where $i \eql j$ are the variable sites in the two-site subspace.  In this case,
$g(\ww)=g(\ww')$
for every $g \in \ker(\gen_{ij})$ by Lemma \ref{cor:spatial_kernel}.  Again, the finite dimensionality of $L^2(\Lij)$ implies that for every configuration $\xx\in\Ln$ and pair of local sites $i \eql j$ there exists a constant $C_2^{(2)}(\xx,i,j) > 0$ such that
\begin{equation}
\sum_{\ww,\ww'\in\Lij} \one{(\ww,\ww')\in\edgee} |g(\ww)-g(\ww')|^2 \leq C_2^{(2)}(\xx,i,j) \ipij{g, (-\gen_{ij}) g}
\end{equation}
for every $g \in L^2(\Lij)$.
See Lemma \ref{lem:kernel_form_bound}. Moreover, as $\move_{ij}$ acts isometrically on each $\Lij$ up to a finite choice of $\xx,i,j$, there exists an upper bound 
\begin{equation}
0 < \sup_{\substack{\xx\in\Ln\\i<j\in[N]}} C_2^{(2)}(\xx,i,j) = C_2 < \infty.
\end{equation}
Using thes bounds for each inner sum in \eqref{eq:two-site-complete2} gives
\begin{equation}\label{eq:pairs_to_gen2}
\sum_{(\ww,\ww') \in \edgee} |\ep f(\ww) - \ep f(\ww')|^2
= \sum_{i \eql j} \sum_{\xx \in \Lyl/\eqij} C_2^{(2)} \ipij{ \ep f, (-\gen_{ij}) \ep f}
\end{equation}
Letting $C_2 = C_2^{(1)} + C_2^{(2)}$, the sum of \eqref{eq:pairs_to_gen1} and \eqref{eq:pairs_to_gen2} is exactly
\begin{equation}\label{eq:pairs_to_gen}
\sum_{(\ww,\ww') \in \edgel} |\ep f(\ww) - \ep f(\ww')|^2
= \sum_{i \eql j} \sum_{\xx \in \Lyl/\eqij} C_2 \ipij{\ep f, (-\gen_{ij}) \ep f}
\end{equation}
which concludes our analysis on the edges.

Another application of item 1 from Lemma \ref{lem:two-site} and Lemma \ref{lem:two-site-split} tells us that for all $i \eql j$,
\begin{equation}\label{eq:gen_ijyl}
\sum_{\xx\in\Lyl/\eqij} \ipij{\ep f, (-\gen_{ij}) \ep f} = \ipyl{\ep f, (-\gen_{ij}) \ep f}
\end{equation}
as $L^2(\Lyl) = \oplus_{\xx\in\Lyl/\eqij} L^2(\Lij)$. Each two-site generator commutes with the $\pP$-conditional expectation operator as proved in Lemma \ref{lem:comm_genep}.  In particular, the right hand side can be rewritten as
\begin{equation}\label{eq:gen_commute}
\ipyl{\ep f, (-\gen_{ij}) \ep f} = \ipyl{\ep f, \ep (-\gen_{ij}) f}
\end{equation}
Moreover, the $\pP$-conditional expectation operator is shown to be an orthogonal projection on the local function space $L^2(\Lyl)$ in Lemma \ref{lem:Pexp_projection}.  Together with the Lemma \ref{lem:definite}, this implies
\begin{equation}\label{eq:cond_project}
\ipyl{\ep f, \ep (-\gen_{ij}) f} = \ipyl{f, \ep(-\gen_{ij})f} \leq \ipyl{f, (-\gen_{ij}) f}.
\end{equation}
Combining \eqref{eq:pairs_to_gen}, \eqref{eq:gen_ijyl}, \eqref{eq:gen_commute}, and \eqref{eq:cond_project} gives
\begin{equation}
\sum_{(\ww,\ww')\in\edgel} |\ep f(\ww) - \ep f(\ww')|^2 \leq C_2 \sum_{i \eql j} \ipyl{f, (-\gen_{ij})f}
\end{equation}
Lastly, appeal to Assumption \ref{ass:heavytail} to see that $(\ell^2/\upsilon)c_{ij}(s) \geq 1$ for all $i \eql j$, $i \neq j$ and obtain our goal \eqref{eq:Poincare_goal},
\begin{equation}
\sum_{(\ww,\ww')\in\edgel}|\ep f(\ww) - \ep f(\ww')|^2 \leq C_2 \sum_{i \eql j} \frac{\ell^2}{\upsilon} c_{ij}(s) \ipyl{f, (-\gen_{ij})f} = C_2 \frac{\ell^2}{\upsilon}\dyl(s;f)
\end{equation}
which concludes the proof of the Poincar\'e inequality.
\end{proof}

\subsection{Auxiliary results for the Poincar\'e Inequality}

Throughout this section, we will refer to the maximal local partition defined in the proof of Proposition \ref{prop:poincare}.

\begin{lemma}[Maximal local partition $\pl$]\label{lem:max_local_part}
Let $\pl$ be the maximal local partition defined in the proof of Proposition \ref{prop:poincare}.  That is, $\pl$ is given by its defining property: $a \eqpl b$ if and only if $y_a \eql y_b$.  Then for all $\xx\in\Lyl$, $\px\leq\pl$.
\end{lemma}

\begin{proof}
Suppose $\xx\in\Lyl$ and that $a \eqp b$ for some $a,b\in[n]$.  Then $y_a \eql x_a = x_b \eql y_b$ where both $\eql$ are by the definition of $\Lyl$ and the equality is by the definition of $\px$.  Therefore, $y_a \eql y_b$ by transitivity of $\eql$ so $\px\leq\pl$.
\end{proof}

\begin{lemma}[Compatibility of the $S_n$ action]\label{lem:Sn_compat}
The restrictions of the $S_n$ action on configuration space preserve the following local and color structures.
\begin{itemize}
\item If $\xx \in \Lyl$ and $\sigma \in \symPl$, then $\sigma\actn\xx\in\Lyl$.
\item If $\xx\in\Ln$, $\pP$ a partition of $[n]$, and $\sigma\in\symP$ with $\px\leq\pP$, then $\pP_{\sigma\actn\xx}\leq\pP$.
\end{itemize}
\end{lemma}

\begin{proof}
For the first item, suppose $\xx\in\Lyl$ and $\sigma\in\symPl$.  Then for all $a\in[n]$, $x_{\sigma(a)} \eql y_{\sigma(a)} \eql y_a$ where the first $\eql$ is from $\xx\in\Lyl$ and the second $\eql$ is from $\sigma\in\symPl$.  Therefore, $\sigma\actn\xx\in\Lyl$.

For the second item, suppose $\xx\in\Ln$ with $\px\leq\pP$ and $\sigma\in\symP$.  Then for all $a \eqpx b \in[n]$, $x_{\sigma(a)}=x_{\sigma(b)}$ so $\sigma(a) \eqpx \sigma(b)$ by Definition \ref{defn:verticalpartitions}.  Thus, $a \eqp \sigma(a) \eqp \sigma(b) \eqp b$ where the outer $\eqp$ are from $\sigma\in\symP$ and the middle $\eqp$ is from $\px\leq\pP$.
\end{proof}

\begin{lemma}[Envelope identities]\label{lem:EPidentity}
For all configurations $\xx \in \Lyl$ and test functions $f \in L^2(\Lyl)$, the following two identities hold
\begin{equation}
\epx f(\xx) = f(\xx)
\end{equation}
where $\ep$ is defined in \eqref{eq:defn_ep} and
\begin{equation}
\pyl \epx f(\xx) = \pyl f(\xx)
\end{equation}
where $\pyl$ is defined in Definition \ref{defn:fakeproj}.
\end{lemma}

\begin{proof}
Let $\sigma \in \symPx$.  For all $a \in[n]$, $x_a \eqpx x_{\sigma(a)}$ by Definition \ref{defn:partition}, so $x_a = x_{\sigma(a)}$ by Definition \ref{defn:verticalpartitions}.  This means that $\sigma \actn \xx = \xx$ for all $\sigma \in \symPx$.  Therefore, the first identity becomes
\begin{equation}
\epx f(x) = \frac{1}{|\symPx|} \sum_{\sigma \in \symPx} f(\sigma\actn\xx) = \frac{1}{|\symPx|} \sum_{\sigma \in \symPx} f(\xx) = f(\xx).
\end{equation}
The second identity is a consequence of the first together with the fact that whenever $\pP\leq\pl$ is a refinement of the maximal partition, $\pyl$ and $\ep$ commute.  Recall that $\px\leq\pl$ by Lemma \ref{lem:max_local_part}.  To ease notation, we use the notion of local $\pP$-neighborhoods to write the local projection as
\begin{equation}\label{eq:local_proj_px}
\pyl f(\xx) = \frac{1}{\pi(\LylPx)} \sum_{\zz\in\LylPx} \pi(\zz)f(\zz).
\end{equation}
Expanding the commutator gives
\begin{multline}\label{eq:second_envelope}
\pyl \ep f(\xx) - \ep \pyl f(\xx) \\
= \frac{1}{|\symP|}\sum_{\sigma\in\symP} \left( \frac{1}{\pi(\LylPx)}\sum_{\zz\in\LylPx} \pi(\sigma\actn\zz)f(\sigma\actn\zz) - \frac{1}{\pi(\LylPsx)}\sum_{\zz\in\LylPsx} \pi(\zz)f(\zz) \right)
\end{multline}
For any $\pP$-compatible permutation $\sigma \in \symP$ and local configuration $\xxt\in\Lyl$, consider the mapping 
\begin{equation}
F_{\sigma,\xxt} : \Lyl(\pP_{\tilde x}) \rightarrow \Lyl(\pP_{\sigma\actn\tilde{x}}) \quad \mbox{defined by} \quad \zz \mapsto \sigma\actn\zz
\end{equation}
Firstly, $F_{\sigma,\xxt}$ is well-defined because $\zz\in\Lyl$ and $\pP\leq\pl$ imply $\sigma\actn\zz\in\Lyl$ by Lemma \ref{lem:Sn_compat}.
Secondly, $F_{\sigma,\xxt}$ is a bijection because $F_{\sigma^{-1},\sigma\actn\xxt}$ is its inverse.  Thirdly, the map is $\pi$-measure preserving because the colorblind image of $\xxt$ and $\sigma\actn\xxt$ coincide (recall that $\pi(\xx)$ depends only on the particle numbers $n_i(\xx)$ and not on specific labels).  Therefore, the inner sums in \ref{eq:second_envelope} match by the change of variables $F_{\sigma,\xx}$ and so do their normalizing coefficients.
\end{proof}

\begin{remark}
The first identity is a triviality as the $\symPx$-orbit of $\xx$ is always a singleton.  The second identity is a type of Fubini's theorem since $\pyl$ can by thought of as an expectation over configurations with sufficiently compatible position partitions whereas $\ep$ is an expectation over configurations in $\symP$-orbits.  It requires only that the $\symP$ action commutes with the $\zz$-sampling.
\end{remark}

\begin{lemma}[$\symP$ action on $\LylP$]\label{lem:lpgraph}
For every partition $\pP\leq\pl$ of $[n]$ which is a refinement of the maximal local partition $\pl$ as defined in the proof of Proposition \ref{prop:poincare}, all edges $(\ww,\ww')\in\edgel$, and all $\pP$-compatible permutations $\sigma \in \symP$, the following two conditions are satisfied
\begin{equation}
\sigma\actn\ww\in\LylP
\end{equation}
and
\begin{equation}
(\sigma\actn\ww,\sigma\actn\ww')\in\edgel.
\end{equation}
\end{lemma}

\begin{proof}
The first condition follows from both parts of Lemma \ref{lem:Sn_compat}.  The second condition requires the additional observation that
\begin{equation}\label{eq:localcommute}
\sigma\actn(m_{\sigma(a)\sigma(b)}^{ij}\xx) = m_{ab}^{ij}(\sigma\actn\xx) \quad \mbox{and} \quad 
\sigma\actn(s_{\sigma(a)\sigma(b)}^{ij}\xx) = s_{ab}^{ij}(\sigma\actn\xx) 
\end{equation}
for every distinguishable configuration $\xx\in\Ln$, permutation $\sigma\in S_n$, pair of sites $i,j\in[N]$, and pair of labels $a,b\in[n]$.  These can be checked directly from the definitions of the two-particle jump and swap operators $m_{ab}^{ij},s_{ab}^{ij}$ in Theorem \ref{thm:levmf} and the $S_n$ action on $\Ln$ in Definition \ref{def:actn}.
\end{proof}

\begin{lemma}[Path counting]\label{lem:pathcounting}
For every $\pP\leq\pl$, there exists a family of lengths $\{\len \in \ZZ_{\geq 0}| \xx,\zz \in \LylP \}$, a family of configurations $\{\family^{\xx\zz}_i \in \LylP | \xx,\zz \in \LylP, 0 \leq i \leq \len\}$, and a constant $C(n)>0$ with the following properties:
\begin{enumerate}
\item For every $\xx,\zz \in \LylP$, $\family^{\xx\zz}_0 = \xx$ and $\family^{\xx\zz}_{\len} = \zz$.
\item For every $\xx,\zz \in \LylP$ and $1 \leq i \leq \len$, there exists $\sigma\in\symP$ such that $(\family^{\xx\zz}_{i-1}, \sigma\actn\family^{\xx\zz}_i) \in \edgel$.
\item For every $\xx,\zz \in \LylP$, $\len \leq C$.
\item For every $(\ww,\ww') \in \edgel$, 
\begin{equation}
|\{(\xx,\zz,i) \in \LylP^2 \times \ZZ_{>0} | 1 \leq i \leq \len, \family^{\xx\zz}_{i-1}=\ww, \family^{\xx\zz}_i = \ww'\}| \leq C\ell^{n/2-1}.
\end{equation}
\end{enumerate}
\end{lemma}

\begin{proof}
Define the $\pP$-deviation between $\xx$ and $\zz$ by 
\begin{equation}
\pdev(\xx,\zz) = \inf_{\sigma\in\symP} |\{a \in [n]| x_{\sigma(a)} \neq z_a\}|
\end{equation}
for any $\xx,\zz \in \Lyl$, not necessarily requiring $\px,\pz\leq\pP$.  Now fix $\xx,\zz \in \LylP$. I will provide an iterative construction of $\family^{xz}_i$, $i=0,1,2\ldots$, and then prove the four properties.
\begin{enumerate}
\item Set $\family^{\xx\zz}_0 = \xx$ and introduce a counter $m=0$. Iterate the following procedure.
\item At this point, $\family^{\xx\zz}_i$ is defined for $0 \leq i \leq m$.  If $\pdev(\family^{\xx\zz}_m, \zz) = 0$, then set $\len = m$, redefine $\family^{\xx\zz}_m = \zz$, and we are done.
\item Otherwise, if there exists $\ww \in \LylP$ such that $(\family^{\xx\zz}_m, \ww) \in \edgem$ and $\pdev(\ww,\zz) = \pdev(\family^{\xx\zz}_m, \zz) - 2$, then set $\family^{\xx\zz}_{m+1} = \ww$.
\item I claim that if no such $w$ from the previous step exists, then there exists $\ww' \in \LylP$ such that $(\family^{\xx\zz}_m, \ww') \in \edgee$ and $\pdev(\ww',\zz) \leq \pdev(\family^{\xx\zz}_m, \zz) - 2$.  In this case set $\family^{\xx\zz}_{m+1}=\ww'$.
\item Regardless, at this point we have defined $\family^{\xx\zz}_{m+1}$.  Return to step 2 and increment $m$ by $1$.
\end{enumerate}
This iterative procedure is guaranteed to terminate because the $\pP$-deviation $\pdev(\family^{xz}_i, z)$ is strictly decreasing in $i$ at each iteration. 

To verify construction's validity, it remains to prove the claim from step 4.  Let's call $\family^{\xx\zz}_m = \xx'$ to ease notation.  If $\pdev(\xx',\zz) > 0$, then there always exists $a \neq b \in [n]$, and $i \eql j$ such that
\begin{equation}
x'_a = x'_b = i \quad \mbox{and} \quad \pdev(m_{ab}^{ij} \xx', \zz) = \pdev(\xx', \zz) - 2
\end{equation}
However, $m_{ab}^{ij}\xx' \not \in \LylP$ or else $\ww=m_{ab}^{ij}\xx'$ would have worked in step 3.  As $m_{ab}^{ij}\xx' \in \Lyl$, it must be the case that $\pP_{m_{ab}^{ij}\xx'}\not\leq\pP$.  The only possibility now is that there exists $c \in [n]$ such that $c \not\eqp a$ with $x'_c = j$.  Since moving $a$ and $b$ to $j$ would reduce the $\pP$-deviation of $\xx'$ with $\zz$, there is no $\sigma \in \symP$ such that $z_{\sigma(c)} = j$.  This means moving the contents of $x'$ at site $j$ elsewhere will not increase the $P$-deviation with $z$, and there exists $k \neq j \in \ZZ$ and $d \in [n]$ with 
\begin{equation}
x'_d = j \quad \mbox{and} \quad \pdev(m_{cd}^{jk} \xx', \zz) = \pdev(\xx', \zz) - 2.
\end{equation}
However, $m_{cd}^{jk}\xx' \not \in \LylP$ or else $\ww = m_{cd}^{jk}\xx'$ would have worked in step 3.  As $m_{cd}^{jk}\xx'\in\Lyl$, it must be the case that $\pP_{m_{cd}^{jk}\xx'}\not\leq\pP$.  The only possibility now is that there exists $e \in [n]$ such that $e \not\sim_P c$ with $x'_e = k$.  Since moving $c$ and $d$ to $k$ would reduce the $\pP$-deviation of $\xx'$ with $\zz$, there is no $\sigma \in \symP$ such that $z_{\sigma(e)} = k$.  In particular, moving the contents of site $j$ to site $k$ will reduce the $\pP$-deviation of $\xx'$ with $\zz$ by at least two, while moving the contents of site $k$ to site $j$ will not increase the $\pP$-deviation of $\xx'$ with $\zz$.  In other words,
\begin{equation}
\pdev((jk)\actN\xx', \zz) \leq \pdev(\xx',\zz) - 2
\end{equation}
so $\ww' = (jk)\actN\xx'$ is a valid choice for step 4.  Note that in this case, $(\ww',\xx')\in\edgee$, $(jk)\in\S_N$ is the $(i,j)$-transposition, and $\actN$ is the $S_N$ action defined in Definition \eqref{defn:actN}.

We now proceed to check that this construction satisfies the desired properties.
Property 1 is immediate from steps 1 and 2.  Steps 3 and 4 ensure Property 2 holds for $i < \len$.  In step 2, we replaced $\family^{\xx\zz}_{\len}$ with $\zz = \sigma\actn\family^{\xx\zz}_\len$ for some $\sigma \in \symP$.  However, for the original $\family^{\xx\zz}_\len$, we had $(\family^{\xx\zz}_\len, \sigma' \sigma \actn \zz) = (\family^{\xx\zz}_{\len-1}, \sigma'\actn\family^{\xx\zz}_\len) \in \edgel$ for some $\sigma'\in\symP$.  Since $\sigma' \sigma \in \symP$, property 2 holds for $i=\len$ as well.  To see Property 3, the iterative procedure concludes after at most $\pdev(\xx,\zz)/2 \leq n/2$ steps because the $\pP$-deviation decreases by at least 2 in each iteration, $\pdev(\family^{\xx\zz}_i, \zz) \leq \pdev(\family^{\xx\zz}_{i-1}, \zz) - 2$, and $\pP$-deviations are at most $n$ by definition.  This shows that $\len \leq n/2$ for all $\xx, \zz \in \LylP$.  

Lastly, for Property 4, observe that in the graph $(\LylP, \edgel)$ the degree of each vertex is at most $( n^4/8 + n^2/4)\ell$.  Indeed, for each $\ww\in\LylP$ there are at most:
\begin{itemize}
\item $n/2$ choices for $i\in[N]$ such that $n_i(\ww)>0$, 
\item $n^2/2$ choices for labels $a$ and $b$, 
\item $n/2 \ell$ choices for $j \eql i$
\end{itemize}
such that $(\ww,m_{ab}^{ij}\ww)\in\edgem$.  Similarly, there are at most
\begin{itemize}
\item $n/2$ choices for $i\in[N]$ such that $n_i(\ww)>0$,
\item $n/2 \ell$ choices for $j \eql i$
\end{itemize}
such that $(\ww,(ij)\actN\ww)\in\edgee$.  Therefore, the number of paths of length $n/2$ passing through a specified edge is bounded by $n/2((n^4/8 + n^2/4) \ell)^{n/2-1}$.
\end{proof}

\begin{lemma}[$\Lij$ partition and edge containments]\label{lem:two-site}
The following containment implications regarding local configurations, edges, partitions, and two-site subspaces hold.
\begin{enumerate}
\item If $\xx\in\Lyl$ and $i \eql j$, then $\Lij \subset \Lyl$.  As a consequence, the local neighborhood is partitioned by two-site subspaces.  That is, for all $i \eql j$
\begin{equation}
\Lyl = \bigsqcup_{\xx\in\Lyl/\eqij} \Lij
\end{equation}
where $\xx$ ranges over representatives, one for each equivalence class of $\Lyl/\eqij$.
\item If $(\ww,\ww')\in\edgem$ is a type 1 edge, then there exist $i \eql j$ and $\xx \in \Lyl$ such that $\ww,\ww' \in \Lij$ and $\Lij \subset \LylP$.
\item If $(\ww,\ww')\in\edgee$ is a type 2 edge, then there exist $i \eql j$ and $\xx \in \Lyl$ such that $\ww,\ww'\in\Lij$.
\end{enumerate}
\end{lemma}

\begin{proof}
For item 1, suppose $\zz\eqij\xx$ and $\xx\in\Lyl$.  Then for all $a\in[n]$, $z_a \in \{x_a,(ij)(x_a)\}$ but in either case $z_a \eql x_a \eql y_a$.  Thus, $\zz\in\Lyl$.

For item 2, there exists $i\eql j$ and labels $a \eqp b$ such that $\ww'=m_{ij}^{ab}\ww\eqij\ww$.  Suppose $\xx\eqij\ww$ and claim that $\px\leq\pww\leq\pP$.  Note that for all $c\in[n]$, $z_c\in\{i,j\}$ implies $c \eqpww a$ and $z_c\not\in\{i,j\}$ implies $w_c=z_c$.  Therefore, whenever $c \eqpx d$, $z_c=z_d$ so $c \eqpww a \eqpww d$.

For item 3, there exists $i \eql j$ such that $\ww'=(ij)\actN\ww\eqij\ww$.
\end{proof}
%
%

\begin{lemma}[Splitting of $\gen_{ij}$]\label{lem:two-site-split}
Let $i < j \in [N]$ be sites and $\xx \in \Ln$ a distinguishable configuration.  Then operators $\move_{ij}$ and $\exch_{ij}$ each split along the orthogonal decomposition of invariant subspaces
\begin{equation}
L^2(\Lambda^n) = \bigoplus_{\xx\in\Ln/\eqij} L^2(\Lij).
\end{equation}
That is, if $f(\zz) = 0$ for all $\zz \in \Ln \backslash \Lij$ then $\move_{ij} f(\zz) = 0$ for all $\zz \in \Ln \backslash \Lij$, and analogously for $\exch_{ij}$.
\end{lemma}

\begin{proof}
This is a basic consequence of the definition of the move and exchange operators $\move_{ij}$ and $\exch_{ij}$ and the observation that the two particle jump and swap operators preserve two-site subspaces: for all $a,b\in[n]$
\begin{equation}\label{eq:two_site_moveswap}
m_{ab}^{ij}\xx \eqij s_{ab}^{ij}\xx \eqij \xx
\end{equation}
which can be checked immediately from the definition of $\eqij$.
\end{proof}

\begin{lemma}[Vanishing exchange]\label{lem:two-site-exch}
Consider a refinement $\pP\leq\pl$, a local $\pP$-configuration $\xx \in \LylP$, local sites $i \eql j \in [N]$ satisfying $\Lij\subset\LylP$.  Then 
\begin{equation}
(\exch_{ij}\ep)|_{L^2(\Lij)} = 0.
\end{equation}
\end{lemma}

\begin{proof}
Suppose $a,b\in[n]$ with $s_{ab}^{ij}\xx \neq \xx$. Then it must be the case that $x_a=i$ and $x_b=j$.  Since $n_i(\xx)\geq1$ and is even, there must be another label $c\in[n]\backslash\{a,b\}$ such that $x_c=i$.  In this case, we have $a \eqpx c$ and $c \eqpsx b$.  Since $\xx,s_{ab}^{ij}\xx\in\Lij$ by \eqref{eq:two_site_moveswap} and $\Lij\subset\LylP$ by assumption, we have $\px,\pP_{s_{ab}^{ij}\xx}\leq\pP$ so $a \eqp b$ by transitivity of $\eqp$.  Therefore, the transposition $(ab) \in S_n$ swapping labels $a$ and $b$ is in fact compatible with $\pP$, $(ab)\in\symP$.  Furthermore, $s_{ab}^{ij}\xx=(ab)\actn\xx$.
Therefore, for all $f \in L^2(\Lij)$
\begin{equation}
\exch_{ij} \ep f(\xx) = 2 \sum_{a \neq b \in [n]} \left( \ep f(s_{ab}^{ij}\xx) - \ep f(\xx) \right) = \sum_{a \neq b \in [n]} \frac{2}{|\symP|} \left( \sum_{\sigma\in\symP}f(\sigma\actn(s_{ab}^{ij}\xx)) - \sum_{\sigma\in\symP}f(\sigma\actn\xx) \right) = 0
\end{equation}
which vanishes because $s_{ab}^{ij}\xx$ and $\xx$ lie in the same $\symP$-orbit.
This concludes the proof since $\xx$ was an arbitrary representative of the two-site subspace $\Lij$.
\end{proof}

\begin{remark}
The idea behind this proof is to consider the induced subgraph on vertices $\Lij\subset\LylP$ with edges coming only from $\edgee$, $(\Lij, \edgee|_{\Lij})$.  This graph is not connected.  Each connected component is contained in a $\symP$-orbit so $f$ is equalized on each connected component by $\ep$.  On the other hand, the operator the operator $\edge_{ij}$ generates a random walk on each connected component and hence annihilates functions constant along connected components.  Therefore, the composition $\exch_{ij}\ep$ annihilates all functions on $\Lij$.
\end{remark}

\begin{lemma}[Finite dimensional kernel / quadratic form bound]\label{lem:kernel_form_bound}
Suppose $\mat$ is a real symmetric $m \times m$ matrix with the property that for all $\vv=(v^1, \ldots, v^m)^\top \in \ker(\mat)$, $v^1=v^2$.  Then
\begin{equation}
|w^1-w^2|^2 \leq \frac{2}{\mu} \ipr{\ww, \mat \ww}
\end{equation}
for all $\ww=(w^1, \ldots, w^m)^\top \in \RR^m$ where $\mu$ is the the spectral gap of $\mat$, that is $\mu$ is the smallest nonzero eigenvalue.
\end{lemma}

\begin{proof}
Let $\mat = \sum_{k=1}^{\rank(\mat)}\mu_k\vu_k\vu_k^\top$ be the spectral decomposition of $\mat$ and let $\ve = \ve_1-\ve_2$ which is in the orthogonal complement to $\ker(\mat)$, by assumption.  Hence, $\ve \in \vspan\{\vu_1,\ldots,\vu_{\rank(\mat)}\}$ which spectrally means $\mat \geq \frac{\mu}{\|\ve\|^2}\ve\ve^\top$ since $\mu = \min\{\mu_k | k\in[\rank(\mat)]\}$.  In other words
\begin{equation}
|w^1-w^2|^2 = \ipr{\vw, \ve\ve^\top \vw} \leq \frac{2}{\mu}\ipr{\vw, \mat \vw}
\end{equation}
for every $\vw\in\RR^m$.
\end{proof}

\begin{lemma}[Generator and $\pP$-expectation commutation]\label{lem:comm_genep}
For all sites $i < j \in [N]$ and all partitions $\pP$ of $[n]$, the commutator between the $\pP$-conditional expectation operator and the two-site generator $\gen_{ij}$ vanishes
\begin{equation}
[\gen_{ij}, \ep]=0.
\end{equation}
\end{lemma}

\begin{proof}
Expanding the terms in the commutator $[\gen_{ij}, \ep]f(\xx)$ gives
\begin{equation}
\gen_{ij}\ep f(\xx) = \begin{aligned}[t]
\frac{n_j(\xx)+1}{n_i(\xx)-1} &\sum_{a \neq b \in [n]} \left( \frac{1}{|\symP|}\sum_{\sigma\in\symP} f(\sigma\actn(m_{ab}^{ij}\xx)) - \frac{1}{|\symP|}\sum_{\sigma\in\symP} f(\sigma\actn\xx) \right) \\
+ \frac{n_i(\xx)+1}{n_j(\xx)-1} &\sum_{a \neq b \in [n]} \left( \frac{1}{|\symP|}\sum_{\sigma\in\symP}f(\sigma\actn(m_{ab}^{ji}\xx)) - \frac{1}{|\symP|}\sum_{\sigma\in\symP}f(\sigma\actn\xx) \right) \\
- 2 &\sum_{a \neq b \in [n]} \left( \frac{1}{|\symP|}\sum_{\sigma\in\symP} f(\sigma\actn(s_{ab}^{ij}\xx)) - \frac{1}{|\symP|}\sum_{\sigma\in\symP} f(\sigma\actn\xx) \right)
\end{aligned}
\end{equation}
and
\begin{equation}
\ep\gen_{ij} f(\xx) = \frac{1}{|\symP|}\sum_{\sigma\in\symP} \left(\begin{aligned}
\frac{n_j(\xx)+1}{n_i(\xx)-1} &\sum_{a \neq b \in [n]} \left( f(m_{ab}^{ij}(\sigma\actn\xx)) - f(\sigma\actn\xx) \right) \\
+ \frac{n_i(\xx)+1}{n_j(\xx)-1} &\sum_{a \neq b \in [n]} \left( f(m_{ab}^{ji}(\sigma\actn\xx)) - f(\sigma\actn\xx) \right) \\
- 2 &\sum_{a \neq b \in [n]} \left( f(s_{ab}^{ij}(\sigma\actn\xx)) - f(\sigma\actn\xx) \right)
\end{aligned}\right)
\end{equation}
All $f(\sigma\actn\xx)$ terms cancel out in the difference, so after rearranging summations, the commutator becomes
\begin{equation}
[\gen_{ij},\ep]f(\xx) = \frac{1}{|\symP|}\sum_{\sigma\in\symP}\left( \begin{aligned}
\frac{n_j(\xx)+1}{n_i(\xx)-1} &\sum_{a \neq b \in [n]} \left( f(m_{ab}^{ij}(\sigma\actn\xx)) - f(\sigma\actn(m_{ab}^{ij}\xx))\right) \\
+ \frac{n_i(\xx)+1}{n_j(\xx)-1} &\sum_{a \neq b \in [n]} \left( f(\sigma\actn(m_{ab}^{ji}\xx)) - f(\sigma\actn(m_{ab}^{ji}\xx)) \right) \\
- 2 &\sum_{a \neq b \in [n]} \left( f(\sigma\actn(s_{ab}^{ij}\xx)) - f(\sigma\actn(s_{ab}^{ij}\xx)) \right)
\end{aligned} \right)
\end{equation}
Now apply the change of variables $(a,b)\mapsto(\sigma\actn a, \sigma \actn b)$ to the second summation in every line and factor out the resulting $a \neq b$ terms.
\begin{equation}
[\gen_{ij},\ep]f(\xx) = \frac{1}{|\symP|}\sum_{\sigma\in\symP}\sum_{a \neq b \in [n]}\left( \begin{aligned}
\frac{n_j(\xx)+1}{n_i(\xx)-1} & \left( f(\sigma\actn(m_{ab}^{ij}\xx)) - f(\sigma\actn(m_{\sigma(a)\sigma(b)}^{ij}\xx))\right) \\
+ \frac{n_i(\xx)+1}{n_j(\xx)-1} & \left( f(\sigma\actn(m_{ab}^{ji}\xx)) - f(\sigma\actn(m_{ab}^{ji}\xx)) \right) \\
- 2 & \left( f(\sigma\actn(s_{ab}^{ij}\xx)) - f(\sigma\actn(s_{ab}^{ij}\xx)) \right)
\end{aligned} \right)
\end{equation}
Every summand vanishes at this point according to the identities from \eqref{eq:localcommute}.
\end{proof}

\begin{lemma}[$\pP$-expectation projection]\label{lem:Pexp_projection}
For any refinement $\pP\leq\pl$, the $\pP$-conditional expectation operator $\ep$ is an orthogonal projection on $L^2(\Lyl)$.
\end{lemma}

\begin{proof}
Firstly, $\ep : L^2(\Lyl) \rightarrow L^2(\Lyl)$ is well-defined by the first item of Lemma \ref{lem:Sn_compat}.
Secondly, we show that $\ep$ is symmetric.  For any $\xx,\yy\in\Ln$,
\begin{equation}
\ipn{\delta_x, \ep \delta_y} = \frac{1}{|\symP|}\sum_{\sigma\in\symP} \delta_\yy(\sigma\actn\xx) = \one{y\in\symP\actn\xx} \frac{|\Stab_{\symP}(\xx)|}{\pi(\xx)}.
\end{equation}
If $\yy \in \symP\actn\xx$, then $\Stab_{\symP}(\xx)=\Stab_{\symP}(\yy)$ and $\pi(\xx)=\pi(\yy)$.  Otherwise, the above expression vanishes.  Regardless, the expression is symmetric in $\xx$ and $\yy$.  Thirdly, $(\ep)^2=\ep$ because
\begin{equation}
(\ep)^2f(\xx) = \frac{1}{|\symP|^2} \sum_{\sigma,\sigma'\in\symP} f(\sigma\sigma'\actn\xx) = \frac{1}{|\symP|}\sum_{\sigma\in\symP} f(\sigma\actn\xx) = \ep f(\xx)
\end{equation}
for any $f\in L^2(\Ln)$ and $\xx \in \Ln$ by counting the number of times $f(\sigma\actn\xx)$ appears in the first sum for every $\sigma\in\symP$.
\end{proof}

\subsection{Nash Inequality}
The Poincar\'e inequality shows us that the ground state energy is bounded by the side length of the configuration space.  In particular, on a smaller configuration space the ground state energy is lower so the bound of the $L^2$ norm by ground state energy times total energy becomes sharper.  To apply this to our setting, rather than simply scaling the configuration space and using a weak ground state energy, the configuration space is dissected into many small neighbohoods.  The $L^2$ deviation from global equilibrium is now controlled by the energy contributions from each neighborhood combined with a dissection cost proportional to the $L^1$ deviation between local and global equilibrium.  As the dissection becomes finer, the local energy bounds become sharper, although the dissection cost grows rapidly.  Optimizing over the coarseness of the dissection produces the following relation between $L^1$ deviation, $L^2$ deviation, and the Dirichlet form -- the Nash inequality.

\begin{proposition}\label{prop:nash}
There exists a constant $\nconst = \nconst(n) > 0$ depending only on $n$ such that if $\gen_s$ satisfies Assumption \ref{ass:heavytail}, then
\begin{equation}
\upsilon\|f-\kproj f\|_2^{2+\frac{4}{n}} \leq \nconst \dir_s(f) \|f\|_1^{\frac{4}{n}}
\end{equation}
uniformly over all $f \in L^1(\Ln)$.  Here $\kproj$ is the global kernel projection operator defined in Lemma \ref{lem:kernel}.
\end{proposition}

\begin{proof}
For each length scale $1 \leq \ell \leq N$, consider the net of distinguishable particle configuraions supported on the $\ell\ZZ$-lattice, $\net = \{\yy\in\Ln | y_a \in \ell[N/\ell] \mbox{ for all }a\in[n]\}\subset[N]$ where the site of any particle $y_a, a\in[n]$ is divisible by $\ell$.  The net $\net$ is uniformly distributed in $\Ln$ in the sense that there exists a constant $c_1>0$ depending only $n$ such that
\begin{equation}\label{eq:goodnet}
1 \leq |\{\yy \in \net | \xx \in \Lyl \}| \leq c_1
\end{equation}
for all $\xx \in \Ln$ and all $1 \leq \ell \leq N$.  
Let $\pyl$ be the projection operator defined in Definition \ref{defn:fakeproj}.  Then
\begin{equation}\label{eq:nashdecomp}
\|f-\kproj f\|_2^2 \leq 2 c_1 \sum_{\yy \in \net} \left(\sum_{\xx\in\Lyl} \pi(\xx) |f(\xx) - \pyl f(\xx)|^2 + \sum_{\xx\in\Lyl}\pi(\xx)|\pyl f(\xx) - \kproj f(\xx)|^2 \right)
\end{equation}
by the net property \eqref{eq:goodnet} and Schwarz.  
The Poincar\'e lemma is used to control the left sum in \eqref{eq:nashdecomp} by the Dirichlet form.  By Proposition \ref{prop:poincare}, for all $\yy$ and $\ell$,
\begin{equation}\label{eq:poincare_in_nash}
\sum_{\xx\in\Lyl} \pi(\xx)|f(\xx)-\pyl f(\xx)|^2 \leq \pconst\frac{\ell}{\upsilon} \dyl(s;f) = \pconst\frac{\ell}{\upsilon} \sum_{i \eql j} c_{ij}(s) \sum_{\xx\in\Lyl/\eqij} \ipij{f, (-\gen_{ij}) f}
\end{equation}
where in the second sum $\xx$ is taken over representatives of equivalence classes in $\Lyl/\eqij$.
As each triple $(i,j,\xx) \in [N]^2\Ln$ satisfies $i \eql j$ and $\xx \in \Lyl$ for at most $c_1$ configurations $\yy\in\net$, the $\yy$ sum is bounded by the Dirichlet form as follows.
\begin{multline}\label{eq:nash_dir}
\sum_{\yy \in \net} \sum_{\xx\in\Lyl} \pi(\xx) |f(\xx) - \pyl f(\xx)|^2 \leq \pconst\frac{\ell}{\upsilon} \sum_{\yy\in\net} \sum_{i < j\in[N]} \sum_{\xx\in\Lyl/\eqij} \one{i \eql j} \one{\xx\in\Lyl} c_{ij}(s) \ipij{f, (-\gen_{ij}) f} \\
\leq c_1 \pconst\frac{\ell}{\upsilon} \sum_{i<j\in[N]} c_{ij}(s) \ipn{f, (-\gen_{ij}) f} = c_1\pconst\frac{\ell}{\upsilon} \dir_s(f)
\end{multline}

Schwarz is used to control the right sum in \eqref{eq:nashdecomp} by the $L^1$ norm.
\begin{equation}\label{eq:nash_localglobal}
\sum_{\xx\in\Lyl} \pi(\xx) |\pyl f(\xx) - \kproj f(\xx)|^2 \leq 2\sum_{\xx\in\Lyl} \pi(\xx) |\pyl f(\xx)|^2 + 2\sum_{\xx\in\Lyl} \pi(\xx) |\kproj f(\xx)|^2
\end{equation}
For the local projection operator, we use the neighborhood following counts.  There exists a constant $c_2>1$ depending only on $n$ such that
\begin{equation}\label{eq:strata_count}
\frac{1}{c_2}\ell^{n/2} \leq \pi(\LylP) \leq \pi(\Lyl) \leq c_2 \ell^{n/2}
\end{equation}
for all refinements $\pP\leq\pl$ uniformly over all $\yy\in\net$ and $1 \leq \ell \leq N$.
This count is used in conjuction with the local projection representation from \eqref{eq:local_proj_px}.  For all $\yy\in\net$ and $1\leq\ell\leq N$,
\begin{multline}\label{eq:nash_local_1}
\sum_{\xx\in\Lyl} \pi(\xx) |\pyl f(\xx)|^2 = \sum_{\xx\in\Lyl} \pi(\xx) \left|\pi(\LylPx)^{-1}\sum_{\zz\in\LylPx}\pi(\zz)f(\zz)\right|^2 \\
\leq c_2^2 \ell^{-n} \sum_{\xx\in\Lyl} \pi(\xx) \left|\sum_{\zz\in\LylPx}\pi(\zz) f(\zz)\right|^2 
\end{multline}
where the inequality uses \eqref{eq:strata_count} to factor out the normalizing coefficient.  The next step is to ensure all inner summands are positive so that the inner sum may be extended to the larger local neighborhood. For all $\xx\in\Lyl$,
\begin{equation}\label{eq:nash_local_2}
\left|\sum_{\zz\in\LylPx}\pi(\zz) f(\zz)\right|^2 
\leq \left(\sum_{\zz\in\LylPx}\pi(\zz)|f(\zz)|\right)^2 
\leq \left(\sum_{\zz\in\Lyl}\pi(\zz)|f(\zz)|\right)^2
\end{equation}
where the first inequality is the triangle inequality and the second inequality is $\LylPx\subset\Lyl$.  Note that this last expression is independent of the configuration $\xx$.  Therefore, the full sum can be compared to the total weight of $\Lyl$ which we have a bound on.  For all $\yy\in\net$ and $1 \leq \ell \leq N$, 
\begin{equation}
\sum_{\xx\in\Lyl} \pi(\xx) |\pyl f(\xx)|^2 \leq c_2^2\ell^{-n}\pi(\Lyl)\left( \sum_{z\in\Lyl} \pi(\zz)|f(\zz)|\right)^2 \leq c_2^3 \ell^{-n/2} \left( \sum_{z\in\Lyl} \pi(\zz)|f(\zz)|\right)^2
\end{equation}
The first inequality is the conclusion from combining \eqref{eq:nash_local_1} and \eqref{eq:nash_local_2}.  The second inequality is another application of \eqref{eq:strata_count}.  Finally, by summing over $\yy\in\net$,
\begin{equation}
\sum_{\yy\in\net}\sum_{\xx\in\Lyl}\pi(\xx)|\pyl f(\xx)|^2 \leq \sum_{\yy\in\net}c_2^3 \ell^{-n/2} \left( \sum_{\zz\in\Lyl} \pi(\zz)|f(\zz)|\right)^2
\leq c_2^3\ell^{-n/2}\left(\sum_{\yy\in\net} \sum_{\zz\in\Ln} \pi(\zz)|f(\zz)|^2\right)^2
\end{equation}
by bringing the $\yy\in\net$ summation inside the square.  An application of \eqref{eq:goodnet} gives
\begin{equation}\label{eq:nash_local}
\sum_{\yy\in\net}\sum_{\xx\in\Lyl}\pi(\xx)|\pyl f(\xx)|^2 \leq c_1^2c_2^3\ell^{-n/2}\|f\|_1^2
\end{equation}
for all $1 \leq \ell \leq N$ by our choice of net.

We now want a similar $L^1$ bound on the global kernel sum in \eqref{eq:nash_localglobal}.  For all $\yy\in\net$ and $1 \leq \ell \leq N$,
\begin{equation}
\sum_{\xx\in\Lyl} \pi(\xx) |\kproj f(\xx)|^2 \leq \pi(\Lyl) \|\kproj f\|_\infty^2 \leq \sgconst^n n^{n/2} c_2\ell^{n/2} N^{-n} \|f\|_1^2
\end{equation}
by the $L^1\rightarrow L^\infty$ bound in Corollary \ref{cor:kproj_bound} and stratum count in \eqref{eq:strata_count} where $\sgconst$ is the universal constant introduced in Lemma \ref{lem:subgaussian}.  As this bound is independent of $\yy\in\net$, we may use the fact that the size of the net is bounded 
\begin{equation}
|\net| \leq n!! \lfloor N/\ell \rfloor^{n/2}
\end{equation}
to take the sum over $\yy\in\net$ and get that for all $1 \leq \ell \leq N$,
\begin{equation}\label{eq:nash_global}
\sum_{\yy\in\net}\sum_{\xx\in\Lyl} \pi(\xx)|\kproj f(\xx)|^2 \leq |\net| \sgconst^n n^{n/2} c_2\ell^{n/2}N^{-n}\|f\|_1^2 \leq n!!\sgconst^n n^{n/2} c_2 N^{-n/2}\|f\|_1^2.
\end{equation}

Letting $C = \max\{c_1\pconst, c_1^2c_2^3, n!!c_2\sgconst^n n^{n/2}\}$ and combining equations \eqref{eq:nashdecomp}, \eqref{eq:nash_dir}, \eqref{eq:nash_localglobal}, \eqref{eq:nash_local}, and \eqref{eq:nash_global} gives
\begin{equation}\label{eq:nash_length}
\|f-\kproj f\|_2^2 \leq C \frac{\ell}{\upsilon}\dir_s(f) + C(\ell^{-n/2} + N^{-n/2}) \|f\|_1^2.
\end{equation}
for any choice of $1 \leq \ell \leq N$.
Consider the quantity which we refer to as the \emph{virtual length scale} given by
\begin{equation}\label{eq:virtual_length}
\ell_0 = \left( \frac{\upsilon\|f\|_1^2}{\dir_s(f)} \right)^{\frac{1}{1+n/2}}
\end{equation}
If the virtual length scale satisfies $1 \leq \ell_0 \leq N$, then set $\ell = \ell_0$ and observe that \eqref{eq:nash_length} reduces to the bound
\begin{equation}
\|f-\kproj f\|_2^2 \leq 3C \nu^{\frac{-n/2}{1+n/2}} \dir_s(f)^{\frac{1}{1+2/n}} \|f\|_1^{\frac{2}{1+n/2}}
\end{equation}
which after rearranging exponents and moving $\nu$ to the left hand side becomes
\begin{equation}\label{eq:virtual_good}
\nu \|f-\kproj f\|_2^{2+4/n} \leq (3C)^{1+2/n} \dir_s(f) \|f\|_1^{4/n}
\end{equation}
which is the desired bound for the Nash inequality with constant $\nconst=(3C)^{1+2/n}$.

Next, consider the case $\ell_0 < 1$, or equivalently $\upsilon\|f\|_1^2 < \dir_s(f)$ by \eqref{eq:virtual_length}.  A property of configuration space $\Ln$ is that for any $f\in L^1(\Ln)$
\begin{equation}
\|f-\kproj f\|_2^2 \leq \|f\|_2^2 \leq \|f\|_1^2
\end{equation}
since $\pi(\xx) \geq 1$ for all $\xx\in\Ln$.  Therefore,
\begin{equation}\label{eq:virtual_low}
\|f-\kproj f\|_2^{2+4/n} \leq \|f\|_1^{2+4/n} \leq \frac{\dir_s(f)}{\upsilon}\|f\|_1^{4/n}.
\end{equation}
so the Nash inequality holds with constant $\nconst=1$.

Lastly, consider the case where $\ell_0 > N$, or equivalently $\dir_s(f) \leq \upsilon \|f\|_1^2 N^{-1-n/2}$.  This case will require seemingly more subtle analysis due to the geometric complexity of configuration space but essentially boils down to a volume bound and another application of the Poincar\'e inequality.  For the remainder, set $\ell=N$.  Note that for this choice of length scale, $|\net|=1$ and for $\yy\in\net$, all pairs of sites are local $i \eql j$ for all $i,j\in[N]$.  In this trivial global setting, denote $\proj=\pyl$.  Then
\begin{equation}
\|f-\kproj f\|_2^2 \leq 2 \|(1-\kproj)f - \proj(1-\kproj)f\|_2^2 + 2\|\proj(1-\kproj)f\|_2^2
\end{equation}
by the Schwarz inequality.  The latter term is bounded by the $L^2$-norm although the techniques used in the argument are disjoint from the rest of the Nash inequality, so the proof will be postponed to the following result.  By Lemma \ref{lem:pkp},
\begin{equation}
\|\proj(1-\kproj)f\|_2^2 = \|\proj(1-\kproj)^2f\|_2^2 \leq \frac{C}{N} \|f-\kproj f\|_2^2
\end{equation}
where the first equality is from $1-\kproj$ being an orthogonal projection and the constant $C>0$ comes from Lemma \ref{lem:pkp} and depends only on $n$.
The previous two equations imply
\begin{equation}\label{eq:virtual_hi_ppp}
\|f-\kproj f\|_2^2 \leq 3\|(1-\kproj)f- \proj(1-\kproj)f\|_2^2
\end{equation}
Now, the Poincar\'e inequality applied to the test function $(1-\kproj)f$ tells us that
\begin{equation}\label{eq:virtual_hi_dir}
\|(1-\kproj)f- \proj(1-\kproj)f\|_2^2 \leq \pconst\frac{N}{\upsilon}\dyl(s;(1-\kproj)f) = \pconst\frac{N}{\upsilon}\dir_s(f)
\end{equation}
by Proposition \ref{prop:poincare} where the last equality comes from the fact that $\dyl(s;\cdot) = \dir_s(\cdot)$ for our global choice of $\yy$ and $\ell$ and that $\dir_s(f-\kproj f) = \dir_s(f)$ since $\kproj f \in \ker(\gen_s)$.  In particular, 
\begin{equation}\label{eq:virtual_hi}
\|f-\kproj f\|_2^{2+4/n} 
\leq \left(3\pconst \frac{N}{\upsilon}\dir_s(f)\right)^{1+2/n} 
\leq \left(3\pconst \frac{N}{\upsilon}\right)^{1+2/n} \dir_s(f) (\upsilon\|f\|_1^2N^{-1-n/2})^{2/n} = (3\pconst)^{1+2/n}\frac{\dir_s(f)}{\upsilon}\|f\|_2^{4/n}
\end{equation}
where the first inequality combines \eqref{eq:virtual_hi_ppp} and \eqref{eq:virtual_hi_dir} and the second inequality uses and \eqref{eq:virtual_length}.  This is the Nash inequality with constant $\nconst=(3\pconst)^{1+2/n}$.

Equations \eqref{eq:virtual_good}, \eqref{eq:virtual_hi}, and \eqref{eq:virtual_low} conclude the proof of the Nash inequality with constant $\nconst=\max\{(3C)^{1+2/n}, 1, (3\pconst)^{1+2/n}\}$.
\end{proof}

\begin{remark}
Much of the intuition behind the Nash inequality can be extracted from \eqref{eq:nash_length} by observing which $\ell$ optimizes the expression on the right hand side.  The heuristic hinted at is that a large Dirichlet form to $L^1$-norm ratio suggests obstructions to local equilibrium persist at small length scales.  On the other hand, a small Dirichlet form to $L^1$-norm ratio suggests local obstructions to equilibrium are insignificantly compared to the global obstruction to equilibrium.  One way to think about the optimization problem is on which length scale is cost-benefit ratio optimized for making corrections towards equilibrium.

A representative 
example for the small length scale regime is when the test function is $f=\delta_\xx$ where $x_a=N/2$ for all $a\in[n]$.  In this regime, $\dir_s(f) \gg \|f\|_1^2$.  A representative example for the large length scale regime is $f(\xx) = \sum_{a\in[n]} x_a/N$.  In this regime, $\dir_s(f) \ll \|f\|_1^2$.
\end{remark}

\begin{lemma}\label{lem:pkp}
The $L^2 \rightarrow L^2$ operator norm of $\proj(1-\kproj)$ satisfies
\begin{equation}
\|\proj(1-\kproj)\|_{2,2} \leq \frac{C}{\sqrt{N}}
\end{equation}
for some constant $C > 0$ depending only on $n$.  Recall that $\proj = \pyl$ with $\ell=N$ and $\yy\in\Ln$ arbitrary as defined in Definition \ref{defn:fakeproj} and $\kproj$ is the orthogonal projection onto the global kernel defined in Lemma \ref{lem:kernel}.
\end{lemma}

\begin{proof}
For each partition $\pP$, let $\psi_\pP, \zeta_\pP : \Ln \rightarrow \RR$ be the indicator functions that a position partition either coincides with $\pP$ or is a refinement of $\pP$, respectively.  That is,
\begin{equation}
\psi_\pP(\xx) = \one{\px=\pP} \quad \mbox{and} \quad \zeta_\pP = \sum_{\pQ\leq\pP} \psi_{\pQ}
\end{equation}
for all partitions $\pP$ of $[n]$.  Then we may conveniently write the makeshift projection as
\begin{equation}\label{eq:proj_repn}
\proj g = \sum_{\pP} \frac{\psi_\pP \ipn{\zeta_\pP, g}}{\|\zeta_\pP\|_2^2}
\end{equation}
for every $g\in L^2(\Ln)$ where the summation is taken over all partitions $\pP$ of $[n]$.  The $\zeta_\pP$ can be written in terms of a different basis where all $\psi_\pQ$ with $\pQ$ minimal in the lattice of even partitions of $[n]$ are replaced by $\chi_\sigma$, $\sigma \in M_n$.  In this representation
\begin{equation}
\zeta_\pP = \sum_{\sigma\in\symP\cap M_n} \chi_\sigma + \sum_{\pQ\not\in\min} a_{\pP\pQ}\psi_\pQ
\end{equation}
for some constants $-1 < a_{\pP\pQ} \leq 1$.  The second summation is over all partitions $\pQ$ which are not minimal.  In other words, $\pQ$ must have at most $n/2-1$ parts.  This representation is convenient because 
\begin{equation}\label{eq:kproj_orthog}
\ipn{\chi_\sigma, (1-\kproj)f} = \ipn{(1-\kproj)\chi_\sigma, f} = \ipn{0,f}=0
\end{equation}
by Lemma \ref{lem:globalker} and the orthogonality of $\kproj$.  Therefore, by \eqref{eq:proj_repn} and \eqref{eq:kproj_orthog}
\begin{equation}
\proj(1-\kproj)f = \sum_{\pP} \|\zeta_\pP\|_2^{-2} \psi_\pP \sum_{\pQ\not\in\min}a_{\pP\pQ}\ipn{\psi_\pQ, (1-\kproj)f}
\end{equation}
Now take $L^2$ norms and apply Cauchy-Schwarz
\begin{equation}
\|\proj(1-\kproj)f\|_2^2 \leq \sum_{\pP}\sum_{\pQ\not\in\min} \frac{\|\psi_\pP\|_2^2\|\psi_\pQ\|_2^2}{\|\zeta_\pP\|_2^4}\|(1-\kproj)f\|_2^2
\end{equation}
using the orthogonality of $\{\psi_\pP\}_\pP$.
By \eqref{eq:strata_count}, $\|\zeta_\pP\|_2^2 \geq c_2^{-1} N^{-n/2}$ and $\psi_\pP \leq c_2 N^{-n/2}$.  Counting configurations in the lower dimensional intersections of strata, there exists a constant $c_3>0$ depending only on $n$ such that
\begin{equation}
\|\psi_\pQ\|_2^2 \leq c_3 N^{n/2-1}
\end{equation}
for all non-minimal even partitions $\pQ$.  Using the naive $n^n$ bound on the number of partitions of $[n]$ gives us the final inequality
\begin{equation}
\|\proj(1-\kproj)f\|_2^2 
\leq n^{2n} c_2^3 c_3 N^{-1} \|(1-\kproj)f\|_2^2
\end{equation}
with constant $C = n^{2n}c_2^3 c_3$.
\end{proof}

\begin{remark}
We conclude this section with a complete classification of the global kernel.  This result is not needed anywhere but is a simple consequence of the results we have amassed thus far so it is included for completeness.
\end{remark}

\begin{corollary}\label{cor:globalker_complete}
The global kernel is exactly given by the stratum indicators
\begin{equation}
\cap_{i<j\in[N]} \ker(\gen_{ij}) = \vspan(\{\chi_\sigma | \sigma\in M_n\})
\end{equation}
\end{corollary}

\begin{proof}
By Lemma \ref{lem:globalker} and the Poincar\'e inequality, Proposition \ref{prop:poincare}, the following containments hold
\begin{equation}
\{\chi_\sigma | \sigma \in M_n\} \subset \ker(1-\kproj) \subset \ker(1-\proj).
\end{equation}
Using the notation from the proof of Lemma \ref{lem:pkp}, if $f\in\ker(1-\proj)$, then
\begin{equation}
f=\sum_{\pP} \frac{\psi_\pP\ipn{\zeta_\pP,f}}{\|\zeta_\pP\|_2^2}
\end{equation}
and in particular, there exist scalars $\{a_\pP\in\RR|\pP \mbox{ an even partition}\}$ such that $f=\sum_{\pP}a_\pP\psi_\pP$.  Matching $\psi_\pP$ coefficients in the previous equation implies
\begin{equation}
a_\pP = \sum_{\pQ} \frac{\ipn{\zeta_\pP,\psi_\pQ}}{\|\zeta_\pP\|_2^2} a_\pQ
\end{equation}
for every even partition $\pP$ where the sum is taken over every even partition $\pQ$.  In particular, for $\pP\in\min$ minimal, $a_\pP$ may be chose freely independent of each other.  Then all $\pQ \not\in\pP$ are fixed and can be computed recursively starting from the bottom of the partition lattice.  The end result is seen to be $f = \sum_{\pP\in\min} a_\pP \chi_\sigma$ where $\sigma$ is the unique element of $M_n \cap \symP$.  This shows that $\ker(1-\proj) \subset \{\chi_\sigma | \sigma\in M_n\}$ concluding the proof.
\end{proof}

\subsection{Ultracontractivity}
Most of the heavy lifting has been done through the combinatorics leading to the local Poincar\'e inequality and the careful configuration space dissection leading to the Nash inequality.  The Nash inequality left us with a lower bound on the Dirichlet form by its time integral, the $L^2$ norm.  For this last step in the energy method, we integrate the Nash inequality over time to obtain restrictive relations between $L^1$ and $L^2$ norms and time.  Then use formal duality to bounce between function spaces $L^1$, $L^2$, and $L^\infty$.

\begin{proposition}\label{prop:contractive}
There exists a constant $\uconst > 0$ depending only on $n$ such that the following implication holds for all times $0 \leq s_1 < s_2$.  If the coefficient array $\{c_{ij}(s)|i,j\in[N]\}$ satisfies Assumption \ref{ass:heavytail} with rate $\upsilon>0$ for all $s_1 \leq s \leq s_2$, then the corresponding transition semigroup satisfies the ultracontractive bound
\begin{equation}
\|(1-\kproj)\semi(s_1, s_2)\|_{2,\infty} \leq \frac{\uconst}{\upsilon^{n/4} |s_2-s_1|^{n/4}}
\end{equation}
where $\semi$ is the operator defined in \eqref{eq:semigroup}.
\end{proposition}

\begin{proof}
The first step is to bound $\|(1-\kproj)\semi(s_1,s_2)\|_{1,2}$.  To this end, suppose $f_{s_1} \in L^1(\Ln)$ with $\|f_{s_1}\|_1 = 1$.  Let $f_s = \semi(s_1, s) f_{s_1'}$ for all $s \geq s_1'$.  Then the time derivative of $L^2$ deviation from equilibrium is given by
\begin{equation}\label{eq:contr_derivative}
\partial_s \|f_s-\kproj f_s\|_2^{-4/n} = -\frac{2}{n} \|f_s-\kproj f_s\|_2^{-2-4/n} \partial_s\|f_s-\kproj f_s\|_2^2 = \frac{4}{n} \|f_s-\kproj f_s\|_2^{-2-4/n} \dir_s(f_s) 
\end{equation}
by Equation \eqref{eq:dir_derivative}.  Applying the Nash inequality and $L^1 \rightarrow L^1$ bound on the transition semigroup's operator norm to the expression in \eqref{eq:contr_derivative},
\begin{equation}\label{eq:contr_nash}
\partial_s \|f_s-\kproj f_s\|_2^{-4/n} \geq \frac{4\upsilon}{n\nconst\|f\|_1^{4/n}} \geq \frac{4\upsilon}{n\nconst(n!!)^{4/n}}
\end{equation}
where the first inequality is Proposition \ref{prop:nash} and the second inequality uses that $\|f_s\|_1 \leq n!! \|f_{s_1}\|_1 = n!!$ by Lemma \ref{lem:L1}.
Now integrate \eqref{eq:contr_nash} over the time interval $[s_1,s_2]$ to obtain
\begin{equation}
\|f_{s_2}-\kproj f_{s_2}\|_2^{-4/n} \geq \frac{4\upsilon}{n\nconst(n!!)^{4/n}}(s_2-s_1) + \|f_{s_1}-\kproj f_{s_1}\|_2^{-4/n} \geq \frac{4\upsilon}{n\nconst(n!!)^{4/n}}(s_2-s_1)
\end{equation}
Rearranging exponents gives
\begin{equation}
\|f_{s_2}-\kproj f_{s_2}\|_2 \leq n!!\left(\frac{n\nconst}{4}\right)^{n/4} \frac{1}{\upsilon^{n/4}(s_2-s_1)^{n/4}}
\end{equation}
from which we conclude the $L^1 \rightarrow L^2$ operator norm bound
\begin{equation}\label{eq:contr_12}
\|(1-\kproj) \semi(s_1,s_2)\|_{1,2} \leq \frac{\uconst}{\upsilon^{n/4}(s_2-s_1)^{n/4}}
\end{equation}
with constant $\uconst = n!!(n\nconst/4)^{n/4}$.

Note that the kernel projection commutes with the transition semigroup since the kernel projection shares an eigenbasis with the generator at every time.  That is, $[\kproj, \gen_s]=0$ for all $s\in[s_1',s_2']$ implies 
\begin{equation}\label{eq:comm_projsemi}
[\kproj, \semi(s_1',s_2')]=0.
\end{equation}
Moreover, the kernel projection is self-adjoint while the adjoint to the transition semigroup is the semigroup for the time reversed dynamics in the sense that $\semi(s_1', s_2')^* f = h_{s_2'}$ where $h$ is the unique solution to $\partial_s h_s = \gen_{s_1'+s_2'-s} h_s$ for all $s\in(s_1',s_2')$ with the initial condition $h_{s_1'}=f$.  This means that, since Assumption \ref{ass:heavytail} also holds for the time reversed coefficient array $\{c_{ij}(s_1+s_2-s) | i,j\in[N]\}$ for all times $s\in[s_1,s_2]$, the Nash inequality and therefore \ref{eq:contr_12} hold for $\semi(s_1', s_2')^*$.  
\begin{equation}\label{eq:contr_12dual}
\|(1-\kproj) \semi(s_1',s_2')^*\|_{1,2} \leq \frac{\uconst}{\upsilon^{n/4}(s_2-s_1)^{n/4}}
\end{equation}
In particular, by appealing to duality
\begin{equation}\label{eq:contr_2inf}
\|(1-\kproj)\semi(s_1',s_2')\|_{2,\infty} = \|\semi(s_1',s_2')(1-\kproj)\|_{2,\infty} = \|(1-\kproj)\semi(s_1',s_2')^*\|_{1,\infty} \leq \frac{\uconst}{\upsilon^{n/4}(s_2-s_1)^{n/4}}
\end{equation}
where the first equality is from identity \eqref{eq:comm_projsemi}, the second is duality, and the inequality is \eqref{eq:contr_12dual}.
\end{proof}

\section{Proof of main results}\label{sec:proof}

\subsection{Asymptotic normality}\label{sec:mainproofs}

\begin{proof}[Proof of Theorem \ref{thm:main}]
Note that the theorem holds trivially for $n=0$ with $\od(0) > 0$ any positive number, take $f=F=1$ in this case for the induction to go through.  Assume the result holds for $n-2$ with exponent $\od = \od(n-2) > 0$.  Without loss of generality, assume that $N^{\od} \ll Nt$, otherwise weaken the inductive hypothesis by scaling $\od$.  Now let
\begin{equation}\label{eq:params}
K = N^{1-\od}t, \quad T_2 = \frac{K}{N} \left(\frac{K}{N^{1+\od}t}\right)^{\frac{1}{n+2}}, \quad \ell_2 = \sqrt{KNT_2}, \quad \ell_1 = K^{3/4}, \quad T_1 = \frac{\sqrt{K}}{N}
\end{equation}
and fix the times $t_0 = t-T_2-T_1$ and $t_1 = t-T_2$.  For any length scale $1 \leq \ell \leq N$, recall the short range operator $\short(s;\ell)$ as defined in Definition \ref{defn:short_range} and is subject to the results of Section \ref{sec:l2relax}.  Consider another variant of the colored eigenvector moment flow operator, the \emph{lattice generator} $\latt(s;\ell)$ defined by
\begin{equation}
\latt(s;\ell) = \sum_{i<j\in[N]} \clatt_{ij}(s) \gen_{ij} 
\quad \mbox{where} \quad 
\clatt_{ij}(s) = \begin{cases} \creg_{ij}(s) &\mbox{ if } |i-j| \leq \ell_2 \mbox{ and } i,j\in \indint \\ \frac{N}{|i-j|^2} & \mbox{ otherwise} \end{cases}
\end{equation}
for every time $s \geq 0$ and sites $i \neq j \in [N]$.  Note that for $t_1 \leq s \leq t$, the coefficient array $\clatt_{ij}(s)$ satisfies Assumption \ref{ass:heavytail} with rate $\upsilon = N^{1-\varepsilon}$ for $\varepsilon>0$ arbitrarily small by Proposition \ref{prop:reg_fc} and \ref{prop:reg_eval}.  Hence, $\latt(s;\ell)$ is subject to the results from Section \ref{sec:energy}.

For each $\yy\in\Ln$ supported on $\indint^\kappa$, let $h_s(\xx) = h_s(\xx; \yy)$, $s \geq t_0$ be the family of $\Ln$-observable that is continuous in time and satisfies the partial differential equation with initial condition and flow given by
\begin{equation}
\begin{cases}
h_{t_0}(\xx; \yy) = \Av(K, \yy) f_{t_0}(\xx) & \\
\partial_s h_s(\xx) = \short(s;\ell_1) h_s(\xx) &\mbox{for all } t_0 < s < t_1 \\
\partial_s h_s(\xx) = \latt(s;\ell_2) h_s &\mbox{for all } s > t_1
\end{cases}
\end{equation}
for all $\xx\in\Ln$.
By Proposition \ref{prop:l2} and Proposition \ref{prop:contractive}, with overwhelming probability
\begin{equation}\label{eq:main_Linf}
\sup_{\yy \in \Lambda^n : y_a\in\indint^\kappa,a\in[n]} \|h_t(\cdot, \yy)\|_\infty^2 \leq N^{2\oc} \left(\frac{K}{NT_2} \right)^\frac{n}{2} \left( \frac{\ell_1}{K} + \frac{NT_1}{\ell_1} + \frac{1}{\sqrt{NT_1}} + \frac{K}{N t} + N^{-\od} \right)
\end{equation}
by making $\varepsilon>0$ and $\ob>0$ arbitrarily small.  The supremum is taken over all particle configurations $\yy\in\Ln$ supported on sites $\indint^\kappa$.  Here $\oc > 0$ is a control parameter which can also be taken arbitrarily small.  Moreover, by the comparison which is saved for the following section, Proposition \ref{prop:compare},
\begin{equation}\label{eq:main_comp}
\sup_{\xx\in\Ln : x_a\in\indint^\kappa,a\in[n]} |h_t(\xx;\xx) - f_t(\xx) + F_t(\xx;\xx)| \leq N^\oc \left( \frac{\ell_1}{K} + \frac{NT_1}{\ell_1} + \frac{\ell_2}{K} + \frac{NT_2}{\ell_2} \right)
\end{equation}
again the supremum is taken over configurations $\xx\in\Ln$ supported on sites $\indint^\kappa$ and the bound holds for $\oc>0$ arbitrarily small.  Combining \eqref{eq:main_Linf} and \eqref{eq:main_comp} gives
\begin{multline}
\sup_{\xx\in\Ln : x_a\in\indint^\kappa,a\in[n]} |f_t(x) - F_t(x;x)| \leq \sup_{\xx\in\Ln : x_a\in\indint^\kappa,a\in[n]} |h_t(\xx;\xx) - f_t(\xx) + F_t(\xx;\xx)| + |h_t(\xx; \xx)| \\ 
\leq N^\oc\left( \frac{\ell_1}{K} + \frac{NT_1}{\ell_1} + \frac{\ell_2}{K} + \frac{NT_2}{\ell_2} + \left(\frac{K}{NT_2} \right)^\frac{n}{4} \left( \frac{\ell_1}{K} + \frac{NT_1}{\ell_1} + \frac{1}{\sqrt{NT_1}} + \frac{K}{N t} + N^{-\od}\right)^{1/2} \right) \\ \leq N^{-\frac{\od}{n+2} + \oc} \leq N^{-\od(n)}
\end{multline}
for any new constant $0 < \od(n) < \od(n)/(n+2)$ with overwhelming probability.  The first inequality in the last line follows from our choice of parameters in \eqref{eq:params}.  The second inequality in the last line comes from taking $\oc>0$ sufficiently small after fixing the desired $\od(n)$.  This proves that the result holds with overwhelming probability.  This is sufficient because there is also a deterministic polynomial bound on the observable and ansatz
\begin{equation}
\|f_t\|_\infty + \|F_t\|_\infty \leq N^{n/2} + N^n
\end{equation}
since the eigenvectors are $L^2$ normalized $\|u_i\|_2=1$ and 
\begin{equation}
\ipr{\vv, \Im\grn_{\fc,t}(z) \vv} = \ipr{\vv, \Im\grn(z+tm_{\fc,t}(z)) \vv} \leq N\tr \Im\grn(z+tm_{\fc,t}(z)) = N\Im m_{\fc,t}(z)
\end{equation}
for all $z$ in the upper half plane and all $\vv \in S^{N-1}$ by the Schwarz inequality.
\end{proof}


\subsection{Comparison of infinite lattice perturbation}

We first collect a set of three useful tools in the following Lemma.  Then use these three to show the main comparison result Proposition \ref{prop:compare} which was used in the proof of the main theorem, Theorem \ref{thm:main}, in the previous subsection.  In what follows, refer to the local indicator function
\begin{equation}
\loc(\xx) = \begin{cases}
1 &\mbox{ if } \xx \mbox{ is supported on } \indint^{\kappa/10} \\
0 &\mbox{ otherwise }
\end{cases}
\end{equation}

\begin{proposition}\label{prop:compare_prelim}
Let $t_0$, $t$, and $K$ be as in the proof of Theorem \ref{thm:main}.  Suppose $t_0 \leq s_1 \leq s_2 \leq t$ are time scales and $N^\oc \leq \ell \leq N^{-\oc} K$ is a length scale satisfying $s_2 - s_1 \leq N^{-\varepsilon}\ell$.  Then there exists a constant $c>0$ such that for any configuration $\yy\in\Ln$ supported on $\indint^\kappa$, the following bounds hold with overwhelming probability
\begin{align}
\|\Av(K,\yy) (\semi_\short(s_1,s_2;\ell) - \semi(s_1,s_2)) f_{s_1}\|_\infty &\leq c\frac{NT}{\ell}\|f_{s_1}\loc)\|_\infty \label{eq:main_comp_SF} \\
\|[\Av(K,\yy), \semi_\short(s_1,s_2;\ell)] f\|_\infty &\leq c\frac{\ell}{K} \|f\loc\|_\infty \label{eq:main_comp_comm} \\
\|(\semi_\latt(s_1,s_2;\ell) - \semi_\short(s_1,s_2;\ell)) \Av(K,\yy) f\|_\infty &\leq c\frac{NT}{\ell}\|f\loc\|_\infty \label{eq:main_comp_LS}
\end{align}
where $f_{s_1}$ is the eigenvector moment observable and $f \in L^\infty(\Ln)$ is arbitrary.
\end{proposition}

\begin{proof}
Equation \eqref{eq:main_comp_SF} is a direct implication of Proposition \ref{prop:duhamel} and its proof, while equation \eqref{eq:main_comp_comm} is a restatement of \eqref{rmk:fspcomm} in the remark following Corollary \ref{cor:FSPcomm}.  It remains to provide an argument for \eqref{eq:main_comp_LS} which proceeds according to a Duhamel expansion similar to the proof of Proposition \ref{prop:duhamel}.  Taking $L^\infty$ norms on both sides of
\begin{equation}\label{eq:latt_duhamel}
(\semi_\latt(s_1,s_2;\ell) - \semi_\short(s_1,s_2;\ell)) \Av(K,\yy) f_{s_1}
= \int_{s_1}^{s_2} \semi_\latt(s,s_2;\ell) (\short(s;\ell) - \latt(s;\ell)) \semi_\short(s_1,s;\ell) \Av(K,\yy) f_{s_1} ds
\end{equation}
gives
\begin{equation}
\|(\semi_\latt(s_1,s_2;\ell) - \semi_\short(s_1,s_2;\ell) \Av(K,\yy) f\|_\infty
\leq n!! T \| (\short(s;\ell) - \latt(s;\ell)) \semi_\short(s_1,s;\ell) \Av(K,\yy) f \|_\infty
\end{equation}
as $\semi_\latt(s_1,s_2;\ell)$, while not an $L^\infty$ contraction, has bounded $L^\infty\rightarrow L^\infty$ operator norm by the dual of Lemma \ref{lem:L1}.  Consider the decomposition $\semi_\short(s_1,s_2;\ell) \Av(K,\yy) f = f^{(0)} + f^{(1)}$ where $f^{(0)}$ is supported on $\Ln(\indint^{\kappa/2})$ and $f^{(1)}$ is supported on $\Ln(\indint^{\kappa/10})\backslash\Ln(\indint^{\kappa/2})$.  This decomposition satisfies $\|f^{(0)}\|_\infty \leq n!! \|\Av(K,\yy)f\|_\infty$ by Lemma \ref{lem:L1} and $\|f^{(1)}\|_\infty \leq e^{-N^{\varepsilon/2}}\|\Av(K,\yy)f\|_1$ by the same logic leading to \eqref{eq:l2_expbound} using Proposition \ref{prop:FSP}.  Therefore,
\begin{equation}\label{eq:main_comp_decomp}
\| (\latt(s;\ell) - \short(s;\ell)) \semi_\short(s_1,s;\ell) \Av(K,\yy) f \|_\infty \leq \|(\latt(s;\ell) - \short(s;\ell))f^{(0)}\|_\infty + \|(\latt(s;\ell) - \short(s;\ell))f^{(1)}\|_\infty
\end{equation}
For the first term in \ref{eq:main_comp_decomp}, note that for every $\xx\in\Ln$ supported on $\indint^{\kappa/2}$, 
\begin{equation}
\|((\latt(s;\ell)-\short(s;\ell))\delta_\xx\|_1 \leq c \sum_{j=\ell}^N \frac{N}{j^2} \leq c\frac{N}{\ell} 
\end{equation}
for some constant $c>0$ depending only on $n$.  By Holder's inequality
\begin{equation}\label{eq:main_comp_0}
\|(\latt(s;\ell) - \short(s;\ell))f^{(0)}\|_\infty \leq \sup_\xx \|((\latt(s;\ell)-\short(s;\ell))\delta_\xx\|_1 \|f^{(0)}\|_\infty \leq c\frac{N}{\ell}\|\Av(K,\yy) f\|_\infty
\end{equation}
where the supremum is taken over all $\xx\in\Ln$ supported on $\indint^{\kappa/2}$.
\begin{equation}
\|(\latt(s;\ell) - \short(s;\ell))f^{(0)}\|_\infty \leq \|\sum_{i\in\indint^{\kappa/2}}\sum_{j\in[N]:|j-i|>\ell} \frac{N}{|i-j|^2} \gen_{ij} f^{(0)}\|_\infty \leq \frac{N}{\ell} \|f^{(0)}\|_\infty \leq \frac{N}{\ell} n!!\|f\|_\infty
\end{equation}
For general $\xx\in\Ln$, the weaker bound still holds
\begin{equation}
\|(\latt(s;\ell) - \short(s;\ell)\delta_\xx\|_1	\leq c \sum_{j=1}^N	\frac{N}{j^2} \leq \frac{c\pi^2}{6} N
\end{equation}
by the definition of the lattice coefficients $\clatt_{ij}(s)$.  By Holder's inequality
\begin{equation}\label{eq:main_comp_1}
\|(\latt(s;\ell) - \short(s;\ell))f^{(1)}\|_\infty \leq \sup_\xx \|((\latt(s;\ell)-\short(s;\ell))\delta_\xx\|_1\|f^{(1)}\|_\infty \leq Ne^{-N^{\varepsilon/2}}\|\Av(K,\yy)f\|_1
\end{equation}
The desired bound in \eqref{eq:main_comp_LS} is now a consequence of \eqref{eq:main_comp_decomp}, \eqref{eq:main_comp_0}, and \eqref{eq:main_comp_1} after taking into account that $|\Av(\xx;K,\yy)| \leq |\loc(\xx)|$ uniformly in $\xx\in\Ln$ and that
\begin{equation}
\|\Av(K,\yy)f\|_1 \leq \|\Av(K,\yy)f\|_\infty \pi(\{\xx\in\Ln | \dist(\xx,\yy) < 3K\})
\end{equation}
and that the measure appearing on the right hand side of this expression grows at most polynomially fast in $N$.
\end{proof}

\begin{proposition}\label{prop:compare}
For any small constant $\oc > 0$, for all $\xx\in\Ln$ supported on $\indint^{\kappa}$, the following bound holds with overwhelming probability
\begin{equation}\label{eq:main_comp_alt}
|h_t(\xx;\xx) - f_t(\xx) + F_t(\xx;\xx)| \leq N^\oc \left( \frac{\ell_1}{K} + \frac{N T_1}{\ell_1} + \frac{\ell_2}{K} + \frac{NT_2}{\ell_2} \right)
\end{equation}
where $\ell_1,\ell_2,T_1,T_2,$ and $h$ are as defined in the proof of Theorem \ref{thm:main} in the previous subsection.
\end{proposition}

\begin{proof}
Fix $\xx\in\Ln$ which is supported in $\indint^\kappa$ and slightly abuse notation to write $F_t(\zz) = F_t(\zz;\xx)$ for all $\zz\in\Ln$.  Since $\Av(K,\xx)$ is a diagonal operator with $\Av(\xx;K,\xx) = 1$ and $F_t = F_t(\cdot; \xx) \in \ker(\gen_{ij})$ for all $i\neq j\in[N]$, we can rewrite the expression appearing on the left hand side by
\begin{equation}
h_t(\xx;\xx) - f_t(\xx) + F_t(\xx) = \left(\semi_\latt(t_1,t);\ell_2) \semi_\short(t_0,t_1;\ell_1) \Av(K,\yy) - \Av(K,\yy) \semi(t_1,t)\semi(t_0,t_1)\right) (f_{t_0} - F_t) (\xx)
\end{equation}
This comparison is done through five interpolating steps following the tools prepared in Proposition \ref{prop:compare_prelim}.
\begin{align}
h_t(\xx;\xx) - f_t(\xx) + F_t(\xx) = 
&& &\semi_\latt(t_1,t;\ell_2) [\semi_\short(t_0,t_1;\ell_1), \Av(K,\xx)] (f_{t_0} - F_t) (\xx) \label{eq:comp1}\\
&&+&\semi_\latt(t_1,t;\ell_2) \Av(K,\xx) \left(\semi_\short(t_0,t_1;\ell_1) - \semi(t_0,t_1)\right) (f_{t_0} - F_t) (\xx) \label{eq:comp2}\\
&&+&\left(\semi_\latt(t_1,t;\ell_2) - \semi_\short(t_1,t;\ell_2)\right)\Av(K,\xx) (f_{t_1} - F_t) (\xx) \label{eq:comp3}\\
&&+&[\semi_\short(t_1,t;\ell_2), \Av(K,\xx)] (f_{t_1} - F_t) (\xx) \label{eq:comp4}\\
&&+&\Av(K,\xx) \left(\semi_\short(t_1,t;\ell_2) - \semi(t_1,t)\right) (f_{t_1} - F_t) (\xx)\label{eq:comp5}
\end{align}
We bound these terms sequentially.  For \eqref{eq:comp1}, use \eqref{eq:main_comp_comm} to see that
\begin{equation}\label{eq:comp_1bd}
\|\semi_\latt(t_1,t;\ell_2) [\semi_\short(t_0,t_1;\ell_1), \Av(K,\xx)] (f_{t_0} - F_t)\|_\infty \leq cn!!\frac{\ell_1}{K}\|(f_{t_0}-F_t)\loc\|_\infty
\end{equation}
where we first used that $\short_\latt$ is bounded in $L^\infty$ by the dual of Lemma \ref{lem:L1}.  For \eqref{eq:comp2}, use \eqref{eq:main_comp_SF}
\begin{equation}
\|\semi_\latt(t_1,t;\ell_2) \Av(K,\xx) \left(\semi_\short(t_0,t_1;\ell_1) - \semi(t_0,t_1)\right) (f_{t_0} - F_t)\|_\infty 
\leq cn!!\frac{NT_1}{\ell_1}\|f_{t_0}\loc\|_\infty
\end{equation}
where again we needed the $\semi_\latt$ boundedness in $L^\infty$.  In addition, this step used Corollary \ref{cor:ansatzker} to say that $\semi_\short(t_0,t_1;\ell_1)F_t = \semi(t_0,t_1)=F_t$. For \eqref{eq:comp3}, 
\begin{equation}
\|\left(\semi_\latt(t_1,t;\ell_2) - \semi_\short(t_1,t;\ell_2)\right)\Av(K,\xx) (f_{t_1} - F_t)\|_\infty \leq c \frac{NT_2}{\ell_2} \|(f_{t_1} - F_t)\loc\|_\infty
\end{equation}
is a direct application of \eqref{eq:main_comp_LS} to the test function $f_{t_1} - F_t$.  For \eqref{eq:comp4}, use \eqref{eq:main_comp_comm}
\begin{equation}
\|[\semi_\short(t_1,t;\ell_2), \Av(K,\xx)] (f_{t_1} - F_t)\|_\infty \leq c\frac{\ell_2}{K} \|(f_{t_1}-F_t)\loc\|_\infty
\end{equation}
This is also a direct application.  For \eqref{eq:comp5}, use \eqref{eq:main_comp_SF}
\begin{equation}
\|\Av(K,\xx) \left(\semi_\short(t_1,t;\ell_2) - \semi(t_1,t)\right) (f_{t_1} - F_t)\|_\infty \leq \frac{NT_2}{\ell_2} \|f_{t_1}\loc\|_\infty
\end{equation}
In this step, we again used Corollary \ref{cor:ansatzker} to say that the ansatz obersable is invariant under all dynamics.
Finally, by delocalization from Corollary \ref{cor:deloc}, $\|f_s\loc\|\leq N^{\frac{n}{2}\ob}$ and Assumption \ref{ass:evec} gives $\|F_t\loc\|_\infty \leq N^{\frac{n}{2}\ob}$.  Taking $\ob > 0$ sufficiently small proves \eqref{eq:main_comp_alt}.
\end{proof}

\subsection{Comparison for matrix models}

In this subsection, we use the main theorem along with semicircle laws from prior works and typical comparison arguments to prove Theorems \ref{thm:genwig} and \ref{thm:p-reg}.  

We start with the generalized Wigner comparison following the methods of \cite{QUE}.  The proof can follow either the moment matching argument from Section 5 in \cite{QUE} or the dynamical argument from Appendix A in \cite{QUE}.  The only main difference is replacing instances of $| u_i \cdot q|^2$ with $(u_i \cdot v) (u_i \cdot w)$.  To minimize the list of reference needed in backtracking a complete proof, the dynamical approach is briefly provided here.

The generalized Wigner case will rely on the isotropic local law (Theorem 2.2 in \cite{isotropic}).  Let 
\begin{equation}
\rho(dx) = \frac{1}{2\pi} \sqrt{(4-x^2)_+} dx \quad \mbox{and} \quad m(z) = \int \frac{\rho(dx)}{x-z} = \frac{-z + \sqrt{z^2 - 4}}{2}
\end{equation}
denote the semicircle law and its Stieltjes transform respectively.  The branch of the square root is chosen so that the Stieltjes transform satisfies $m(z) \rightarrow 0$ as $z \rightarrow \infty$.
\begin{theorem}[Isotropic local semicircle law, \cite{isotropic} and \cite{rigidity}]\label{thm:iso-genwig}
Suppose $H$ is a generalized Wigner ensemble and let $\varepsilon, \xi, \kappa > 0$ be small.  Then with overwhelming probability
\begin{equation}
|\langle v, G(z) w \rangle - \langle v, w \rangle m(z)| \leq \frac{N^\varepsilon}{\sqrt{N\eta}} \quad \mbox{and} \quad |m_N(z) - m(z)| \leq \frac{N^\varepsilon}{N\eta}
\end{equation}
uniformly for $z = E+i\eta \in (-2+\kappa, 2-\kappa) + i (N^{-1+\xi}, 1)$ and all unit vectors $v, w \in \RR^N$.
\end{theorem}

Consider the matrix valued Ornstein-Uhlenbeck process generated by
\begin{equation}\label{eq:gendbm}
dh_{ij}(t) = \frac{dB_{ij}(t)}{\sqrt{N}} - \frac{1}{2Ns_{ij}} h_{ij}(t) dt
\end{equation}
initialized at a generalized Wigner matrix $H_0$.  Here $B$ is symmetric with $(B_{ij})_{i \leq j}$ iid Brownian motions, and $s_{ij} = \EE{h_{ij}(0)^2}$.  It is clear that for any $t \geq 0$, $\EE{h_{ij}(t)} = 0$ and $\EE{h_{ij}(t)^2} = s_{ij}$.  This \emph{generalized Dyson Brownian motion} flow admits a continuity estimate up to time $t = N^{-1/2}$ for smooth functions of the symmetric matrix with bounded third order derivatives (see Lemma A.1 in \cite{QUE}).  This result is used in the following proposition.

\begin{proposition}\label{prop:compare-genwig}
Let $\alpha > 0$, $0 < \delta < 1/2$, and $t = N^{-1+\delta}$.  Denote by $H_t$ be the solution to \eqref{eq:gendbm} with a generalized Wigner matrix $H_0$ as the initial data.  Let $m$ be a positive integer and $\Theta : \RR^{2m} \rightarrow \RR$ a smooth function satisfying
\begin{equation}
\sup_{k \leq 5, x \in \RR} \frac{|\Theta^{(k)}(x)|}{(1+|x|)^C} < \infty
\end{equation}
for some $C>0$.  Let $u_1(t), \ldots u_N(t)$ denote the eigenvectors of $H_t$ corresponding to eigenvalues $\lambda_1(t) \leq \ldots \leq \lambda_N(t)$ respectively.   Then there exists $\varepsilon > 0$ depending on $\Theta, \delta, \alpha$ such that
\begin{equation}
\sup_{\substack{i_1, \ldots i_m \in [\alpha N, (1-\alpha)N] \\ v_1, \ldots v_m \in S^{N-1} \\ w_1, \ldots w_m \in S^{N-1}}} |(\ee^{H_t} - \ee^{H_0}) \Theta \left((N(\lambda_{i_a} - \gamma_{i_a}), N(u_{i_a} \cdot v_a)(u_{i_a} \cdot w_a))_{a \in [n]}\right) | \leq N^{-\varepsilon}
\end{equation}
\end{proposition}

\begin{proof}
The proof is identical to Corollary A.2 in \cite{QUE} except for the minor generalization to statement (i).  In particular, let $H$ and $\tilde H$ refer to two generalized Wigner ensembles. For any $\kappa > 0$, there exists $\xi, \varepsilon > 0$ such that for any $N^{-1-\xi} < \eta < 1$ and any smooth function $F$ with polynomial growth,
\begin{equation}\label{eq:subcritical}
\sup_{\substack{v,w \in S^{N-1} \\ E_1, \ldots, E_m \in (-2+\kappa, 2-\kappa)}} |(\ee^H - \ee^{\tilde H}) F((\langle v_k, G(z_k) w_k \rangle )_{k=1}^m)| \leq C N^{-\varepsilon}\left( \frac{1}{N\eta} + \frac{1}{\sqrt{N\eta}} \right)
\end{equation}
where $C > 0$ depends on $\kappa$ and $F$ and $z_k = E_K + i \eta$.  This together with the level repulsion estimate which is classical for generalized Wigner (see \cite{EY15}, eqaution (5.32)]) make up the necessary inputs for the argument from Section 5 in \cite{knowles2013eigenvector} to conclude the proof.  

To prove \eqref{eq:subcritical}, we appeal to the isotropic local law (Theorem \ref{thm:iso-genwig}) and the following replacement argument introduced in Lemma A.1 from \cite{QUE}.
\begin{lemma}[$t < N^{-1/2}$ Continuity, \cite{QUE}]
Denote $\partial_{ij} = \frac{\partial}{\partial h_{ij}}$.  Suppose $F$ is a smooth function of matrix elements $h_{ij}$ and satisfies
\begin{equation}\label{eq:continuity}
\sup_{\substack{0 \leq s \leq t \\ 1 \leq i \leq j \leq N \\ 0 \leq \theta \leq 1}} \EE{(N^{3/2} |h_{ij}(s)|^3 + \sqrt{N} |h_{ij}(s)|) \left| \partial_{ij}^3 F(H_s - \theta h_{ij}(s)(e_ie_j^\top + e_je_i^\top))\right|} \leq M
\end{equation}
where the input to $F$ interpolates between $H_s$ and $H_s$ with the $(i,j)$ and $(j,i)$ entries zeroed out.  Then
\begin{equation}
\EE{F(H_t)} = \EE{F(H_0)} + O(tN^{1/2}) M.
\end{equation}
\end{lemma}
Returning to the proof of \eqref{eq:subcritical}, for simplicity only consider the case $m=1$, $z_1=z = E+i\eta$, $-2+\kappa < E< 2-\kappa$, and $N^{-1-\xi} < \eta < 1$.  In bounding the third derivative, we get
\begin{equation}
\partial_{ij}^3 \langle v, G(z) w\rangle = -\sum_{a,b} \sum_{\alpha,\beta} v_a G(z)_{a,\alpha_1} G(z)_{\beta_1 \alpha_2} G(z)_{\beta_2\alpha_3} G(z)_{\beta_3 b} w_b
\end{equation}
where $\{\alpha_k, \beta_k\} = \{i,j\}$.  From the isoperimetric local semicircle law \ref{thm:iso-genwig}, the following four expressions
\begin{equation}
\sum_{a} v_a G(z)_{a,\alpha_1}, \quad G(z)_{\beta_1 \alpha_2}, \quad G(z)_{\beta_2\alpha_3}, \quad \sum_b G(z)_{\beta_3 b} w_b
\end{equation}
are bounded by $N^{2\xi}((N\eta)^{-1} + (N\eta)^{-1/2})$ with overwhelming probability for all $\alpha,\beta$ when $\eta > N^{-1+\xi}$.  By the dyadic decomposition from Section 8 in \cite{EYY12}, for any $0 < y \leq \eta$,
\begin{multline}
|\langle v, G(E+iy) w \rangle|^2 \leq \frac{1}{2}|\langle v, G(E+iy) v\rangle| + \frac{1}{2}|\langle w, G(E+iy)w \rangle|^2 \\
\leq C \left(\log N \frac{\eta}{y}\right)^2 (|\langle v, \Im G(E+i\eta) v\rangle| + |\langle w, \Im G(E+i\eta) w\rangle|)
\end{multline}
where the first inequality holds because $G$ is symmetric, albeit complex.  Therefore, \eqref{eq:continuity} holds for $M = C N^{5\xi}((N\eta)^{-1} + (N\eta)^{-1/2})$.  The above continuity lemma finishes the proof of \eqref{eq:subcritical}.
\end{proof}

\begin{proof}[Proof of Theorem \ref{thm:genwig}]
This argument again follows the final proof in Appendix B of \cite{QUE}.  Sticking with the above setup, let $H_0$ be a generalized Wigner matrix driven forward through $H_t, t \geq 0$ by \eqref{eq:gendbm}.  Now fix $\delta \in (0,1/2)$, time $t = N^{-1+\delta}$, and scale $\nu = \frac{1}{2} \inf_{ij} s_{ij}(1-e^{-\frac{t}{Ns_{ij}}})$.  Define a new matrix ensemble $\tilde H$ by
\begin{equation}
\tilde h_{ij} = \frac{1}{\sqrt{1-(N+1)\nu}} \left( e^{-\frac{t}{2Ns_{ij}}} h_{ij}(0) + W_{ij} \sqrt{s_{ij} \left(1-e^{-\frac{t}{Ns_{ij}}}\right) - \nu} \right)
\end{equation}
where $(W_{ij})_{i \leq j}$ are iid standard Gaussians and $W_{ij} = W_{ji}$.  Then $\tilde H$ is generalized Wigner and $H_t$ admits the same distribution as $\tilde H + \sqrt{N \tilde \nu} Z$ where $Z$ is a GOE and $\tilde \nu = \nu / (1-(N+1)\nu)$.  Therefore,
\begin{align}
& \sup_{\substack{i_a \\ v_a}} |\EE{P(\sqrt N u_{i_a} \cdot v_a)} - \EE{P(N_{i_a} \cdot v_a)}| \\
\leq & \sup_{\substack{i_a \\ v_a}} |\EE{P(\sqrt N u_{i_a} \cdot v_a)} - \EE{P(\sqrt N u_{i_a}(t) \cdot v_a)}| 
+ |\EE{P(\sqrt N u_{i_a}(t) \cdot v_a)} - \EE{P(N_{i_a} \cdot v_a)}| \\ 
\leq & N^{-\varepsilon} + N^{-\od}
\end{align}
The first term in the last line uses Proposition \ref{prop:compare-genwig} after noting that we may assume $|\{a \in [m] | i_a = j\}|$ is even for every $j \in [N]$.  The second term in the last line is the consequence of Theorem \ref{thm:main} applied to $\tilde H$ at time $\tilde t = N \tilde \nu$ after noting that we may take $\nu = N^{-2 + \xi}$ for some $\xi \in (0,1)$ and that generalized Wigner ensembles satisfy Assumptions \ref{ass:eval} and \ref{ass:evec} by Theorem \ref{thm:iso-genwig}.
\end{proof}

The comparison for the graph models follows the argument in Section 4 of \cite{bourgade2017eigenvector}.  Again, all estimates on $|u_i \cdot q|^2$ terms must be replaced by analogous estimates on $(u_i \cdot v)(u_i \cdot w)$ terms.  Begin by importing the contents of Propositions 4.2 and 4.3 from \cite{bourgade2017eigenvector}.

Let $H = A/\sqrt{p(1-p/N)}$ be the normalized adjacency matrix of an Erd\H{o}s--R\'enyi graph as given in Definition \ref{defn:graphs}.  Drive $H_t$, $t \geq 0$, forward by the stochastic differential equation corresponding to the matrix valued Ornstein-Uhlenbeck process centered at $f = p/N (p(1-p/N))^{-1/2}$
\begin{equation}
d h_{ij}(t) = \frac{dB_{ij}(t)}{\sqrt{N}} - \frac{1}{2}(h_{ij}(t) - f)dt
\end{equation} 
where $(B_{ij}(t))_{i \leq j}$ are independent Brownian motions with variance $(1+\delta_{ij})t$ and $B_{ij} = B_{ji}$.  In particular, $H_t$ shares the same distribution with $f + e^{-t/2}(H-f) + \sqrt{1-e^{-t}} Z$ where $Z$ is an independent GOE.

\begin{proposition}\label{prop:perturb-ER}
Suppose $N^\delta \leq p \leq N/2$, $0 < \ob \leq \delta/3$, $ 0 \leq s \ll 1$, and $v,w \in \RR^N \cap e^\perp$ (where $e = (1,1,\ldots,1)^\top/\sqrt{N}$) are unit vectors. Then the following statements hold for any $\oc > 0$ with overwhelming probability:
\begin{enumerate}
\item all eigenvectors of $H_s$ are delocalized in base directions and in $v,w$: for all $i \in [N]$, \begin{equation} \max_{j \in [N]} |\langle e_j, u_i \rangle|^2 + |\langle v, u_i \rangle|^2 + |\langle w, u_i \rangle|^2 \leq C N^{-1+\oc}\end{equation}
\item eigenvalues do not accumulate: for any interval $I$ of length $|I| \geq N^{-1+\oc}$, we have \begin{equation}|\{i : \lambda_i \in I\}| \leq C|I|N. \end{equation}
\item Assumptions \ref{ass:eval} and \ref{ass:evec} hold for $\tilde H_s = f + e^{-t/2}(H-f)$.
\end{enumerate}
\end{proposition}

\begin{proof}
The proof is identical to Proposition 4.3 in \cite{bourgade2017eigenvector} except with the equation (4.11) replaced by 
\begin{equation}
|\langle v, G(z) w \rangle - m(z) \langle v, w \rangle| \leq (\log N)^{C \log \log N} \left( \frac{1}{p^{1/2}} + \frac{1}{(N\eta)^{1/2}} \right)
\end{equation}
uniformly for $z = E+i\eta$ with $|E| \leq 5$ and $0 < \eta \leq 1$.  This inequality is true for $G(z) = (H_s - z)^{-1}$ and $G(z) = (\tilde H_s - z)^{-1}$ with overwhelming probability by a combination of the exchangability between $H_s$ and $\tilde H_s$, Theorem 8.3 and Remark 8.6 in \cite{BKY}, and Theorem 2.9 in \cite{EKYY12}.
\end{proof}

Recall the quantity $Q_i$ defined by Tao and Vu in \cite{TV10} and \cite{TV11} on the space of symmetric $N \times N$ real matrices given by $Q_i(A) = N^{-2} \sum_{j \neq i} |\lambda_i(A) - \lambda_j(A)|^{-2}$ for the sake of capturing derivatives of eigenvalues.

\begin{proposition}\label{prop:Qi-deriv}
Suppose $A$ is a deterministic $N \times N$ real symmetric matrix satisfying items (i) and (ii) from Proposition \ref{prop:perturb-ER} and $Q_i \leq N^{2\tau}$.  Then 
\begin{align}
|\partial_{ab}^{(k)} \lambda_i(A)| &= O(N^{-1+(k-1)\tau + (2k-1)\oc}) \\
|\partial_{ab}^{(k)} Q_i(A)| &=O(N^{(k+2)\tau + (2k+2)\oc}) \\
|\partial_{ab}^{(k)} \langle v, u_i(A)\rangle \langle u_i(A), w \rangle|  &= O(N^{-1+k\tau+(2k+1)\oc})
\end{align}
for $k=1,2,3$.
\end{proposition}

\begin{proof}
The first two bounds are proved in Proposition 4.6 of \cite{huang2015bulk}. Provided that $v=w$, the proof of the third bound is contained in Proposition 4.5 of \cite{bourgade2017eigenvector}.  For $v \neq w$, the computation still follows through.  Indeed, the contour formula says that
\begin{equation}
\partial_{ab}^{(k)} \langle v, u_i \rangle \langle u_i, w \rangle = - \frac{1}{2\pi i}\partial_{ij}^{(k)} \oint \langle v, G(z) w \rangle dz = \frac{(-1)^{k+1} k!}{2\pi i} \oint \langle v, (G(z)\Theta)^kG(z)w\rangle dz
\end{equation}
where the integral is taken over a contour which contains $\lambda_i$ and no other eigenvalue.  Here, $\Theta = e_ae_b^\top + e_be_a^\top$ is the matrix with $\Theta_{cd} = 1$ if $\{c,d\}=\{a,b\}$ and $0$ otherwise.  Thanks to the delocalization estimate in Proposition \ref{prop:perturb-ER} item (i), the derivatives are sufficiently bounded.  For example
\begin{equation}
|\partial_{ab} \langle v, u_i \rangle \langle u_i, w\rangle| = \left| \sum_{j \neq i} \frac{\langle v, u_j \rangle \langle u_j, \Theta u_j \rangle \langle u_j, w \rangle}{\lambda_j - \lambda_i} \right| \leq \frac{2N^{2\oc}}{N^2} \sum_{j \neq i} \frac{1}{|\lambda_j - \lambda_i|} \leq N^{-1+\tau + 3\oc}
\end{equation}
by the dyadic argument from Lemma 4.5 in \cite{huang2015bulk} which is a consequence of $Q_i \leq N^{2\tau}$.
\end{proof}

\begin{proof}[Proof of Theorem \ref{thm:p-reg}]
The argument follows in exactly the same way as the proof of \cite[Theorem 1.1]{bourgade2017eigenvector} at the end of \cite[Section 4]{bourgade2017eigenvector} after replacing all reliances on \cite[Propositions 4.2 and 4.3]{bourgade2017eigenvector} with Propositions \ref{prop:Qi-deriv} and \ref{prop:perturb-ER}.
The extension to $p$-regular graphs also follows in the same way.
\end{proof}


\appendix

\section{Derivation of the colored dynamics}\label{app:derivedynamics}

This appendix is devoted to proving Theorem \ref{thm:levmf} on the colored eigenvector moment flow.  The derivation requires the SEE differentials which we state here for completeness.  See \cite[Appendix B]{QUE} for a proof.
\begin{theorem}[Stochastic Eigenstate Equation]\label{thm:evdbm}
Let $\EV \in O(N)$ an orthogonal matrix and $\bl$ be a diagonal matrix with increasing entries $\lambda_1 < \ldots < \lambda_n$.  Let $\dat = \EV\bl\EV^\top$ be a symmetric matrix.  Then there are strong solutions $\dat(s)$, $\bl(s)$, and $\EV(s)$ which are stochastic processes in the space of symmetric matrices, increasing-diagonal matrices, and orthogonal matrices, respectively, satisfying the following stochastic partial differential equations with initial data $\dat(0) = \dat$, $\bl(0)=\bl$, and $\EV(0) = \EV$:
\begin{align}
dh_{\alpha\beta}(s) &= \sqrt{\frac{1+\delta_{\alpha\beta}}{N}} W_{\alpha\beta}(s) \label{appeq:underlying} \\ 
d\lambda_i(s) &= \frac{dB_{ii}(s)}{\sqrt{N}} + \frac{1}{N} \sum_{j\neq i} \frac{1}{\lambda_j(s) - \lambda_i(s)} dt  \label{appeq:value} \\ 
d u_i^\alpha(s) &= \frac{1}{\sqrt{N}} \sum_{j = 1, j \neq i}^N \frac{u_j^\alpha(s) dB_{ij}(s)}{\lambda_j(s) - \lambda_i(s)} - \frac{1}{2N} \sum_{j = 1, j \neq i}^N \frac{u_i^\alpha(s) ds}{\left(\lambda_j(s) - \lambda_i(s)\right)^2} \label{appeq:vector}
\end{align}
where $W_{\alpha\beta}(s)=W_{\beta\alpha}(s)$ and $B_{ij}(s)=B_{ji}(s)$ are $N(N+1)$ mutually independent and identically distributed standard Brownian motions for all $1 \leq \alpha \leq \beta \leq N$ and $1 \leq i \leq j \leq N$.  Moreover, the random symmetric matrices $H(s)$ and $\EV(s)\bl(s)\EV(s)^\top$ are identically distributed. 
\end{theorem}

\subsection{Deriving the colored eigenvector moment flow}

\begin{proof}[Proof of Theorem \ref{thm:levmf}]
Applying It\^o's formula to \ref{appeq:vector}, we obtain
\begin{multline}
\partial_s \EE{ \prod_{a=1}^n \sqrt{N} \ipr{\vu_{x_a}(s), \vv_a} } = \sum_{b=1}^n \EE{ \left(\prod_{a \neq b} \sqrt N \ipr{\vu_{x_a}(s), \vv_a} \right) \left( - \frac{1}{N} \sum_j \frac{\sqrt{N} \ipr{\vu_{x_b}(s), \vv_b}}{|\lambda_i(s) - \lambda_j(s)|^2} \right)} \\
+ \sum_{b < c} \EE{
\begin{gathered}
\left( \prod_{a \neq b,c} \sqrt{N} \ipr{\vu_{x_a}(s), \vv_a} \right) \phantom{ \left(\prod_{a \neq b,c} \sqrt{N} \ipr{\vu_{x_a}(s), \vv_a} \right)} \\
\times \left( \one{x_b = x_c} \sum_{j} \frac{2}{N} \frac{N \ipr{\vu_j(s), \vv_b} \ipr{\vu_j(s), \vv_c}}{|\lambda_j(s) - \lambda_{x_b}(s)|^2} - \one{x_b \neq x_c} \frac{2}{N} \frac{N \ipr{\vu_{x_b}(s), \vv_c} \ipr{\vu_{x_c}(s), \vv_b}}{|\lambda_{x_b}(s) - \lambda_{x_c}(s)|^2}\right) 
\end{gathered}
}
\end{multline}
In terms of the prenormalized observable $\tilde f_s(\xx) = \EE{ \prod_{a = 1}^n \sqrt{N} \ipr{\vu_{x_a}(s), \vv_a} } = \sqrt{\pi(\xx)} f_s(\xx)$, this time derivative can be rewritten as
\begin{equation}
\partial_s \tilde f_s(\xx) = \sum_{i=1}^N \sum_{j=1}^N c_{ij}(s) \left( -\sum_{a=1}^n \one{x_a=i}\tilde f_s(\xx) + \sum_{b \neq c \in [n]} \left( \one{x_b=x_c=i} \tilde f_s(m_{bc}^{ij}\xx) - \one{x_b=i} \one{x_c=j} \tilde f_s(s_{bc}^{ij}\xx) \right) \right)
\end{equation}
where $c_{ij}(s)=N^{-1}|\lambda_i(s)-\lambda_j(s)|^2$ when $i \neq j \in [N]$ and vanishes otherwise.
By symmetrizing over $b$ and $c$ and desymmetrizing over over $i$ and $j$, we can rewrite the sum as
\begin{equation}\label{eq:derivnormal}
\partial_s \tilde f_s(x) = \sum_{i=1}^{N-1} \sum_{j=i+1}^{N} c_{ij}(s) \Bigg( \begin{aligned}[t]& \sum_{b \neq c \in [n]} \left[ \one{x_b=x_c=i} \tilde f_s (m_{bc}^{ij}\xx) 
+ \one{x_b=x_c=j} \tilde f_s(m_{bc}^{ji}\xx) 
- 2\one{x_b=i}\one{x_c=j} \tilde f_s(s_{bc}^{ij}\xx) \right] \\
&- \tilde f_s(\xx) \sum_{a = 1}^n[\one{x_a=i}+\one{x_a=j}] \Bigg)\end{aligned}
\end{equation}
Divide equation \eqref{eq:derivnormal} by $\sqrt{\pi(\xx)}$ and use the relations
\begin{equation}
\sqrt{\frac{\pi(m_{bc}^{ij}\xx)}{\pi(\xx)}} = \frac{n_j(\xx)+1}{n_i(\xx)-1}
\quad \mbox{and} \quad
\sqrt{\frac{\pi(s_{bc}^{ij}\xx)}{\pi(\xx)}} = 1
\end{equation}
when $x_b=x_c=i$ in the first identity and $x_b=i$ and $x_c=j$ in the second identity.  Note that the right hand side of these relations are independent of colors and depend only on the total particle numbers at sites $i$ and $j$.  In the end, we have the time derivative of the original observable
\begin{equation}
\partial_s f_s(\xx) = \sum_{i=1}^{N-1} \sum_{j=i+1}^{N} c_{ij}(s) \Bigg( \begin{aligned}[t]
& \sum_{b \neq c \in [n]} \bigg[ \begin{aligned}[t] &\one{x_b=x_c=i} \frac{n_j(\xx)+1}{n_i(\xx)-1} f_s(m_{bc}^{ij}\xx) 
+ \one{x_b=x_c=j} \frac{n_i(\xx)+1}{n_j(\xx)-1} f_s(m_{bc}^{ji}\xx) \\
- & 2\one{x_b=i}\one{x_c=j} \tilde f_s(s_{bc}^{ij}\xx)\bigg] \end{aligned}\\
&- \tilde f_s(\xx) \sum_{a = 1}^n[\one{x_a=i}+\one{x_a=j}] \Bigg)\end{aligned}
\end{equation}
Lastly, counting coefficients, we see that the coefficient in front of $f(x)$ is the negation of the sum over coefficients of the remaining terms.  Indeed for each $i<j\in[N]$,
\begin{equation}
\sum_{a=1}^n \left( \one{x_a=i} + \one{x_a=j} \right) = n_i(\xx) + n_j(\xx)
\end{equation}
while
\begin{multline}
\sum_{b \neq c \in [n]} \bigg[ \one{x_b=x_c=i} \frac{n_j(\xx)+1}{n_i(\xx)-1}
+ \one{x_b=x_c=j} \frac{n_i(\xx)+1}{n_j(\xx)-1} 
- 2\one{x_b=i}\one{x_c=j} \bigg] \\ = n_i(\xx)(n_i(\xx)-1) \frac{n_j(\xx)+1}{n_i(\xx)-1} + n_j(\xx)(n_j(\xx)-1) \frac{n_i(\xx)+1}{n_j(\xx)-1} - 2n_i(\xx)n_j(\xx) = n_i(\xx) + n_j(\xx)
\end{multline}
This observation allows us to put the the derivative in parabolic form
\begin{equation}
\partial_t f_s(\xx) = \sum_{i=1}^{N-1} \sum_{j=i+1}^{N} c_{ij}(s) \Bigg( \sum_{b \neq c \in [n]} \bigg[ \begin{aligned}[t] & \frac{n_j(\xx)+1}{n_i(\xx)-1} (f_s(m_{bc}^{ij}\xx) - f_s(\xx))
+ \frac{n_i(\xx)+1}{n_j(\xx)-1} (f(m_{bc}^{ji}\xx) - f_s(\xx)) \\
- & 2(f_s(s_{bc}^{ij}\xx) - f_s(\xx)) \bigg]\Bigg) \end{aligned}
\end{equation}
Terms in the first and second lines of each summand contribute to the move operator, while terms in the third line of each summand contribute to the exchange operator.
\end{proof}

\subsection{Reversible measure}\label{app:rev_meas}

\begin{remark}
The reversible measure for $\gen(s)$ on $\Ln$ refines the reversible measure for the jump process studied in \cite{QUE} in that the reversible measure in that paper is precisely the pushforward of $\pi$ along the colorblind map $\forget: \Ln \rightarrow \Omega_{n/2}$, $\forget(\xx)_i = n_i(\xx)/2$.  Each $\bet \in \Omega_{n/2}$ has $\binom{2n}{2\bet}$ preimages, each weighted equally.  Therefore,
\begin{equation}
\pi(\forget^{-1}\bet) = (2n)! \prod_{i=1}^N \frac{(2\eta_i)!!^2}{(2\eta_i)!}
\end{equation}
which coincides with the reversible measure from \cite{QUE} up to the constant factor of $(2n)!$.
\end{remark}

\begin{proposition}\label{prop:Mrevmeasn}
For all $\xx, \yy \in \Ln$ and all $i < j \in [N]$,
\begin{equation}
\ipn{ \delta_\xx, \move_{ij} \delta_\yy } = \ipn{ \move_{ij} \delta_\xx, \delta_\yy }
\end{equation}
where we use the inner product and convention for $\delta_\xx$ given in \ref{def:inner_product}.
\end{proposition}

\begin{proof}
There are three disjoint possibilities for any pair $(\xx,\yy)$. Either
\begin{enumerate}
\item $\xx=\yy$
\item There is no pair $a,b \in [n]$ such that $\yy \in \{m_{ab}^{ij}\xx, m_{ab}^{ji}\xx,\xx\}$.
\item There is exactly one subset of size two $\{a,b\} \subset [n]$ such that $\yy \in \{m_{ab}^{ij}\xx, m_{ab}^{ji}\xx\} \backslash \{\xx\}$.
\end{enumerate}
In case 1, the equation is vacuously true as the left and right hand sides are identical.  In case 2, both sides are 0.  This can be seen because all delta functions in
\begin{equation}
\move_{ij} \delta_\xx(\yy) = \sum_{a \neq b} \frac{n_j(\yy)+1}{n_i(\yy)-1}(\delta_\xx(m_{ab}^{ij}\yy) - \delta_\xx(\yy)) + \sum_{a\neq b} \frac{n_i(\yy)+1}{n_j(\yy)-1}(\delta_\xx(m_{ab}^{ji} \yy) - \delta_\xx(\yy))
\end{equation}
and
\begin{equation}
\move_{ij} \delta_\yy(\xx) = \sum_{a\neq b} \frac{n_j(\xx)+1}{n_i(\xx)-1}(\delta_\yy(m_{ab}^{ij}\xx) - \delta_\yy(\xx)) + \sum_{a\neq b} \frac{x_i+1}{x_j-1}(\delta_\yy(m_{ab}^{ji}\xx) - \delta_\yy(\xx))
\end{equation}
evaluate to 0.  Lastly in case 3, the above expressions reduce to
\begin{equation}
\move_{ij}\delta_\xx(\yy) = \frac{1}{\sqrt{\pi(\xx)\pi(\yy)}} = \move_{ij}\delta_\yy(\xx)
\end{equation}
because $\yy=m_{ab}^{ij}\xx$ if and only if $\xx=m_{ab}^{ji}\yy$.
\end{proof}

\begin{proposition}\label{prop:Erevmeasn}
For all $\xx,\yy \in \Ln$ and $i < j \in [N]$,
\begin{equation}
\ipn{ \delta_\xx, \exch_{ij} \delta_\yy } = \ipn{ \exch_{ij} \delta_\xx, \delta_\yy }
\end{equation}
\end{proposition}

\begin{proof}
Again, consider three disjoint cases for the pair $(\xx,\yy)$.  Either
\begin{enumerate}
\item $\xx=\yy$
\item $\yy \not\in \{s_{ab}^{ji} \xx, \xx\}$ for any $a,b \in [n]$.
\item There is exactly one pair $a \neq b \in [n]$ such that $\yy = s_{ab}^{ij}\xx$. 
\end{enumerate}
In the first case, the equation is trivial.  In the second case, both sides are zero since all delta functions in the expressions
\begin{equation}
\exch_{ij} \delta_\xx(\yy) = 2 \sum_{a\neq b} (\delta_\xx(s_{ab}^{ij}\yy) - \delta_\xx(\yy))
\quad \mbox{and} \quad
\exch_{ij} \delta_\yy(\xx) = 2 \sum_{a \neq b} (\delta_\yy(s_{ab}^{ij}\xx) - \delta_\yy(\xx))
\end{equation}
evaluate to 0.  Now suppose we are not in the first two cases.  Then we must be in the third case with $\yy = s_{ab}^{ij} \xx$ and $\xx = s_{ab}^{ij} \yy$. The expressions then evaluate to
\begin{equation}
\ipn{ \delta_\xx, \exch_{ij} \delta_\yy } = \frac{2}{\pi(\yy)} = \frac{2}{\pi(\xx)} = \ipn{\exch_{ij} \delta_\xx, \delta_\yy}
\end{equation}
Since all particle numbers are the same $n_k(\xx) = n_k(\yy)$ for all $k \in \ZZ$ and $\pi$ depends only on particle numbers.
\end{proof}

\begin{remark}
The above two propositions show that $\move_{ij}$ and $\exch_{ij}$, and hence $\gen_{ij} = \move_{ij} - \exch_{ij}$, are reversible with respect to the measure $\pi(\xx)$.
\end{remark}

\bibliography{color}
\bibliographystyle{abbrv}

\end{document}